\documentclass[11pt,reqno]{amsart}
\usepackage{amsmath, amsthm, amsopn, amssymb, microtype, bm}
\usepackage{mathrsfs} 
\usepackage{mathscinet}
\usepackage{bbm}
\usepackage{enumerate,tikz, etoolbox, intcalc, geometry, caption, subcaption, afterpage}
\usepackage[english]{babel}
\usepackage{enumitem}
\usepackage{graphs} 
\usepackage{MnSymbol}

\usepackage[colorlinks=true,urlcolor=blue,pdfborder={0 0 0}]{hyperref}
\hypersetup{linkcolor=[rgb]{0,0,0.6}}
\hypersetup{citecolor=[rgb]{0,0.6,0}}
\usepackage[nameinlink]{cleveref}

\theoremstyle{definition}
\newtheorem{definition}{Definition}[section]
\theoremstyle{remark}
\newtheorem{example}[definition]{Example}

\newtheorem{remark}[definition]{Remark}

\theoremstyle{plain}
\newtheorem{thm}[definition]{Theorem}
\newtheorem{theorem}[definition]{Theorem}
\newtheorem*{theorem*}{Theorem}
\newtheorem{prop}[definition]{Proposition}
\newtheorem{lemma}[definition]{Lemma}

\newtheorem{corol}[definition]{Corollary}

\def\ol{\overline}
\def\ov{\overline}
\def\tl{\widetilde}
\def\wh{\widehat}

\newcommand{\N}{\mathbb N}
\newcommand{\Z}{\mathbb Z}

\newcommand{\B}{\mathcal B}

\newcommand{\be}{\begin{equation}}
\newcommand{\ee}{\end{equation}}
\newcommand{\ba}{\begin{aligned}}
\newcommand{\ea}{\end{aligned}}
\newcommand{\mc}{\mathcal}
\newcommand{\ignore}[1]{}

\newcommand{\Om}{\Omega}

\numberwithin{equation}{section}

\begin{document}

\title{Subdiagrams and invariant measures for generalized Bratteli diagrams}

\author[Bezuglyi]{Sergey Bezuglyi}
\address{Department of Mathematics, University of Iowa, Iowa City, IA 52242-1419, USA}
\email{sergii-bezuglyi@uiowa.edu}
\author[Jorgensen]{Palle Jorgensen}
\address{Department of Mathematics, University of Iowa, Iowa City, IA 52242-1419, USA}
\email{palle-jorgensen@uiowa.edu}
\author[Karpel]{Olena Karpel}
\address{AGH University of Krakow, Faculty of Applied Mathematics, al. Adama Mickiewicza~30, 30-059 Krak\'ow, Poland, and
B.~Verkin Institute for Low Temperature Physics and Engineering, 47~Nauky Ave., Kharkiv, 61103, Ukraine}

\email{okarpel@agh.edu.pl}

\author[Raszeja]{Thiago Raszeja}
\address{Pontifícia Universidade Católica do Rio de Janeiro, Departamento de Matemática, Rua Marquês de São Vicente, 225, Gávea - Rio de Janeiro, RJ - Brazil}
\email{tcraszeja@gmail.com}

\author[Sanadhya]{Shrey Sanadhya}
\address{School of Mathematics and Statistics F07, University of Sydney, NSW 2006, Australia, and Einstein Institute of Mathematics, 
The Hebrew University of Jerusalem, 
Edmond J. Safra Campus,
Jerusalem, Israel.}
\email{shrey.sanadhya@sydney.edu.au}



\date{}

\begin{abstract} The results of this paper contribute to the study of invariant measures of Borel dynamical systems that can be modeled using generalized Bratteli diagrams. In this context, we study tail invariant measures on the path spaces of generalized Bratteli diagrams, allowing countably infinite vertex sets at each level. Our main focus is on subdiagrams of generalized Bratteli diagrams and the problem of extending tail invariant probability measures from vertex and edge subdiagrams to the ambient diagram. We establish necessary and sufficient conditions for the finiteness of such extensions, formulated in terms of incidence matrices and associated stochastic matrices. Several classes of generalized Bratteli diagrams and their subdiagrams are analyzed in detail, including simple, stationary, and bounded size diagrams. We develop constructive, step-by-step procedures for measure extension and for approximating invariant measures by measures supported on suitable subdiagrams. In addition, we provide explicit examples of generalized Bratteli diagrams that admit no probability tail invariant measures, a phenomenon absent for standard Bratteli diagrams with finite vertex sets. Finally, we address convergence questions for sequences of invariant measures arising from approximations by subdiagrams, clarifying the relationship between combinatorial structure and measure-theoretic behavior.
\end{abstract}

\maketitle
\tableofcontents

\section{Introduction}

The study of ergodic invariant measures is a central problem in the theory of dynamical systems. Its importance is underscored by the fact that the collection of ergodic invariant probability measures serves as an invariant for several equivalence relations arising in ergodic theory as well as in Borel and Cantor dynamics. In particular, it was proved in \cite{Dougherty_Jackson_Kechris1994} that the number of ergodic invariant probability measures is a complete invariant for Borel isomorphism (orbit equivalence) of aperiodic nonsmooth hyperfinite equivalence relations.
Since every such equivalence relation can be realized as a tail equivalence relation on the path space of a generalized Bratteli diagram, methods from the theory of Bratteli diagrams provide a natural framework for the explicit description of these measures. In the present paper, we focus on the construction of tail invariant measures generated by subdiagrams, which constitutes the main theme of our work.

\subsection{Why do we need Bratteli diagrams?}
Bratteli diagrams and the dynamical systems arising on their associated path spaces provide a flexible and powerful framework for modeling a broad class of equivalence relations and dynamical phenomena. Originally introduced in the study of AF $C^*$-algebras, Bratteli diagrams have since become an important tool in the construction of models in measurable, Cantor, and Borel dynamics, where they encode orbit structures, invariant measures, and asymptotic properties of dynamical systems in a combinatorial manner. 

Over the last several decades, Bratteli diagrams have
been intensively studied and used in various areas of dynamics. 
We refer below to other papers where the role of Bratteli diagrams has been discussed in detail. Here we outline only several directions of research and give the corresponding references (the list of cited papers is far from complete). 

Discrete structures such as graphs and sequences of finite partitions have proved to be very useful in the theory of dynamical systems; we mention some old papers by Krieger \cite{Krieger1976}, Vershik \cite{ Vershik1973, Vershik_1981, Vershik_1982}, and a more recent paper \cite{GambaudoMartens2006}. 

In the 1990s, these ideas found new applications in Cantor dynamics. Putnam \cite{Putnam1989} 
showed that, for every minimal homeomorphism $\varphi$ 
of a Cantor set $X$, there exists a  sequence of refining partitions into clopen sets 
that approximates the orbits of $\varphi$ and the topology on
$X$. Based on this result, Herman-Putnam-Skau \cite{HermanPutnamSkau1992} proved that every minimal homeomorphism of a Cantor set can be realized as a
homeomorphism $\varphi_B$ (called a \textit{Vershik map}) 
of a path space $X_B$ of a Bratteli diagram $B$. This crucial result opened the study of (i) \textit{orbit equivalence} of minimal Cantor systems, see \cite{GiordanoPutnamSkau1995, GiordanoPutnamSkau1999}, Glasner-Weiss 
\cite{GlasnerWeiss_1995}, \cite{GiordanoMatuiPutnamSkau2010, GiordanoMatuiPutnamSkau2008} and (ii) constructing \textit{models} of transformations in Cantor and Borel dynamics, see \cite{Forrest1997}, \cite{Durand_Host_skau_1999}, \cite{BezuglyiDooleyKwiatkowski_2006}, 
 \cite{Medynets_2006}, \cite{DownarowiczKarpel_2019},
\cite{Shimomura2020}. 

Another direction of research concerns the study of various classes of Bratteli diagrams. Among them, we distinguish the following classes: (a) substitution dynamical systems are represented by \textit{stationary Bratteli diagrams} (simple and non-simple, standard and generalized); they were considered in \cite{Forrest1997}, \cite{Durand_Host_skau_1999}, \cite{BezuglyiKwiatkowskiMedynetsSolomyak2010}, \cite{Bezuglyi_Jorgensen_Sanadhya_2023}, \cite{BezuglyiJorgensenKarpelSanadhya2025}, (b) \textit{finite rank}  Bratteli diagrams, which, in particular, represent interval exchange transformations, were the subject of study in  \cite{Downarowicz_Maass_2008}, \cite{BressaudDurandMaass2010}, \cite{DonosoDurandMaassPetite_2021}, \cite{BezuglyiKwiatkowskiYassawi2014}, (c) finite and $\sigma$-finite \textit{tail invariant measures} on standard and generalized Bratteli diagrams were studied in \cite{AdamskaBezuglyiKarpelKwiatkowski2017},  
 \cite{BezuglyiKarpelKwiatkowski2019}, \cite{BezuglyiKarpel_2020},
 \cite{BezuglyiJorgensenKarpelSanadhya2025}, (d) \textit{eigenvalues} of Cantor minimal systems were considered  in a series of papers, see e.g. 
\cite{DurandFrankMaass_2019} for references, (e) other applications of Bratteli diagrams can be found in  \cite{AminiElliottGolestani2021},
\cite{DurandPerrin2022}, 
\cite{GiordanoGoncalvesStarling2017},
\cite{GiordanoMatuiPutnamSkau2010}, \cite{DownarowiczMaass2008}, 
\cite{GjerdeJohansen2000}, 
\cite{Putnam2018},  \cite{Trevinio2018}.  
The reader who is interested in the classification of stationary Bratteli diagrams, various links to operator algebras, and $K$-theory can find more information in 
\cite{Bratteli1972}, \cite{Effros1981}, 
\cite{EffrosHandelmanShen1980}, 
\cite{BratteliJorgensenKimRoush2000},
\cite{BratteliJorgensenKimRoush2001},
\cite{BratteliJorgensenKimRoush2002}.

\subsection{What do we study in the paper?}
In this paper, we study generalized Bratteli diagrams with countably infinite vertex sets at each level and focus on questions concerning tail invariant measures, including their existence, structure, and behavior under restriction to subdiagrams. This work continues our recent investigations 
(see the author’s previous work cited in the references) on various aspects of generalized Bratteli diagrams. The main emphasis here is on \textit{subdiagrams} and the corresponding problems of extension and finiteness of tail invariant measures.
We also use sequences of standard subdiagrams to approximate the dynamical properties of dynamical systems and tail invariant measures. In some sense, we can view this problem as an approximation of Borel dynamical systems on a zero-dimensional Polish space by Cantor dynamical systems.

A related family of problems was studied for standard Bratteli diagrams in \cite{AdamskaBezuglyiKarpelKwiatkowski2017}. The generalized setting is substantially more challenging, due to its richer combinatorial structure and the fact that the associated path space is, in general, a zero-dimensional Polish space that is not locally compact.

\subsection{What is a Bratteli diagram?}
We briefly discuss the main notions used throughout the paper. A \textit{Bratteli diagram} is a countable graded graph
$B = (V,E)$ where the sets of vertices $V$ and edges $E$ are decomposed into disjoint unions 
 $V = \bigcup_{n \in \N_0} V_n$, $E = \bigcup_{n \in \N_0} E_n$.
 For a standard Bratteli diagram, every level $V_n$ is finite
 while for a generalized Bratteli diagram, the sets $V_n$
are allowed to be countably infinite. The set $E_n$
consists of edges connecting the vertices of level $V_n$
to the vertices of level $V_{n+1}$.
We assume that every vertex has finitely many incoming edges.
The set $E_n$ determines the 
\textit{incidence matrix} 
$F_n = (f_{v,w}^{(n)})$ where $f_{v,w}^{(n)}$ is the number of 
edges connecting $v \in V_{n+1}$ and $w\in V_n$. 

Associated with a diagram $B$ is its path space $X_B$ consisting of all infinite paths obtained by concatenating edges across successive levels.
Two infinite paths are called \textit{tail equivalent} if they coincide from some level onward. This relation defines the \textit{tail equivalence relation} $\mathcal R$, which serves as a prototypical example of a hyperfinite equivalence relation arising in Cantor and Borel dynamics. 

Given a (generalized) Bratteli diagram $B$, we can consider a subgraph $\ol B$ of $B$ such that $\ol B$ is again a (generalized) Bratteli diagram; we call $\ol B$ a subdiagram of $B$. Let $\ol{\mathcal R}$ be the tail equivalence relation on $X_{\ol B}$ and $\nu$ a probability $\ol{\mathcal R}$-invariant measure. The set $X_{\ol B}$ is a subset of $X_B$ which, in general, is not tail invariant with respect to $\mathcal R$. Then the measure $\nu$ can be extended to the set $\mathcal R(X_{\ol B})$ to preserve tail invariance, and we obtain in such a way an $\mathcal R$-invariant measure on $X_B$. This procedure is called the \textit{measure extension} and is described in detail in Section  \ref{sect meas ext subdgrms}.
 
The problem of existence and classification of $\mathcal R$-invariant measures is one of the central questions in the theory of Bratteli diagrams, motivated by their role as models for hyperfinite equivalence relations and dynamical systems.
 Further definitions and background on generalized Bratteli diagrams are collected in \Cref{Sec: Basics}. 
  
\subsection{Motivation}
A key motivation for our work comes from Borel and Cantor dynamics. As mentioned above, every homeomorphism of a Cantor set is realized as a transformation (Vershik map) on the path space of a standard Bratteli diagram \cite{HermanPutnamSkau1992}, \cite{Medynets_2006}, \cite{DownarowiczKarpel_2019}, \cite{Shimomura2020}.  
In Borel dynamics, the following result was proved in \cite{BezuglyiDooleyKwiatkowski_2006}.

\textit{Let $T$ be an aperiodic Borel automorphism acting on a standard Borel space $(X, \B)$. Then there exists a 
generalized ordered Bratteli diagram $B=(V,E,\geq)$ and a Vershik automorphism $\varphi : X_B \to X_B$ such that $(X, T)$ is Borel isomorphic to $(X_B,\varphi)$.}

Equivalently, the orbit equivalence relation of such a system coincides with the tail equivalence relation on the diagram. As a consequence, every aperiodic hyperfinite countable Borel equivalence relation (CBER) admits a realization as a tail equivalence relation on a Bratteli diagram, standard or generalized.  

Hence, many structural properties of a Cantor or Borel dynamical system 
$(X, T)$ can be encoded by a corresponding Bratteli diagram. In particular, the set of ergodic $T$-invariant measures is exactly the set of tail invariant measures on the Bratteli diagram. We refer to recent books \cite{Putnam2018}, \cite{DurandPerrin2022}, and surveys \cite{Durand2010}, \cite{BezuglyiKarpel2016} where this correspondence has been discussed in detail.

One of the central problems in dynamical systems is the description of all ergodic invariant measures.
The importance of this problem is underscored by the following result proved in \cite{Dougherty_Jackson_Kechris1994}: \textit{the number of ergodic probability invariant measures is a complete invariant for Borel isomorphism (equivalently, orbit equivalence) of aperiodic nonsmooth hyperfinite countable Borel equivalence relations (CBER).} 
This shows that finding the set of ergodic probability tail invariant measures for a generalized Bratteli diagram is equivalent to the classification of tail equivalence relations up to isomorphism, because such equivalence relations are hyperfinite CBER.

While invariant measures for standard Bratteli diagrams are by now well understood, the generalized setting exhibits new phenomena that do not appear in the case of standard diagrams. In particular, generalized Bratteli diagrams may fail to support any probability tail invariant measures, may support infinitely many such measures with complicated structure, or may admit measures whose behavior under restriction to subdiagrams is highly nontrivial. These issues are closely related to recurrence and transience properties of associated infinite incidence matrices and to subtle asymptotic features of the diagram.

\subsection{Three methods for constructing tail invariant measures}
There are three general methods for constructing tail invariant measures for a given Bratteli diagram.
 If the diagram is stationary, that is $F_n = F$ for all $n$, 
 then Perron–Frobenius theory for finite and infinite matrices 
 can be used to describe such measures, see \cite{BezuglyiKwiatkowskiMedynetsSolomyak2010}, \cite{BezuglyiJorgensenKarpelSanadhya2025}. The second approach, called the inverse limit method, was developed in the paper 
\cite{BezuglyiKarpelKwiatkowskiWata2024}, where we showed how ergodic tail invariant measures can be obtained as limits of infinite-dimensional vectors.
 The third method is based on the study of subdiagrams and measures supported on them. This method is the core of the current paper. 

The primary focus of this paper is the systematic investigation of subdiagrams and measure extension problems for generalized Bratteli diagrams. Given a generalized Bratteli diagram 
$B$, one may form subdiagrams either by restricting to proper subsets of vertices (\textit{vertex subdiagrams}) or by removing edges while keeping all vertices (\textit{edge subdiagrams}). Each subdiagram has its own path space and tail equivalence relation, and an invariant probability measure on the subdiagram can be canonically extended to a (possibly infinite) invariant measure on the saturation of the subdiagram inside the original path space. A central question is to determine when such extensions are finite, and how this depends on the combinatorial structure of the diagram. 

\subsection{The main results and outline of the paper} \Cref{Sec: Basics} presents a collection of essential definitions and results used throughout the paper. We recall the basic definitions and results concerning generalized Bratteli diagrams, path spaces, tail equivalence relations, and invariant measures, including the connection with infinite Perron–Frobenius theory. For additional background, the reader may consult the introductory sections of other papers such as \cite{BezuglyiKarpel2016}, \cite{BezuglyiKarpel_2020}, 
\cite{BezuglyiKarpelKwiatkowskiWata2024}, \cite{BezuglyiJorgensenKarpelSanadhya2025}.

Our first group of results, proved in \Cref{sect meas ext subdgrms}, establishes necessary and sufficient conditions for the finiteness of measure extensions from vertex and edge subdiagrams (see \Cref{thm3 sufficient cond}, \ref{edge_fin_krit}, and \ref{Thm sect 3 adding edges}). This partially resolves \cite[Problem 3, Section 9]{BezuglyiJorgensenKarpelSanadhya2025}. These criteria are formulated in terms of the incidence matrices of the diagram and, more importantly, in terms of an associated sequence of row-stochastic matrices naturally derived from the diagram. This approach allows us to give quantitative conditions that apply uniformly to all tail invariant probability measures, rather than depending on a specific choice of measure. In particular, we obtain sufficient conditions guaranteeing the finiteness of extensions for every probability tail invariant measure on a given subdiagram. 

A substantial part of the paper is devoted to constructive methods. In \Cref{sect ext from classes sbdgrms}, we consider the measure extension procedure applied to some classes of generalized Bratteli diagrams. Assuming that a subdiagram is simple and therefore standard, we found conditions that guarantee the finiteness of the measure extension (see \Cref{thm meas ext finite sbdgrm} and \Cref{thm ness cond finite sbdgrm}). Next, we considered a similar problem for a stationary subdiagram. The results are formulated in \Cref{prop ext from stat BD} and \Cref{prop s4 stat diagr ext}. One more class of generalized Bratteli diagrams studied in this section consists of bounded size diagrams containing a ``fat odometer''. This means that there exists an odometer with a dominating number of edges in the diagram. It turns out that depending on the parameters of the bounded size diagram and the odometer, we can obtain finite or infinite values for the extended measure, see \Cref{thm: ext fat odom}. 

In \Cref{sect step-by-step} and \Cref{sect Approximation}, we present step-by-step procedures for extending measures from subdiagrams and for approximating invariant measures by measures supported on simpler subdiagrams, including stationary and vertex subdiagrams. These constructions are particularly useful in situations where direct analysis of the full diagram is difficult, but where controlled approximations are available. The main results of \Cref{sect step-by-step} are proved in \Cref{thm s5} and \Cref{thm s5-2}. In \Cref{sect Approximation}, we focus on the following problem. Let $B = B(F)$ be a generalized stationary Bratteli diagram, $(F_n)$ a sequence of finite square matrices approximating $F$, and $B_n = B(F_n)$ the corresponding sequence of stationary standard subdiagrams. Let $\mu$ and $\mu_n$ be the tail invariant measures defined by Perron-Frobenius eigenpairs for $F$ and $F_n$. How is the approximation of an infinite matrix by a sequence of finite matrices translated in terms of tail invariant measures defined on the corresponding stationary diagrams? The answer is given in \Cref{prop appr truncated matrices}, where we prove that $\mu_n$ is weakly convergent on cylinder sets to $\mu$. 
As an application of this result, we study generalized Bratteli diagrams arising from infinite Leslie matrices and related models. We give a constructive approach to convergence questions, clarifying how invariant measures behave under approximation by subdiagrams and how limits can be described in terms of the underlying combinatorial data. We note that the convergence of measures using sequences of finite subgraphs or sublattices is a crucial technique in fields such as thermodynamic formalism and statistical mechanics.

Another important aspect of our work is the identification of generalized Bratteli diagrams that admit no probability tail invariant measures, see \Cref{sect B(2^k)}. This phenomenon has no analogue in the setting of standard Bratteli diagrams and highlights the qualitative differences between standard and generalized diagrams. We provide explicit constructions and criteria illustrating how the combinatorial growth of incidence matrices can prevent the existence of invariant probability measures altogether. This partially resolves \cite[Problem 4, Section 9]{BezuglyiJorgensenKarpelSanadhya2025} The main results are proved in \Cref{thm:varphi_i_i-t0}, \Cref{Thm:wand_1} and \Cref{s8 cor1}, \Cref{s8 cor2}. 

\Cref{ssect path space} is devoted to the study of the path space of a subdiagram of a Bratteli diagram $B$. We give conditions under which a closed subset $P$ of $X_B$ is generated by a subdiagram, see \Cref{thm s8 path space}. We also discuss the 0-1 procedure and the duality between edge and vertex subdiagrams. The question of finding properties that are invariant under the 0-1 procedure is discussed in \Cref{prop:0-1properties}. 

In the final section (\Cref{Sec: conv of measures}), we again address the question of convergence of measures on cylinder sets applied to the path space of a generalized Bratteli diagram. 
\medskip

\section{Basics on Bratteli diagrams and tail invariant
measures}\label{Sec: Basics}

This section recalls the main definitions and results
concerning generalized Bratteli diagrams. This class of Bratteli diagrams was defined in \cite{BezuglyiDooleyKwiatkowski_2006} and then studied in recent publications
 \cite{BezuglyiJorgensen2022}, 
\cite{BezuglyiJorgensenKarpelSanadhya2025}, \cite{Bezuglyi_Jorgensen_Sanadhya_2023}, \cite{BezuglyiKarpelKwiatkowski2024}, see also \cite{BezuglyiKarpelKwiatkowskiWata2024}, \cite{BezuglyiJorgensenKarpelKwiatkowski2025}, \cite{BezuglyiDudkoKarpel2024}. 

In this paper, we use the standard notation $\N$, $\Z$, and $\N_0 = \N \cup \{0\}$ for natural numbers, integers, and nonnegative integers, respectively.  

\begin{definition}\label{Def:generalized_BD} A 
\textit{generalized Bratteli diagram} is a graded graph 
$B = (V, E)$ such that the vertex set $V$ and the edge set $E$ are represented as partitions into levels, $V = \bigsqcup_{i=0}^\infty  V_i$ and $E = \bigsqcup_{i=0}^\infty  E_i$, such that the following properties hold:
\begin{enumerate}[label=(\roman*)]
  \item The number of vertices at each level $V_i$, $i \in \N_0$, is countably infinite (if necessary, we will identify each $V_i$ with $\Z$ or $\N$). For $i \in \N_0$, the set $V_i$ is called the $i$-th level of the diagram $B$. We assume that for all $i \in \N_0$, 
the set $E_i$ of edges between $V_i$ and $V_{i+1}$ is countable.
  \item For every edge $e\in E$, we define the \textit{range} and \textit{source} maps $r$ and $s$ such that $r(E_i) = V_{i+1}$ and $s(E_i) = V_{i}$ for $i \in \N_0$. It is required that $s^{-1}(v)\neq \emptyset$ for all $v\in V$ and $r^{-1}(v)\neq\emptyset$ for all $v \in V\setminus V_0$.
  
  \item For every vertex $v \in V \setminus V_0$, the set $r^{-1}(v)$ of edges with range $v$ is finite (we write $|r^{-1}(v)| < \infty$ where $|\cdot|$ denotes the cardinality of a set). 
\end{enumerate} 
If $B =(V, E)$ has the property that $|V_n| < \infty$ for each level $n \in N_0$, then $B$ is called a \textit{standard Bratteli diagram.}
\end{definition} 

When the vertices
at each level are indexed by $\Z$, the generalized Bratteli diagram $B$ is called a \textit{two-sided infinite} diagram, and when the vertices are indexed by
$\N$ (or $\N_0$), $B$ is called a \textit{one-sided infinite} diagram.

The structure of a generalized Bratteli diagram $B$ is
completely determined by a sequence of countably infinite nonnegative matrices. Let $n,m \in \N_0$ such that $m>n$. For a vertex $v \in V_m$ and a vertex $w \in V_{n}$, denote
by $E(v, w)$ the set of all finite paths between $v$ and
$w$ (this set may be empty). For $n \in \N_0$, let $f^{(n)}_{v,w} = |E(v, w)|$ for all $w \in V_n$ and
$v \in V_{n+1}$. Clearly, we can use the equality $E(v, w) = E(w, v)$ because this set represents all finite paths between the vertices $w$ and $v$. The notation $E(V_0, v)$ means that we consider all finite paths between vertices from $V_0$ and $v$. We define the infinite matrix $F_n$, $n \in \N_0$ by setting 
\begin{equation}\label{Notation:f^i}
    F_n = (f^{(n)}_{v,w} : v \in V_{n+1}, w\in V_n),\ \   
    f^{(n)}_{v,w}  \in \N_0.
\end{equation} 
The matrices $F_n$, $n \in \N_0$, are called \textit{incidence matrices}. Note that in the notation of the entries of $F_n$ in \eqref{Notation:f^i}, the order of the indices $v$ and $w$ is important.
The assumption $|r^{-1}(v)| <
\infty$ implies that every row of $F_n$, $n \in \N_0$ contains finitely many
nonzero entries. On the other hand, the columns of $F_n$ may contain infinitely many nonzero entries.  
We will use the notation $B = B(F_n)$ when we need to indicate what matrices determine a generalized Bratteli diagram. If  $F_n = F$ for every $n \in \N_0$, then the diagram $B$ is called
\textit{stationary}. We will write $B = B(F)$ in this
case.

\begin{remark}
Note that the notation $f^{(n)}_{v,w}$ will be used in two cases: for the $(v,w)$-entry of $F^n$ (if $B$ is a stationary diagram) and for the $(v,w)$-entry of the incidence matrix $F_n$. It will be clear from the 
context in which the case is considered.
\end{remark}

We will also consider generalized Bratteli diagrams whose incidence matrices have the properties of equal row sums (ERS) and equal column sums (ECS). More precisely, $B = B(F_n)$ has the $ERS(r_n)$ property for a sequence of positive integers $(r_n)_{n \in \N_0}$ if, for every $n \in \N_0$ and every $v \in V_{n+1}$,
$$
\sum_{w\in V_n} f^{(n)}_{v,w} = r_n. 
$$ 
If for every $n \in \N_0$ and every $w \in V_{n}$, we have 
$$
\sum_{v\in V_{n+1}} f^{(n)}_{v,w} = c_n,
$$
then we say that $B(F_n)$ has the $ECS(c_n)$ property for a sequence of positive integers $(c_n)_{n \in \N_0}$.

To define the path space of a generalized Bratteli diagram
$B$, we consider a finite or infinite sequence of edges
 $ x = (e_i: e_i\in E_i)$ (it is called a \textit{path}) such 
that $s(e_i)=r(e_{i-1})$. We denote the set of all infinite
paths $x$ starting at some vertex in $V_0$ by $X_B$ and call 
it the \textit{path space} of the diagram $B$. For a finite path
$\ol e = (e_0, ... , e_n)$, we write
$s(\ol e) = s(e_0)$ and $r(\ol e) = r(e_n)$. The set
$$
    [\ol e] := \{x = (x_i) \in X_B : x_0 = e_0, ..., x_n = e_n\}, 
$$ 
is called the \textit{cylinder set} associated with $\ol e$. Note that sometimes it is convenient to consider two sided cylinder set. Let $m > n$ and let $\ol e = (e_n, ... , e_{m})$ where $s(e_n) \in V_n$ and $r(e_{m}) \in V_{m+1}$. The set
$$
[\ol e] := \{x = (x_i) \in X_B : x_n = e_n, ..., x_{m} = e_m\}, 
$$ 
is called the \textit{two sided cylinder set} associated with $\ol e$.

The \textit{topology} on the path space $X_B$ is generated by
cylinder sets that are clopen in this topology.
This topology coincides with the topology defined by the 
following metric on $X_B$: for $x = (x_i), \, y = (y_i)$, set 
$$
\mathrm{dist}(x, y) = \frac{1}{2^N},\ \ \ N = \min\{i \in \N_0 : 
x_i \neq y_i\}.
$$
The path space $X_B$ is a zero-dimensional Polish space and, therefore, a standard Borel space. In general, $X_B$ is not locally compact. As usual, we will consider the Bratteli diagrams whose path spaces have no isolated points.

\begin{definition}\label{Def:Tail_equiv_relation}
Two paths $x= (x_i)$ and $y=(y_i)$ in $X_B$ are called 
\textit{tail equivalent} if there exists an $n \in \mathbb{N}_0$ 
such that $x_i = y_i$ for all $i \geq n$. This notion defines a \textit{countable Borel equivalence relation (CBER)} $\mathcal R$ in the path space $X_B$, which is called the \textit{tail equivalence relation}.
\end{definition}

In this paper, we are interested in  \textit{tail invariant measures}
on the path space $X_B$ of a generalized Bratteli diagram
$B$. The term \textit{measure} is always used for a
non-atomic positive Borel measure. We are mostly 
interested in \textit{full measures}, i.e., every cylinder 
set must be of positive measure. The set of 
probability tail invariant measures is denoted by 
$M_1(\mc R)$.

Every tail invariant measure can be characterized in terms of a sequence of 
positive vectors associated with the vertices of each level, see  \Cref{BKMS_measures=invlimits} below. 

\begin{definition}\label{def: tail inv meas} Let 
$B =(V, E)$ be 
a generalized Bratteli diagram and $\mathcal R$ the tail equivalence relation on the path space $X_B$. A measure $\mu$ on 
$X_B$ is called \textit{tail invariant} if, for any cylinder sets
$[\ol e]$ and $[\ol e']$ such that $r(\ol e) = r(\ol e')$, we have
$\mu([\ol e]) = \mu([\ol e'])$.
\end{definition}

We note that if a Borel measure $\mu$ on $X_B$ takes finite values on all cylinder sets, then $\mu$ is uniquely 
determined by its values on cylinder sets in $X_B$. 

For every generalized Bratteli diagram, there exists a naturally defined 
sequence of \textit{Kakutani-Rokhlin towers}.

\begin{definition} \label{Def:Kakutani-Rokhlin} Let 
$B =(V, E)$ 
be a generalized Bratteli diagram. For $w \in V_n, 
n \in \N_0$, denote 
$$
X_w^{(n)} = \{x = (x_i)\in X_B : s(x_{n}) = w\}.
$$
The collection of all such sets forms a partition
$\zeta_n$ of $X_B$ into  
\textit{Kakutani-Rokhlin towers}
corresponding to the vertices of $V_{n}$. 
Each finite path $\ov e = (e_0, \ldots, e_{n-1})$ with
$r(e_{n-1}) 
= w$, determines a ``floor'' of the tower $X_w^{(n)}$. It is identified with the set 
$$
X_w^{(n)}(\ov e) = \{x = (x_i)\in X_B : x_i = e_i,\; i = 
0,\ldots, n-1 \}.
$$
Then 
$$
X_w^{(n)} = \bigcup_{\ol e \in E(V_0, w)} X_w^{(n)}(\ov e).
$$ 
From this definition, it follows that the sequence of partitions $(\zeta_n)_{n \in \N_0}$ is refining. 
\end{definition}

\begin{definition}\label{Def:Height} Let $n \in \N$, for $v \in V_n$ and $v_0 \in
V_0$, we set $h^{(n)}_{v_0, v} = |E(v_0, v)| $ and define 
$$
H^{(n)}_v = \sum_{v_0 \in V_0} h^{(n)}_{v_0, v}, \ \ n \in \N.
$$ 
Set $H^{(0)}_v = 1$ for all $v\in V_0$. Thus in case of generalized Bratteli diagrams we set $H^{(0)} =\{ H_{v}^{(0)}: v \in V_0\}$ to be an infinite vector with all entries equal to $1$. Recall that in the papers on standard Bratteli diagrams (see \cite{BezuglyiKarpelKwiatkowski2015, AdamskaBezuglyiKarpelKwiatkowski2017}, for example), the authors used a root vertex $\{v_0\}$ and multiple edges connecting $\{v_0\}$ with the vertices of the first level. Hence, unlike the current case, for standard diagrams, the initial vector $H^{(0)}$ would be nontrivial. For $n \in \N$ we set $H^{(n)}_v = |E(V_0, v)|$. 
This gives us the vector $H^{(n)} = \langle H^{(n)}_{v} : 
v \in V_n \rangle$ associated with every level $n\in \N_0$. 
We call $H^{(n)}_v$ the \textit{height of the tower} $X_v^{(n)}$ 
corresponding to the vertex $v\in V_n$.
\end{definition}

From \Cref{Def:Height} it follows that $h_{w,v}^{(n)}$ is the 
$(v,w)$-entry in the product of matrices $F_{n-1} \ \cdots \ F_0$  and 
 \be\label{eq2: FH=H}
F_n H^{(n)} = H^{(n+1)}, \quad n \in \N_0.
 \ee
 
\begin{thm}\label{BKMS_measures=invlimits}
 Let $B = (V,E)$ be a Bratteli diagram (generalized or standard) 
 with the sequence of incidence matrices $(F_n)_{n  \in \N_0}$. 
\begin{enumerate}

\item Let $\mu \in  M_1(\mc R)$. For every 
$n\in \N_0$, define the vector $\ol p^{(n)} = \langle p^{(n)}_w : w \in V_n \rangle$ where 
\begin{equation}\label{eq:def_p_n}
    p^{(n)}_w= \mu(X_w^{(n)}(\ov e)),\ \ r(\ol e) = w, \ \ w\in V_n. 
\end{equation} 
Then the vectors   $\ol p^{(n)}$ satisfy the relation 
\begin{equation}\label{eq:formula_p_n}
F_n^{T} \ol p^{(n+1)} =\ol p^{(n)}, \quad n \in \N_0.
\end{equation}

\item Conversely, suppose that $\{\ol p^{(n)}= (p_w^{(n)}) 
\}_{n \in \N_0}$ is a sequence of nonnegative vectors such 
that $F_n^{T}\ol p^{(n+1)} =\ol p^{(n)}$ for all $n \in 
N_0$. Then there exists a uniquely determined tail 
invariant measure $\mu$ such that, for every finite path $\ol e$ with $r(\ol e) = w,$ we have $\mu(X_w^{(n)}(\ov e))= p_w^{(n)}$ for $w\in V_n, n \in \mathbb N_0$. The measure $\mu$ is in $M_1(\mc R)$ if the vector $\ol p^{(0)}$ is probability. 
\end{enumerate}
\end{thm}
For \textit{proof} of \Cref{BKMS_measures=invlimits} see 
\cite[Theorem 2.9.]{BezuglyiKwiatkowskiMedynetsSolomyak2010} (for 
standard Bratteli diagrams) and 
\cite[Theorem 2.3.2]{BezuglyiJorgensen2022} (for generalized Bratteli 
diagrams). 

It follows from \Cref{BKMS_measures=invlimits},
\begin{equation}\label{eq: Measure_Xv_H_P}
    p_v^{(n)} = \dfrac{\mu(X_v^{(n)})}{H_v^{(n)}}, \quad n \in \N_0.
\end{equation}
In our study of generalized \textit{stationary} Bratteli diagrams, the Perron-Frobenius theory for infinite matrices will play a crucial role, see \cite[Chapter 7]{Kitchens1998}, \cite{Seneta2006}, 
\cite{VereJones_1967}, \cite{VereJones_1968} for references. A concise treatment of Perron-Frobenius theory of infinite matrices is provided in \cite[Appendix A.]{BezuglyiJorgensenKarpelSanadhya2025}. Below, we recall the main result.

Let $A$ be an infinite, nonnegative, irreducible, and aperiodic matrix. For $i \in \Z$, let 
\[
\lambda = \lim\limits_{n \rightarrow \infty} \sqrt[n]{(A^n)_{i,i}} = \sup\limits_{n\in \mathbb{N}} \sqrt[n]{(A^n)_{i,i}} \leq \infty.
\] 
Then the value of $\lambda$ does not depend on $i$. We call $\lambda$ the Perron eigenvalue of $A$. We will assume that $\lambda < \infty$. The matrix $A$ is called \textit{recurrent} if for $i \in \Z$
$$
\sum\limits_{n = 0}^{\infty} \frac{(A^n)_{ii}}{\lambda^n} = \infty, 
$$
and \textit{transient} if
$$
\sum\limits_{n = 0}^{\infty} \frac{(A^n)_{ii}}{\lambda^n} < \infty. 
$$

\begin{thm}[Generalized Perron-Frobenius theorem] 
\label{thm Perron-Frobenius Thm}
Let $A$ be a real, nonnegative, irreducible, aperiodic, and recurrent infinite matrix with a finite Perron eigenvalue $\lambda$. Then 
\begin{enumerate}
\item  there exist strictly positive eigenvectors $\eta = (\eta_i)_{i\in \Z}$ and $\xi = (\xi_i)_{i \in \Z}$ such that $\eta A = \lambda \eta$, $A \xi = \lambda \xi$;

\item $\eta$ and $\xi$ are unique up to constant multiples;

\item $\langle\eta, \xi\rangle < \infty$ if and only if $A$ is 
positive recurrent;

\item if $A$ is positive recurrent and $\langle \eta, \xi\rangle = 1$, then
\[
\lim_{n\to\infty} \frac{(A^n)_{i,j}}{\lambda^n} = \xi_i \eta_j
\quad \text{for all } i,j \in \Z.
\]
\end{enumerate}
\end{thm} We will call $(\lambda, \xi)$ the \textit{Perron eigenpair} (or simply the \textit{eigenpair}) for $A$. For stationary generalized Bratteli diagrams, an application of \Cref{BKMS_measures=invlimits} and \Cref{thm Perron-Frobenius Thm} gives the following result (see
\cite{BezuglyiJorgensen2022, BezuglyiJorgensenKarpelSanadhya2025}).

\begin{thm} \label{thm PFThm} Let $B = B(F)$ be a stationary Bratteli diagram such 
that the incidence matrix $F$ is 
irreducible, aperiodic, and recurrent. Let $\xi = (\xi_v : v \in V_0)$ be a right eigenvector corresponding to the Perron eigenvalue
$\lambda$ for $A= F^{T}$, $A\xi = \lambda \xi$. 

\begin{enumerate}
    \item Then there exists a tail invariant measure $\mu$ on the path space $X_B$ such that for every finite path $\ol e$, $r(\ol e) = v, v \in V_n $, 
 \be\label{eq inv meas stat BD}
\mu^{(n)}_v : =  \mu([\ol e]) = \frac{\xi_v}{\lambda^{n}}.
\ee 
Here $\mu^{(n)}_v$ denotes the measure of a cylinder set given by any finite path ending at the vertex $v \in V_n$.

\item The measure $\mu$ is finite if and only if the right eigenvector
$\xi = (\xi_v)$ has the property $\sum_v \xi_v <\infty$.
\end{enumerate}
\end{thm}

\begin{remark}\label{Standard_stationary} A study of tail invariant measures for standard stationary diagrams was done in \cite{BezuglyiKwiatkowskiMedynetsSolomyak2010}, where an expression similar to \eqref{eq inv meas stat BD} was obtained. See \cite[Remark 3.9.]{BezuglyiKwiatkowskiMedynetsSolomyak2010}.
\end{remark}

For a generalized Bratteli diagram $B = B(F_n)$ with the incidence matrices $(F_n)_{n \in \N_0}$, we define a sequence of \textit{stochastic incidence matrices} $(Q_n)_{n \in \N_0}$. These matrices will play a key role in our quantitative analysis of generalized Bratteli diagrams. For $n \in \N_0$ we set
$Q_n = (q_{v,w}^{(n)} : v \in V_{n+1}, w \in V_n)$, where 
\be\label{eq_stoch matrix Q_n}
q_{v,w}^{(n)} =  {f}_{v,w}^{(n)}\ \cdot \frac{H_{w}^{(n)}}
{H_{v}^{(n + 1)}}.
\ee
Then we get from \eqref{eq2: FH=H} that 
\begin{equation}\label{e7}
\sum_{w \in V_{n}}q_{v,w}^{(n)} =1, \quad v \in\ 
V_{n + 1},  
\end{equation}
so that $Q_n$ is row stochastic for each $n \in \N_0$. 
\medskip 

In the paper, we will often consider generalized Bratteli diagrams with additional properties, which are defined as follows.

\begin{definition}[Bratteli diagrams of bounded size]
\label{Def:BD_bdd_size} A generalized Bratteli 
diagram $B(F_n)$ with $V_n = \mathbb{Z}$ for all $n \in \N_0$ is called of \textit{bounded size} if there exists 
a sequence of pairs of natural numbers $(t_n, L_n)_{n \in \N_0}$ 
such that, for all $n \in \mathbb{N}_0$ and all $v \in V_{n+1}$,
\begin{equation}\label{eq: Bndd size}
s(r^{-1}(v)) \in \{v - t_n, \ldots, v + t_n\} \quad \mbox{and} 
\quad \sum_{w \in V_{n}} f^{(n)}_{v,w} = \sum_{w \in V_{n}} |E(w,v)| 
\leq L_n.
\end{equation} 
If the sequence $(t_n, L_n)_{n \in \N_0}$ is constant, i.e., $t_n = t$ and $L_n = L$ for all $n \in \N_0$, then we say that the 
diagram $B(F_n)$ has \textit{uniformly bounded size}. 
\end{definition}
Such diagrams were studied in 
\cite[Section 3]{BezuglyiJorgensenKarpelSanadhya2025}. We will implicitly use the conventions: $(a)$ for all bounded size Bratteli diagrams, the vertices of every level are indexed by integers; $(b)$ for each $n \in \N_0$, the parameters $(t_n, L_n)$ are chosen as minimally possible.

\begin{definition}\label{Def:irreducible_GBD} 
A generalized Bratteli diagram $B =(V,E)$, where all levels $V_i$ are identified with a set $V_0$ (e.g. $V_0 = 
\mathbb{N}$ or $\mathbb{Z}$),  is called  \textit{irreducible} if 
for any vertices $i, j \in V_0$ and any level $V_n$ there exist
$m > n$ and a finite path connecting $i \in V_n$ and $j \in V_m$. In 
other words, the $(j, i)$-entry of the matrix $F_{m-1} \cdots 
F_n$ 
is nonzero. Otherwise, the diagram is called \textit{reducible}.

\end{definition}

In the class of generalized Bratteli diagrams, we can consider a new phenomenon of stationarity, namely, horizontally stationary Bratteli diagrams.

\begin{definition}\label{Def: Horz}
 Let $B = (V, E)$ be a generalized Bratteli diagram defined by the
 sequence of incidence matrices $(F_n)$. Suppose that the following properties hold: for every $n\in \N_0$, 
\begin{enumerate}[label=(\roman*)]
    \item the vertices of $V_n$ are identified with $\Z$;
    
    \item for every $i \in V_{n+1}$ and $j \in V_n$, the equality
$f_{i,j}^{(n)} = f_{i+1,j+1}^{(n)}$ holds.
\end{enumerate}
\noindent
We refer to such a Bratteli diagram as \textit{horizontally stationary}.
\end{definition} 

Such diagrams have been defined and studied in
\cite{BezuglyiJorgensenKarpelKwiatkowski2025}. The incidence matrices $F_n$ of such diagrams are banded Toeplitz matrices, i.e., the diagonals parallel to the main diagonal of $F_n$ consist of equal entries. 

Let $B = (V, E)$ be a horizontally stationary generalized Bratteli diagram. 
Denote by $\tau$ the shift on $\Z$: $\tau (i) = i+1$, $i\in \Z$. Then $\tau$ defines a transformation acting on the set of all edges $E$. For 
$e \in E$, 
$\tau(e)$ is the edge such that $s(\tau (e)) = \tau (s(e))$ and
$r(\tau (e)) = \tau (r(e))$. This operation is well defined because $B$ is horizontally stationary. 
Two edges, $e$ and $f$ from the set $E$, are called \textit{parallel}
if there exists $k \in \Z$ such that $s(f) = \tau^k(s(e))$ 
and $r(f) = \tau^k(r(e))$. We will also write $\tau^k(e) =f$. 
Therefore, the shift $\tau$ generates a transformation acting on the path space $X_B$: If $\ol x = (x_n)_{n \in \N_0}$, and $\tau(\ol x) = \ol y$ where
$\ol y = (y_n)_{n \in \N_0}$ then for each $n \in \N_0$, we have $y_n = \tau(x_n)$. 
Such paths, $\ol x = (x_n)_{n \in \N_0}$ and $\ol y = (y_n)_{n \in \N_0}$, are called \textit{parallel}. In this paper, we provide a class of horizontally stationary generalized Bratteli diagrams with no probability tail invariant measure (see \Cref{sect B(2^k)}). Such diagrams are also studied in \Cref{Ex_hor_stat_IO}, \Cref{Ex:OmegaYdense} and \Cref{ex:0-1proc}.

We also recall the notion of isomorphism of two Bratteli diagrams (see, e.g., \cite[Page 326]{Durand2010}, \cite[Definition 2.8.]{BezuglyiJorgensenKarpelSanadhya2025}):

\begin{definition} \label{def_isom BD}
Two (standard or generalized) Bratteli diagrams $B = (V, E)$ and $B' = 
(V', E')$ 
are called \textit{isomorphic} if there exist two sequences of
bijections $(g_n : V_n \rightarrow V_n')_{n \in \N_0}$ and 
$(h_n : E_n \rightarrow E_n')_{n \in \N_0}$ such that 
for every $n \in \N_0$, we have $g_n(V_n) = V_n'$ and $h_n(E_n) = 
E_n'$, and $s' \circ h_n = g_n \circ s$, $r' \circ h_n = 
g_n \circ r$, where $s'$ and $r'$ are source and range maps in $B'$.

\end{definition} 

\section{Subdiagrams and measure extension from a subdiagram}
\label{sect meas ext subdgrms}

In this section, we consider the procedure for the measure extension from vertex and edge subdiagrams of a generalized Bratteli diagram. We will apply this method to some specific classes of diagrams in the following sections. 

\subsection{Vertex and edge subdiagrams and measure extension}

Let $B = (V,E)$ be a generalized Bratteli diagram and let 
$(F_n)_{n \in \N_0}$ be the corresponding sequence of incidence matrices. To produce a subdiagram, one removes some edges 
from $E_n$ or restricts the diagram to a proper subset of vertices 
$W_n$ of $V_n$ for every level $n$. In this way, we construct 
edge and vertex subdiagrams, see the definition below.  
More generally, a closed subset $Y$ of the path space $X_B$ 
represents the path space of a non-trivial subdiagram if and only if $\mathcal R|_{Y \times Y}$ is a countable Borel equivalence relation. 

We will study the two basic cases of vertex and edge subdiagrams separately, although some ideas and results are similar for both classes. Moreover, it is possible to represent an edge subdiagram as a vertex subdiagram after using the $ 0-1$ procedure described below. As a rule, objects related to a subdiagram are denoted by barred symbols.

\begin{definition}\label{def: vertex sbdgr} 
Let $B = (V, E)$ be a generalized Bratteli diagram. Consider a sequence $\overline W = (W_n)_{n \in \N_0}$ of proper subsets $W_n$ of $V_n$. In other words, $W'_n = V_n \setminus W_n \neq \emptyset$ for all $n \in \N_0$. The {\em vertex subdiagram} $\overline B =  (\ol W, \ol G)$ is formed by the vertices $\ol W= (W_n)_{n \in \N_0}$ and the set of edges $\ol G = (G_n)_{n \in \N_0}$ where for every $n\in \N_0$, $e \in  G_n$ if and only if $s(e) \in W_{n}$ and $r(e) \in W_{n+1}$. We call the sequence $(W_n)_{n \in \N_0}$ \textit{admissible} if
$\ol B$ is a generalized (or standard) Bratteli diagram.
 \end{definition} 

\begin{remark} 
We mention here the following obvious properties of subdiagrams.
\begin{enumerate}
    \item A sequence $\ol W = (W_n)_{n \in \N_0}$ of subsets of $(V_n)_{n \in \N_0}$ is admissible if for all $n \in \N_0$
$$
W_n \subset \bigcup_{v \in W_{n+1}} s(r^{-1}(v)). 
$$
\item Suppose that $\ol B =(\ol W, \ol G)$ is a subdiagram of a generalized Bratteli diagram $B$. 
If the set $W_n$ is finite, then every $W_m$ is finite for $m < n$. If there exists a level $n$ such that $W_n$ is infinite, then all levels $W_m$ for $m \geq n$ are infinite. Indeed, every vertex of $W_n$ should have an outgoing edge, and every vertex in $W_{n+1}$ should have only finitely many incoming edges. After telescoping up to level $n$, we obtain that $\ov B$ is a generalized Bratteli diagram. Hence, without loss of generality, we can assume that either all $W_n$ are
finite or infinite.

\item For $n \in \N_0$ the incidence matrix $\ol F_n$ of $\overline B =  (\ol W, \ol G)$ has the size $|W_{n+1}| \times |W_n|$, and it is represented by a truncated part of $F_n$ corresponding to the vertices from $W_{n}$ and $W_{n+1}$. We say, in this case, that $\ol W = (W_n)_{n \in \N_0}$ is the support of $\ol B$.
\end{enumerate}
\end{remark} Recall that for infinite matrices $A = (a_{i,j}: i,j \in \Z)$, $B = (b_{i,j}: i,j \in \Z)$, we say that $A \leq B$ if for every $i,j \in \Z$, we have $a_{i,j} \leq b_{i,j}$.
\begin{definition}\label{Def_edge_Sub_dig} A subdiagram $\overline{B} (\overline{F}_n)$ of a generalized Bratteli $B= B(F_n)$ is called an {\em edge subdiagram} 
 if it is supported by all vertices $V_n$, $n \in \N_0$, and 
 whose incidence matrices $(\overline{F}_n)_{n \in \N_0}$ are countably infinite matrices such that $\overline{F}_n \leq F_n$ for every $n \in \mathbb{N}_0$. 
 \end{definition}
 Thus, by definition, an {\em edge subdiagram} is obtained from the diagram $B$ by ``removing'' some edges and leaving all vertices of $B$ unchanged. It is natural to assume that, after removing edges from $B$, we have a nontrivial subdiagram whose path space 
 $X_{\ol B}$ is an uncountable Polish space 
 supporting the countable tail equivalence relation $\ol{\mc R}$.
 
 For every $n \in \N_0$, we denote $F'_n = F_n - \overline{F}_n$. We also assume, without loss of generality, that  $\overline{F}_n < F_n$ for infinitely many $n$. It is not hard to see that the study of any subdiagram $\ol B$ of $B$ is reduced to the cases of edge and vertex subdiagrams. Indeed, $\ol B$ can be viewed as an edge subdiagram of a vertex subdiagram of $B$.

\begin{remark}\label{rem on edge diagram}
For a generalized Bratteli diagram $B(F_n)$, we define an edge subdiagram $\ol B$ by removing some edges from the sets $E_n$. It produces the sequence of new incidence matrices $(\ol F_n)$ such that $\ol F_n \leq F_n$. The entries $\ol f^{(n)}_{v, w}$ show the number of edges from $\ol B$ that connect the vertices $w \in V_n$ and $v\in V_{n+1}$. But this information is not sufficient to define $\ol B$ uniquely because we do not know which edges of 
the set $E(w, v)$ are removed.

We also observe that a subset $Y$ of the path space $X_B$ can be defined by forbidding some cylinder sets. For example, if $C = [e_n, \ldots, e_{n+m}]$ is a fixed cylinder set, consider the set $Y_C$ that consists of all infinite paths that do not go through $C$. We note that the set $Y_C$ is not the path space of an edge subdiagram because we do not remove any edges. That is, for every edge $e_i$ from $C$, there exists an infinite path in $X_B$ that belongs to $[e_i]$. In this case, we have $\ol F_n = F_n$ for all $n$. This definition is similar to the definition of the shift of finite type in symbolic dynamics. In the present paper, we do not consider this procedure.  
\end{remark}

\subsection{Measure extension from a subdiagram.} 

 Similarly to \Cref{Def:Kakutani-Rokhlin}, given $n \in \N_0$ and $v \in W_n$, we set 
\[
\overline{X}_v^{(n)} := \{x = (x_i)\in X_{\overline B} : s(x_{n}) = v\}.
\]
Then $\overline{X}_v^{(n)}$ denotes the ``inner tower'' in a subdiagram $\overline{B}$ that is determined by a vertex $v$. This set is formed by the cylinder subsets of $X_{\ol B}$ corresponding to the finite paths in $\ol B$ that terminate at $v \in W_n$. For $n \in \N$, the number of such paths is denoted by $\overline{H}_v^{(n)}$, and it is the height of the tower $\ol X_v^{(n)}$. As before, we assume that every entry $H_{v}^{(0)}$ of the vector $\ol H^{(0)}$ is equal to $1$.

Similarly to \Cref{Def:Height}, for $n \in \N$, we 
set $\ol H^{(n)} = \langle \ol H^{(n)}_v : v \in V_n \rangle$. Then 
\begin{equation}\label{formula for heights}
\ol F_n \ol H^{(n)} = \ol H^{(n+1)}, \ \ n \in \N_{0}.
\end{equation}

Suppose that $\ol B$ is a (vertex or edge) subdiagram of a generalized Bratteli diagram $B$. Then we can consider an ``inner'' probability measure $\ol\mu$ defined on the path space $X_{\ol B}$ and invariant with respect to $\mathcal {\ol R}$. It is obvious that $\mathcal R|_{X_{\ol B}\times X_{\ol B}} = \mathcal {\ol R}$. For $n \in \N_0$ and vertex $v$ from the $n$-th level of $\ol B$, we set 
\begin{equation}\label{eq: p_v_n_bar}
  \overline{p}_v^{(n)} = \dfrac{\overline{\mu}(\overline{X}_v^{(n)})}{\overline{H}_v^{(n)}}, \quad n  \in \N_0.  
\end{equation} Then, $\overline{p}_v^{(n)}$ is the $\ol \mu$-measure of a single cylinder set $[\ol e] \subset X_{\ol B}$  with $r(\ol e) =v$. 

Let $\widehat X_{\ol B}:= \mathcal R(X_{\ol B})$ be the subset of paths in $X_B$ that are tail equivalent to paths from $X_{\ol B}$. In other words, the $\mathcal R$-invariant subset $\widehat X_{\ol B} $ of $X_B$ is the saturation of $X_{\ol B}$ with respect to the equivalence relation $\mathcal R$ (or $X_{\ol B}$ is a countable complete section of  $\mathcal R$ on $\widehat X_{\ol B}$). Let $\ov \mu$ be a probability measure on $X_{\ol B}$ invariant with respect to the tail equivalence relation $\mathcal{\ol R}$ defined on $\ol B$. Then $\ov \mu$ can be canonically extended to the measure $\widehat{\ov \mu}$ (finite or infinite) on the space $\widehat X_{\ol B}$ by tail invariance. In case of need, we can think that $\widehat{\ov \mu}$ is extended to the whole space $X_{B}$ by setting $\widehat {\ov \mu} (X_B \setminus \widehat{X}_{\ol B}) = 0$.

Specifically, take a finite path $\ol e$ in the subdiagram $\ol B$ from a vertex $v_0$ to a vertex $v$. 
Let $[\ol e]$ denote the cylinder subset of $X_{\ol B}$ determined by $\ol e$. For any finite path $f \in E(V_0, v)$ from the diagram $B$ with the same range $v$, we set $\widehat{\ov \mu} ([f])  = \ov \mu([\ol e])$. In such a way, the measure $\widehat{\ov \mu}$ is extended to the  $\sigma$-algebra of Borel subsets of $\wh X_{\ol B}$ generated by all cylinder sets.
By construction, $\wh {\ov \mu}$  is $\mathcal R$-invariant and its restriction on $X_{\ol B}$ coincides with $\ov \mu$. We note that the total value $\wh{\ov \mu}(\wh X_{\ol B})$ can be either finite or infinite depending on the structure of $\ol B$ and $B$ (see \eqref{extension_method}, \eqref{extension_method_vertex},\Cref{thm3 sufficient cond} and \Cref{edge_fin_krit} below). The set $\wh X_{\ol B}$  is said to be the {\em support} of $\widehat{\ov \mu}$.

\begin{remark}
Denote by $\wh X_{\ol B}^{(n)}$ the set of all paths 
$x = (x_i)_{i= 0 }^{\infty}$  from $X_B$ such that $(x_0, \ldots, x_{n-1})$ is a finite path in $B$ which starts at level $V_0$ and ends at some vertex $v$ from $\ol B$, and the tail  $(x_{n},x_{n+1},\ldots)$ belongs to $\overline{B}$.
For instance, for a vertex subdiagram $\ov B$ and $n \in \N_0$ we have
\begin{equation}\label{n-th level}
\wh X_{\ol B}^{(n)} = \{x = (x_i)\in \wh X_{\ol B} : s(x_i) \in W_i, \ \forall i \geq n\}.
\end{equation}
It is obvious that for $n \in \N_0$, $\wh X_{\ol B}^{(n)} \subset \wh X_{\ol B}^{(n+1)}$ and
\begin{equation}\label{extension_method}
\widehat{\ov \mu}(\wh X_{\ol B}) = \lim_{n\to\infty} \widehat{\ov \mu}(\wh X_{\ol B}^{(n)}).
\end{equation}
More precisely, if $\ov B$ is a vertex subdiagram then, for $n \in \N_0$, $\widehat{\ov \mu}(\wh X_{\ol B}^{(n)}) = \sum_{w\in W_n}  H^{(n)}_w \ov p^{(n)}_w$ and thus we have
\begin{equation}\label{extension_method_vertex}
\widehat{\ov \mu}(\wh X_{\ol B}) = \lim_{n\to\infty}\sum_{w\in W_n}  H^{(n)}_w \ov p^{(n)}_w
\end{equation}
where for $n \in \N$, $\ol p_w^{(n)}$ is the measure of the cylinder set from $X_{\ol B}$ terminating at $w \in V_n$ and $H^{(n)}_w$ is the height of $w$-tower in $B$. 
\end{remark}

In the case of an edge subdiagram $\ol B$, the vertex set of $\ol B$ at level $n \in \N_0$ is $V_n$ and we obtain  a slightly different formula
\begin{equation}\label{extension_method_edge}
\widehat{\ov \mu}(\wh X_{\ol B}) = \lim_{n\to\infty}\sum_{w\in V_n}  H^{(n)}_w \ov p^{(n)}_w.
\end{equation}

\subsection{Measure extension from a vertex subdiagram} 
In this subsection, we prove the statements focusing on conditions for the finiteness of measure extension from vertex subdiagrams. 

We fix our notation here. Let $B = B(F_n)$ be a generalized Bratteli 
diagram determined by the sequence of incidence matrices $(F_n)$. 
By $\overline B$, we denote a vertex subdiagram supported by 
a sequence $(W_n)$, $W_n \subset V_n$. Recall that $|W_n|$ can be finite or infinite. Let $\ol \mu$ be a probability tail invariant measure on the path space $X_{\ol B}$. 
The extension of $\ol \mu$ onto  $\wh X_{\overline B}$ is denoted
by $\wh{\ol \mu}$. 
By $(Q_n)_{n \in \N_0}$ we denote the sequence of row stochastic matrices defined 
in terms of $(F_n)_{n \in \N_0}$ and the vector of heights, see \eqref{eq_stoch matrix Q_n}. In our formulas below, for $n \in \N_0$ we also use the vector $\ol p^{(n)}$
defined by the measure $\ol \mu$ similarly to \eqref{eq:def_p_n} and \eqref{eq:formula_p_n}. We recall that notations containing a bar are related to subdiagrams, in general. 

The following theorem gives a criterion for $\wh{\ol  \mu} \,\big( \wh X_{\overline B} \big)$ to be finite. Recall for $n \in \N_0$, $W'_n = V_n \setminus W_n$,

\begin{theorem} \label{thm3 criterion vertex subd}
Let $\overline B$ be a vertex subdiagram of $B(F_n)$, $n \in \N_0$ and let $\ol \mu$, $H_v^{(n)}$, $\ol p^{(n)}$, $(Q_n)$, $X_{\overline B}$, and $\wh X_{\overline B}$ be as defined above. Then we have 
\be\label{eq3:three criteria}
\ba
\wh{\ol \mu}(\wh X_{\overline B})  < \infty  
\Longleftrightarrow & \sum_{n=0}^\infty \sum_{v \in W_{n+1}} 
\sum_{w \in W'_n} f^{(n)}_{v,w} H_w^{(n)} \ol p_v^{(n+1)} < \infty \\
\Longleftrightarrow & \ 
\sum_{n=0}^\infty \sum_{v \in W_{n+1}} \wh{\ol \mu }
(X_v^{(n+1)}) \sum_{w \in W'_n} q^{(n)}_{v,w} < \infty \\
\Longleftrightarrow & \ 
\sum_{n=0}^\infty \,\,\bigg( \sum_{v \in W_{n+1}}
H_v^{(n+1)}\overline p_v^{(n+1)} - 
 \sum_{w \in W_{n}} H_w^{(n)}\overline p_w^{(n)}\bigg) 
< \infty.
\ea
\ee
\end{theorem}

\begin{proof} 
A detailed proof can be found in \cite{BezuglyiKarpelKwiatkowski2015}, where this result is proved for standard Bratteli diagrams (see \cite[Theorem 2.2]{BezuglyiKarpelKwiatkowski2015}). The case of generalized Bratteli diagrams is considered analogously. 
The statement is based on the following representation of the set $\wh X_{\ol B}$:
\begin{equation*}
     \wh X_{\overline B} = \wh X_{\overline B}^{(0)} \,\bigcup\, \big( \wh X_{\overline B}^{(1)} \setminus \wh X_{\overline B}^{(0)} \big)\,\bigcup\, \big( \wh X_{\overline B}^{(2)} \setminus \wh X_{\overline B}^{(1)} \big)\,\bigcup \cdots.
\end{equation*} 
Thus, it follows that 
\begin{align}
    \wh{\ol  \mu} \,\big( \wh X_{\overline B} \big)   =&\  1 + \underset{n=0}{\overset{\infty}{\sum}}\, \wh{\ol  \mu} \,\,\big( \wh X_{\overline B}^{(n+1)} \setminus \wh X_{\overline B}^{(n)} \big) \nonumber\\
=&\ 1 + \sum_{n=0}^\infty \sum_{v \in W_{n+1}} 
\sum_{w \in W'_n} f^{(n)}_{v,w} H_w^{(n)} \ol p_v^{(n+1)}\label{eq: exten_representation}\\
=&\  1 +  \sum_{n=0}^\infty \,\,\bigg( \sum_{v \in W_{n+1}}
H_v^{(n+1)}\overline p_v^{(n+1)} - 
 \sum_{w \in W_{n}} H_w^{(n)}\overline p_w^{(n)}\bigg).\label{eq: exten_representation_1}
\end{align}
We leave the details to the reader. 
\end{proof}

The following statement provides a convenient sufficient condition 
for the finiteness of the extension of a measure from a vertex subdiagram. It is important to note that the condition
of \Cref{thm3 sufficient cond} is formulated in terms 
of the (stochastic) incidence matrix. The conclusion of this theorem 
holds for any probability tail invariant measure $\ol\mu$ on $\ol B$.
We use here the notations from \Cref{thm3 criterion vertex subd}. For a similar result in the case of standard Bratteli diagrams, see
\cite[Theorem 2.2]{AdamskaBezuglyiKarpelKwiatkowski2017}.

\begin{theorem}\label{thm3 sufficient cond}
Let $B$ be a generalized Bratteli diagram, 
$(Q_n)_{n \in \N_0}$ the sequence of stochastic matrices, and $\overline B$ a vertex subdiagram supported by $(W_n)_{n \in \N_0}$. If \begin{equation}\label{eq3:suff cond finiteness}
    \sum_{n=0}^\infty\,\,\underset{v \in W_{n+1}}{\sup} 
    \Big(\,\underset{w \in W'_{n}}{\sum} q_{v,w}^{(n)} \Big) < \infty,
\end{equation} then for any probability tail invariant measure $\ol\mu$ on $X_{\overline B}$, the extension $\wh{\ol\mu} \big(\wh X_{\overline B}\big)$ is finite.
\end{theorem}

\begin{proof}
Let $\ol H_v^{(n)}$ be the height of the inner tower $\ol X_v^{(n)}$, $v \in W_n$. By \eqref{eq: exten_representation}, we have 
\begin{equation}\label{eq3 estimate}
\ba 
    \wh{\ol \mu} \,\big( \wh X_{\overline B} \big) = &\ 1 + \,\, \underset{n=0}{\overset{\infty}{\sum}}\,\, \underset{v \in W_{n+1}}{\sum}\,\, \underset{w \in W'_n}{\sum} f_{v,w}^{(n)}\,\, H_{w}^{(n)}\,\, \overline p_v^{(n+1)}\\
 = &\ 1+ \sum_{n =0}^\infty \underset{v \in W_{n+1}}{\sum}\,\,  \underset{w \in W'_n}{\sum} f_{v,w}^{(n)}\, \frac{H^{(n)}_{w}}{H_v^{(n+1)} }\,\, (\overline p_v^{(n+1)}\,  
 \ol H_v^{(n+1)})  \cdot \frac{H_v^{(n+1)}}{\ol H_v^{(n+1)}}\\
    = &\ 1+ \sum_{n =0}^\infty \underset{v \in W_{n+1}}{\sum}\,\, \ol \mu \Big(\overline X_v^{(n+1)}\Big) \dfrac{ H_v^{(n+1)}}{\overline H_v^{(n+1)}}  \underset{w \in W'_n}{\sum} q^{(n)}_{v,w} \\
&\ \leq 1+  \sum_{n =0}^\infty \underset{v \in W_{n+1}}{\mathrm{sup}} \Big(\underset{w \in W'_n}{\sum} q_{v,w}^{(n)}\Big) \underset{v \in W_{n+1}}{\sum}\,\, 
\ol\mu \Big(\overline X_v^{(n+1)}\Big) \frac{ H_v^{(n+1)}}{\overline H_v^{(n+1)}}. 
    \ea
\end{equation}
Next, we estimate the quantity $\dfrac{H_v^{(n+1)}}{\overline H_v^{(n+1)}}$, $v \in W_{n+1}$, from above. Let 
\be\label{eq3 M_n}
M_n = \sup_{w \in W_n} \  \frac{H_w^{(n)}}{\overline H_w^{(n)}}  
\ee 
(it will be shown below that $M_n$ is finite). Then using that 
$F_n H^{(n)} = H^{(n+1)}$ and \eqref{eq3 M_n}, we write for 
$v \in W_{n+1}$
$$
\ba 
\frac{H_v^{(n+1)}}{\overline H_v^{(n+1)}} = &\ 
\frac{1}{\overline H_v^{(n+1)}} \Bigg(\underset{w \in W_n}{\sum} f_{v,w}^{(n)} H_w^{(n)} + \underset{w \in W'_n}{\sum} f_{v,w}^{(n)} H_w^{(n)}\Bigg) \\
\leq  &\ 
\frac{M_n}{\overline H_v^{(n+1)}} \underset{w \in W_n}{\sum} \ol f_{v,w}^{(n)} \overline H_w^{(n)} + \frac{1}{\overline H_v^{(n+1)}}  \underset{w \in W'_n}{\sum} f_{v,w}^{(n)} H_w^{(n)}.
\ea
$$
Since  $\overline H^{(n+1)}_v = \sum_{w \in W_n} \overline f_{v,w}^{(n)}\, \overline H^{(n)}_w$,  we get
$$
\ba
\frac{H_v^{(n+1)}}{\overline H_v^{(n+1)}} \leq &\  
M_n + \frac{H_v^{(n+1)}}{\overline H_v^{(n+1)}} \underset{w \in W'_n}{\sum} f_{v,w}^{(n)} \,\, \dfrac{H_w^{(n)}}{H_v^{(n+1)}}\\
= &\ M_n + \dfrac{H_v^{(n+1)}}{\overline H_v^{(n+1)}} \underset{w \in W'_n}{\sum} q_{v,w}^{(n)}\\
\leq &\  M_n + \dfrac{H_v^{(n+1)}}{\overline H_v^{(n+1)}} \underset{v \in W_{n+1}}{\sup}\bigg(\underset{w \in W'_n}{\sum} q_{v,w}^{(n)}\bigg).
\ea
$$
Denote $\varepsilon_n = \underset{v \in W_{n+1}}{\sup}\bigg(\underset{w \in W'_n}{\sum} q_{v,w}^{(n)}\bigg)$. Thus, we get for $v \in W_{n+1}$
$$
 \dfrac{H_v^{(n+1)}}{\overline H_v^{(n+1)}} (1 - \varepsilon_n) \leq M_n 
$$ 
and therefore, by definition of $M_n$, we have
$
 M_{n+1} \leq \dfrac{M_n}{1- \varepsilon_n}$. Hence,
$$
    M_n \leq \dfrac{M_1}{\underset{i=1}{\overset{n}{\prod}}(1- \varepsilon_i)} \leq \dfrac{M_1}{\underset{i=1}{\overset{\infty}{\prod}}(1- \varepsilon_i)} : = M.
$$
We note that, by \eqref{eq3:suff cond finiteness}, $\sum_{n=0}^{\infty} \varepsilon_n < \infty$, the infinite product 
$\prod_{i=1}^{\infty}(1- \varepsilon_i)$
converges. To finish the proof, we use the upper bounds found above and estimate 
the $n$-th term of \eqref{eq3 estimate} as follows:
$$
\ba 
 \underset{v \in W_{n+1}}{\sum}\,\, \underset{w \in W'_n}{\sum} f_{v,w}^{(n)}\,\, H_{w}^{(n)}\,\, \overline p_v^{(n+1)} 
 \leq &\ 
 \underset{v \in W_{n+1}}{\mathrm{sup}} \Big(\underset{w \in W'_n}{\sum} q_{v,w}^{(n)}\Big) M_{n+1} \underset{v \in W_{n+1}}{\sum}\,\, \ol\mu \Big(\overline X_v^{(n+1)}\Big) \\ 
 = &\ \underset{v \in W_{n+1}}{\mathrm{sup}} \Big(\underset{w \in W'_n}{\sum} q_{v,w}^{(n)}\Big) M_{n+1},\\
\ea
$$
where we used $\sum_{v \in W_{n+1}} \ol\mu ( X_v^{(n+1)}) = 1$.
Finally, we obtain 
\begin{equation}\label{mes}
\wh{\ol  \mu} \,\big( \wh X_{\overline B} \big) \leq 1 + M \,\, \underset{n=0}{\overset{\infty}{\sum}} \,\, \underset{v \in W_{n+1}}{\mathrm{sup}} \Big(\underset{w \in W'_n}{\sum} q_{v,w}^{(n)}\Big).
\end{equation}. 
Again using \eqref{eq3:suff cond finiteness}, we get $\wh{\ol  \mu} \,\big( \wh X_{\overline B} \big) < \infty$. 
\end{proof} 

\begin{remark}
It is important to note that the condition \eqref{eq3:suff cond finiteness} implies that the heights of the towers $X^{(n)}_w$
and $\ol X^{(n)}_w$ in $B $ and $\ol B$ grow proportionally as $n$ tends to infinity. This means that there exists some $M$ such that
for all $w \in W_n$
\be\label{eq3: heights ratio}
\frac{ H^{(n)}_w}{\ol H^{(n)}_w} \leq M.
\ee
 In other words, if the measure extension $\wh{\ol\mu}(\wh X_{\ol B})$ is finite,  then
\eqref{eq3: heights ratio} holds. 

Reformulating this result, we can
state that if there exists a sequence $(n_k)$ such that, for some $w\in W_n$,
$$
\frac{\ol H^{(n_k)}_w}{ H^{(n_k)}_w} \to 0, \quad n_k \to \infty,
$$
then the extension of the measure $\wh{\ol\mu}(\wh X_{\ol B})$ is infinite. 
\end{remark}

Let $B$ be a generalized Bratteli diagram and $\mu$ be a probability tail invariant measure on $X_B$. Let $\ov B$ be a (vertex) subdiagram of $B$ defined by a sequence of vertices $(W_n)_{n \in \N_0}$, where $W_n \subset V_n$ for every $n \in \N_0$. For every $n \in \mathbb{N}$ and every $v \in W_n$, we recall that 
$$
\ov H_w^{(n)} = |\{\ov e \in E(W_0, w): \ov e \mbox{ lies in } \ov B\}|
$$ 
and
$$
H_w^{(n)} = |\{\ov e \in E(V_0, w): \ov e \mbox{ lies in } B\}|.
$$

The following result gives a sufficient condition for the measure of the path space of a subdiagram to be zero.

\begin{prop}
Assume that for every $\varepsilon>0$ there exists $n : = n_{\varepsilon} \in \N_{0}$ such that
$$
\forall w \in W_n, \quad \dfrac{\ov H_w^{(n)}}{H_w^{(n)}} < \varepsilon.
$$
Then $\mu(X_{\overline{B}}) = 0$.
\end{prop}

\begin{proof} Fix $\varepsilon > 0$ and let $n = n_{\varepsilon}$ be such that for every $w \in W_n$ $\dfrac{\ov H_w^{(n)}}{H_w^{(n)}} < \varepsilon$. Observe that
$$
\sum_{w \in W_n} H_w^{(n)} p_w^{(n)} < \sum_{w \in V_n} H_w^{(n)} p_w^{(n)} = 1.$$
Hence we have
$$
\mu(X_{\overline{B}}) \leq \mu(\bigcup_{w \in W_n} \ol X_w^{(n)}) = \sum_{w \in W_n} \ov H_w^{(n)}p_w^{(n)} = \sum_{w \in W_n} \dfrac{\ov H_w^{(n)}}{H_w^{(n)}} H_w^{(n)} p_w^{(n)} < \varepsilon \sum_{w \in W_n} H_w^{(n)} p_w^{(n)} < \varepsilon.
$$ Since $\varepsilon>0
$ is arbitrary, we obtain $\mu(X_{\overline{B}}) = 0$.
\end{proof}

\begin{remark}
It is important to note that the infiniteness of the measure extension from a vertex subdiagram $\ov B = \ov B(W_n)$ may arise for a different reason. In the cases considered above, like 
\Cref{thm3 criterion vertex subd}, the infinity of $\wh{\ol\mu}(\wh X_{\ol B})$ was caused by the divergence of a series indexed by $n$. By contrast, when 
 $\ol B$ itself is a generalized Bratteli diagram - that is, when each $W_n$ is a countable set for all $n$ - the value $\wh{\ol\mu}(\wh X_{\ol B})$ may already be infinite at some level $V_n$ due to the summation over all $w \in W_n$. It is not difficult to construct examples illustrating this phenomenon.
 
\end{remark}

\begin{example}
To illustrate the result of \Cref{thm3 sufficient cond}, we consider the case of a generalized Bratteli diagram $B = B(F_n)$ 
where every matrix $F_n$ for $n\in \N_0$ has the equal row sum property. In other words, 
for every $n \in N_{0}$ and $v \in V_{n+1}$ there is an $r_n \in \N$, such that 
\be\label{eq3:ERS}
\sum_{w \in V_n} f^{(n)}_{v,w} = r_n.
\ee
We write, in this case, that $B \in ERS(r_n)$. As a simple consequence
of \eqref{eq3:ERS}, we have
$$
H^{(n)}_v = r_0\ \cdots \ r_{n-1}, \quad v \in V_n.
$$
Hence, it follows from \eqref{eq_stoch matrix Q_n} that
$q^{(n)}_{v,w} = \dfrac{1}{r_n}f^{(n)}_{v,w}$. 
Then \eqref{eq3:suff cond finiteness} has the form (here $W'_n =
V_n \setminus W_n$)
\be\label{eq3:example}
\ba
R = & \sum_{n=0}^\infty\  \sup_{v \in W_{n+1}} \left( \frac{1}{r_n}
\sum_{w \in W'_n} f^{(n)}_{v,w}\right)\\ 
= & \sum_{n=0}^\infty\ \frac{1}{r_n} \sup_{v \in W_{n+1}} \left(
\sum_{w \in (V_n \setminus W_n)} f^{(n)}_{v,w}\right)\\
= & \sum_{n=0}^\infty\ \frac{1}{r_n} \sup_{v \in W_{n+1}} \left(
\sum_{w \in V_n} f^{(n)}_{v,w} -  \sum_{w \in W_n} \ol f^{(n)}_{v,w}  \right)\\ 
= &  \sum_{n=0}^\infty\ \frac{1}{r_n} \sup_{v \in W_{n+1}} \left( r_n - \sum_{w \in W_n}\ol f^{(n)}_{v,w}\right)\\
= & \sum_{n=0}^\infty\ \frac{1}{r_n} \sup_{v \in W_{n+1}} ( r_n - \ol r_v^{(n)})
< \infty, 
\ea
\ee
where $\ol r_v^{(n)}$ is the sum of entries in the $v$-th row of 
$\ol F_n$. 
If additionally, the subdiagram $\ol B  = \ol B(\ol F_n) $ has
incidence matrices $\ol F_n$ that satisfy the ERS property with the sum of the rows equal to $\ol r_n$, then the condition
$$
R = \sum_{n=0}^\infty\ (1 - \frac{\ol r_n}{r_n}) < \infty
$$
guarantees the finiteness of the measure extension $\wh{\ol\mu}
(\wh X_{\ol B})$. 
\end{example}

\subsection{Measure extension from an edge subdiagram}
In this subsection, we consider the measure extension 
from an edge subdiagram.  

\begin{thm}\label{edge_fin_krit}
Let $\overline{B} =  B(\ol F_n)$ be an edge subdiagram of a generalized Bratteli diagram $B = B(F_n)$. For a tail invariant probability measure $\overline{\mu}$ on $X_{\overline{B}}$, the measure extension from $\ol B$ is determined by the formula:
\be\label{eq s3 edge meas ext}
\widehat{\overline{\mu}}(\widehat{X}_{\overline{B}})  = 1+ 
\sum_{n=0}^\infty \sum_{v\in V_{n+1}} \sum_{w\in V_{n}} {f'}_{v,w}^{(n)} H_w^{(n)} \overline{p}_v^{(n+1)},
\ee
where ${f'}^{(n)}_{v,w} = {f}^{(n)}_{v,w} - \ol{f}^{(n)}_{v,w}$ is the entry of 
$F'_n = F_n - \ol F_n$ and $\overline{p}_v^{(n+1)}$ is the measure
$\ol\mu$ of a cylinder set with range $v \in V_{n+1}$. 
Hence, the extension $\widehat{\overline{\mu}} (\widehat{X}_{\overline{B}})$ is finite if and only if
\begin{equation}\label{eq:edge_sbd_mu_ext}
\sum_{n=0}^{\infty} \sum_{v \in V_{n+1}} \sum_{w \in V_{n}} 
{f'}_{v,w}^{(n)} H_w^{(n)}\overline{p}_v^{(n+1)} < \infty.
\end{equation}
\end{thm}

\begin{remark} Formula \eqref{eq s3 edge meas ext} remains true when both sides are infinite. We also note that the case when the measure extension can be infinite 
is different for generalized Bratteli diagrams. 
In contrast to standard Bratteli diagrams, there exist edge 
generalized subdiagrams such that 
$$
\sum_{v\in V_{n+1}} \sum_{w\in V_{n}} {f'}_{v,w}^{(n)} H_w^{(n)} \overline{p}_v^{(n+1)} = \infty
$$
for some level $n$ because the summation is over an infinite set. Hence, the extension of the measure
$\widehat{\overline{\mu}}(\widehat{X}_{\overline{B}})$ will be finite if both of the following conditions are satisfied: for every level, the added edges produce
a finite extension of the measure, and the sum of these increments over $n$ 
converges. Clearly, the second condition is more difficult to control.
\end{remark}

\begin{proof}[Proof of \Cref{edge_fin_krit}] Recall that $\widehat X_{\ol B} $ is the saturation of edge subdiagram $X_{\ol B}$ with respect to the tail equivalence relation $\mathcal R$. As before for $w \in V_n$ and $n \in \N_0$, we denote
\[
\widehat{X}_w^{(n)} = \{x = (x_i)\in \widehat{X}_{\ol B} : s(x_{n}) = w\}, \hspace{1cm} \wh X_{\ol B}^{(n)} = \bigcup_{w \in V_n} \widehat{X}_w^{(n)}.
\]
For a probability measure $\overline{\mu}$ on $\overline{B}$, as before we have
\[
\widehat{\overline{\mu}}
 (\wh X _{\ov B}) = \lim_{n \rightarrow \infty}
\widehat{\overline{\mu}} (\wh X_{\ov B}^{(n)}).
\]
For $n \in \N_0$, let ${H'}_w^{(n)} = H_w^{(n)} - \ov H_w^{(n)}$ and set
$$
R_{n} = \sum_{w \in V_{n}} {H'}_w^{(n)} \ov p_w^{(n)}.
$$
Then for $n \in \N_0$ we have $\widehat{\overline{\mu}} (\wh X_{\ov B}^{(n)}) = \ov \mu (X_{\ov B}) + R_n = 1 + R_n$.
In other words, $R_n$ shows the value of the measure of added cylinder sets ending at level $n$. Hence $R_n$ can be represented as $\sum_{i= 1}^n (R_i - R_{i-1})$ (Note that $R_0 = 0$). The quantity $R_i - R_{i-1}$ indicates the increment of the measure when 
the length of cylinder sets is increased by one. Note that for $n \in \N_0$
$$
{H'}_v^{(n+1)} = \sum_{w \in V_n}  {f'}_{v,w}^{(n)} H_w^{(n)} +  \sum_{w \in V_n} \ov f_{v,w}^{(n)} {H'}_w^{(n)}.
$$ 
Hence we have for $n \in \N_0$,
\begin{equation}\label{formula}
R_{n+1} = \sum_{v \in V_{n+1}} \sum_{w \in V_n} {f'}_{v,w}^{(n)} H_w^{(n)} \ov p_v^{(n+1)} + \sum_{v \in V_{n+1}} \sum_{w \in V_n} \ov f_{v,w}^{(n)} {H'}_w^{(n)} \ov p_v^{(n+1)}.
\end{equation}
On the other hand, we can represent the second summand in (\ref{formula}) as follows:
$$
\sum_{v \in V_{n+1}} \sum_{w \in V_n} \ov f_{v,w}^{(n)} {H'}_w^{(n)} \ov p_v^{(n+1)} = \sum_{w \in V_n}  {H'}_w^{(n)} \ov p_w^{(n)} = R_n.
$$
Thus for $n \in \N_0$
\be\label{eq s3 R(n+1) - R(n)}
R_{n+1} - R_n = \sum_{v \in V_{n+1}} \sum_{w \in V_n} {f'}_{v,w}^{(n)} H_w^{(n)} \ov p_v^{(n+1)}.
\ee
Hence it follows
$$
\wh{\ol\mu} (\wh X_{\ol B}) = 1 + \lim_{n\to\infty} R_n =
1+ \lim_{n\to\infty} \sum_{i= 1}^n (R_i - R_{i-1}) =
1+ \lim_{n\to\infty} \sum_{i= 0}^{n-1} (R_{i+1} - R_{i})
$$
Using \eqref{eq s3 R(n+1) - R(n)}, we get 
\begin{eqnarray*}
\widehat{\overline{\mu}}(\widehat{X}_{\overline{B}})
  =  1 + \sum_{n=0}^\infty \sum_{v\in V_{n+1}} \sum_{w\in V_{n}} {f'}_{v,w}^{(n)} H_w^{(n)} \overline{p}_v^{(n+1)},
\end{eqnarray*} as needed.
\end{proof} 

\begin{thm}
Let $B$, $\ov B$, $\ov \mu$ be as in \Cref{edge_fin_krit}. 
Define for $n \in \N_0$,
\be\label{eq s3 sup-inf}
a_n = \sup_{v \in V_n} \frac{H_v^{(n)}}{\ol H_v^{(n)}},\qquad
b_n = \inf_{v \in V_n} \frac{H_v^{(n)}}{\ol H_v^{(n)}}
\ee
where, $H_v^{(n)}$ and $\ol H_v^{(n)}$ denote the heights of the $v$-th tower in $B$ and $\ol B$, respectively, $v \in V_n$, $n \in \N_0$.
Then the extension of the measure $\widehat{\overline{\mu}}(\wh X_{\ov B})$ is finite if 
\be \label{eq s3 edge diagr ext}
\sum_{n=0}^\infty (a_{n+1} - b_n) < \infty.
\ee
\end{thm}

\begin{proof} Based on \Cref{edge_fin_krit}, it suffices to show that if \eqref{eq s3 edge diagr ext} holds, then the series
$$
J = \sum_{n=0}^\infty \sum_{v \in V_{n+1}} \sum_{w \in V_n} {f'}_{v,w}^{(n)} H_w^{(n)} \ov p_v^{(n+1)}
$$
converges. In the computation below, we use the fact that $\ol\mu$ is a probability measure on $\ol B$, that is, $\sum_{v \in V_{n}} \ov \mu (\ov X_v^{(n)}) = 1$ for $n \in \N_0$. We also use $F_n' = F_n - \ol F_n$ and relation \eqref{eq2: FH=H}.
\begin{align*}
\sum_{n=0}^\infty\sum_{v \in V_{n+1}} \sum_{w \in V_n} {f'}_{v,w}^{(n)} H_w^{(n)} \ov p_v^{(n+1)} =&\ \sum_{n=0}^\infty \sum_{v \in V_{n+1}} \sum_{w \in V_n} {f'}_{v,w}^{(n)} H_w^{(n)} \ov p_v^{(n+1)} \frac{\ov H_v^{(n+1)}}{\ov H_v^{(n+1)}}  \\
 =&\ \sum_{n=0}^\infty\sum_{v \in V_{n+1}} \sum_{w \in V_n} 
 {f'}_{v,w}^{(n)} H_w^{(n)} \ov \mu (\ov X_v^{(n+1)}) \frac{1}{\ov H_v^{(n+1)}}  \\
 \leq &\  \sum_{n=0}^\infty \sum_{v \in V_{n+1}} \ov \mu (\ov X_v^{(n+1)}) \sup_{v \in V_{n+1}} \left( \sum_{w \in V_n} {f'}_{v,w}^{(n)} \frac{H_w^{(n)}}{\ov H_v^{(n+1)}} \right)\\
  = & \ \sum_{n=0}^\infty \sup_{v \in V_{n+1}} \left( \sum_{w \in V_n} {f'}_{v,w}^{(n)} \frac{H_w^{(n)}}{\ov H_v^{(n+1)}} \right)\\
= &\ \sum_{n=0}^\infty \sup_{v \in V_{n+1}} \left( \sum_{w \in V_n}
(f_{v,w}^{(n)} - \ol f_{v,w}^{(n)}) \frac{H_w^{(n)}}{\ov H_v^{(n+1)}} \right) \\
= &\ \sum_{n=0}^\infty \sup_{v \in V_{n+1}} \left( \sum_{w \in V_n}
[f_{v,w}^{(n)}\frac{H_w^{(n)}}{ H_v^{(n+1)}}\frac{H_v^{(n+1)}}{\ol H_v^{(n+1)}} - \ol f_{v,w}^{(n)} \frac{\ol H_w^{(n)}}
{\ol H_v^{(n+1)}} \frac{H_w^{(n)}}{\ol H_w^{(n)}}]\right)\\
\leq  &\ \sum_{n=0}^\infty \sup_{v \in V_{n+1}} \left( \frac{H_v^{(n+1)}}{\ol H_v^{(n+1)}} - b_n\right)\\
\leq &\ \sum_{n=0}^\infty (a_{n+1} - b_n).
\end{align*}
Note that, for every $v \in V_{n+1}$, the summation over
$w$ ranges over a finite set. In other words, $f_{v,w}^{(n)}$ is nonzero for finitely many $w \in V_n$. 
In the final steps of the calculation, we used the definition of $a_n$, $b_n$ as in \eqref{eq s3 sup-inf}. 
Therefore, condition \eqref{eq s3 edge diagr ext} implies that
$\widehat{\overline{\mu}}(\wh X_{\ov B})$ is finite. 
\end{proof}

Now we provide an example of a generalized Bratteli diagram $B$ and its edge subdiagram $\ov B$ such that $\widehat{\overline{\mu}}(\widehat{X}_{\ol B}) = \infty$.

\begin{example}\label{ex: edge}
Let $B$ be a stationary generalized Bratteli diagram with the matrix $A$ transpose to the incidence matrix, where  
\begin{equation*}
A =  \left(
  \begin{array}{cccccccccccc}
   \ddots & \vdots & \vdots & \vdots & \vdots & \vdots & \vdots 
   & \vdots & \vdots & \vdots & \vdots & \udots\\
    \cdots & 2^4 & 0 & 2^3 & 0 & 0 & \textbf  0 & 0 & 0 & 0 & 0 & \cdots\\
    \cdots & 0 & 2^3 & 0 & 2^2 & 0 & \textbf 0 & 0 & 0 & 0 & 0 & \cdots\\
    \cdots & 0 & 0 & 2^2 & 0 & 2 & \textbf  0 & 0 & 0 & 0 & 0 & \cdots\\
    \cdots & 0 & 0 & 0 & 2 & 1 & \textbf 1 & 0 & 0 & 0 & 0 & \cdots\\
    \cdots & \textbf 0 & \textbf 0 &\textbf  0 &\textbf  0 &\textbf{1} & \textbf 1 & 
    \textbf{2} &\textbf 0 &\textbf 0 &\textbf 0 & \cdots\\
    \cdots & 0 & 0 & 0 & 0 & 0 & \textbf  2 & 0 & 2^2 & 0 & 0 & \cdots\\
    \cdots & 0 & 0 & 0 & 0 & 0 &\textbf  0 & 2^2 & 0 & 2^3 & 0 & \cdots\\
    \cdots & 0 & 0 & 0 & 0 & 0 &\textbf  0 & 0 & 2^3 & 0 & 2^4 & \cdots\\
    \udots & \vdots & \vdots & \vdots & \vdots & \vdots & \vdots & 
    \vdots & \vdots & \vdots & \vdots & \ddots\\
    \end{array}
\right)
\end{equation*} 
(we use the bold font to indicate the 0-th row and 0-th column in the matrix.)

Let $\ov B$ be a stationary edge subdiagram of $B$ defined by the matrix $\ov A = {\ov F}^{T}$ transpose to the incidence matrix: 
\begin{equation*}
\ov A =  \left(
  \begin{array}{cccccccccccc}
   \ddots & \vdots & \vdots & \vdots & \vdots & \vdots & \vdots 
   & \vdots & \vdots & \vdots & \vdots & \udots\\
    \cdots & 2 & 0 & 1 & 0 & 0 & \textbf  0 & 0 & 0 & 0 & 0 & \cdots\\
    \cdots & 0 & 2 & 0 & 1 & 0 & \textbf 0 & 0 & 0 & 0 & 0 & \cdots\\
    \cdots & 0 & 0 & 2 & 0 & 1 & \textbf  0 & 0 & 0 & 0 & 0 & \cdots\\
    \cdots & 0 & 0 & 0 & 2 & 1 & \textbf 1 & 0 & 0 & 0 & 0 & \cdots\\
    \cdots & \textbf 0 & \textbf 0 &\textbf  0 &\textbf  0 &\textbf{1} & \textbf 1 & 
    \textbf{2} &\textbf 0 &\textbf 0 &\textbf 0 & \cdots\\
    \cdots & 0 & 0 & 0 & 0 & 0 & \textbf  1 & 0 & 2 & 0 & 0 & \cdots\\
    \cdots & 0 & 0 & 0 & 0 & 0 &\textbf  0 & 1 & 0 & 2 & 0 & \cdots\\
    \cdots & 0 & 0 & 0 & 0 & 0 &\textbf  0 & 0 & 1 & 0 & 2 & \cdots\\
    \udots & \vdots & \vdots & \vdots & \vdots & \vdots & \vdots & 
    \vdots & \vdots & \vdots & \vdots & \ddots\\
    \end{array}
\right)
\end{equation*} 
Then the subdiagram $\ov B$ supports a finite tail-invariant measure 
$\ov \mu$ generated by the eigenvalue $\lambda = 3$ and the 
corresponding right eigenvector $\xi$ of $\ov A$ given by
$$
\xi = \left(\ldots, \dfrac{1}{2^3}, 
\dfrac{1}{2^2}, \dfrac{1}{2},1, \mathbf{1},
\dfrac{1}{2}, \dfrac{1}{2^2}, 
\dfrac{1}{2^3}, \ldots \right)^T
$$ 
The 0-th entry of $\xi$ is shown in bold font. 
It is easy to see that 
$$
\widehat{\overline{\mu}}(\widehat{X}_{\ol B}^{(1)}) = \sum_{i \in \mathbb{Z}} \frac{1}{2^i} \cdot 2^i = \infty.
$$
This example shows that $\widehat{\overline{\mu}}(\widehat{X}_{\ol B}) = \infty$ because the extension from the first level is already infinite. We refer our readers to \cite[Example 7.7]{BezuglyiJorgensenKarpelSanadhya2025} for a simple generalization of \Cref{ex: edge}.
\end{example}

It is easy to observe that instead of removing edges from a generalized Bratteli diagram to obtain an edge subdiagram (as described above), we can also add edges to obtain an ``edge extension" of the diagram. In particular,
\begin{definition}
Let $B' = (V, E')$ be a generalized Bratteli diagram and $B = (V, E)$ be a diagram obtained from $B'$ by adding edges to every level in $E'$. Then we call the generalized Bratteli diagram $B$ an \textit{edge extension} of $B'$. 
\end{definition}

\Cref{Thm sect 3 adding edges} below concerns edge extensions of generalized Bratteli diagrams of bounded size (see \Cref{Def:BD_bdd_size}). The proof uses the result about the structure of such diagrams, which is proved in \cite{BezuglyiJorgensenKarpelSanadhya2025}.
For the reader's convenience, we state this result.

\begin{prop}\cite[Corollary 3.6]{BezuglyiJorgensenKarpelSanadhya2025}\label{upper_cone}
Let $B=(V, E)$ be a generalized Bratteli diagram of bounded size defined by a sequence $(t_n, L_n)_{n \in \N_0}$ (see \Cref{Def:BD_bdd_size}). 
Let $n,m \in \mathbb{N}_0$, such that $m\leq n$, let $v \in V_{n+1}$ and $E(V_m, v)$ be the set of all finite paths 
$\ov e = (e_m, \ldots, e_{n})$ such that $r(e_n) = v \in V_{n+1}$ and $s(e_m) \in V_m$, then
$$
s(E(V_m, v)) \subset \left\{v - \sum_{i = m}^{n} t_i, \ldots, v + 
\sum_{i = m}^n t_i\right\}.
$$
\end{prop}

We will call the set of vertices $C_v^{(n+1)} = \{s(E(V_m, v)): m \leq n\}$ the \textit{upper cone} corresponding to the vertex $v \in V_{n+1}$. Note that the union of the cylinder sets that correspond to the finite paths from $E(V_0, v)$ forms the tower $X_v^{(n+1)}$, and the upper cone $C_v^{(n+1)}$ consists of all vertices from the levels $0, \ldots, n$, which are connected to $v \in V_{n+1}$ by finite paths.

\begin{theorem}\label{Thm sect 3 adding edges}
Let $\ol B = (V,\ol E)$ be a generalized Bratteli diagram of bounded size defined by a sequence $(t_n, L_n)_{n \in \N_0}$ (see \Cref{Def:BD_bdd_size}). Let $\ov \mu$ be a probability tail invariant measure on $\ov B$. Then there exists an edge extension $B$ of $\ol B$ obtained by adding countably many edges to $\ol E_n$ for every level $n \in \N_0$, such that the measure extension $\wh {\ov \mu}$ to $X_B$ is finite. 
\end{theorem} 

\begin{proof} Identify the set of vertices $V_n$ of $B$ at each level $n \in \N_0$ with integers. We first describe the procedure of adding edges to the diagram $\ol B$ to obtain $B$. At the first step of the construction, we add exactly one edge to $\ol E_0$. At the second step, we add one edge each to $\ol E_0$ and $\ol E_1$. Proceeding in this fashion at the $n$-th step, we add one edge to each $\ol E_0,\ldots, \ol E_{n-1}$.

Recall that for bounded size Bratteli diagrams, the set of vertices that can be reached from a given vertex and lie above this vertex forms the upper cone (see \Cref{upper_cone} above). We add the edges to the diagram $\ov B$ so that the upper cones corresponding to the vertices that form the ranges of the added edges, are pairwise disjoint. Now we estimate the measure of the added infinite paths using the measures of the towers of the diagram $\ov B$ corresponding to the range of the added edges. By choosing the range of the added edges to be a large enough number, we ensure that the measures of the corresponding towers are small and form a convergent series. 
We will choose additional edges so that the measure extension of the path space on step $n$ is bounded from above as follows:
$$
\wh{\ol\mu}(\wh X^{(n)}_{\ol B}) \leq 
\sum_{k = 0}^{l_n} \frac{1}{2^k},
$$
where for $n \in \N$, $l_n = 1 + 2 +\ldots + n$. In this process, overall, we will add infinitely many edges to each level $E_n$ of the diagram, and the measure extension to the whole path space $\wh X_{\ov B}$ will be finite. Now we discuss how the bound mentioned above could be achieved. Since $\ov \mu$ is a  probability measure, we have
$$
\sum_{w \in V_n} \ov \mu (\ov X^{(n)}_w) =1, \qquad n \in \N_0.
$$
Thus, for every $ n \in \N_0$ and every $k\in \mathbb{N}$ 
there exists $m^{(n+1)}_k\in V_{n+1}$ such that
for every $w \in V_{n+1}$ with $w \geq m^{(n+1)}_k$ one has
$$
\ov \mu (\ov X^{(n+1)}_w) < \frac{1}{2^k}.
$$

In the first step, pick any vertex $w^{(1)}_1 \in V_1$ such that $w^{(1)}_1 \geq m_1^{(1)}$, and
add an edge to $\ov E_0$  such that this edge connects the vertex $w^{(1)}_1$ at level $V_1$ to a vertex $u^{(0)}_1$ at level $V_0$. Choose $u^{(0)}_1$ such that in the initial diagram we had $E(w^{(1)}_1, u^{(0)}_1) \neq \emptyset$.
We will not add any more edges in this step. 
Since $E(w^{(1)}_1, u^{(0)}_1) \neq \emptyset$, the measure of the paths that we have added is not greater than $\ov \mu(\ov X^{(1)}_{w^{(1)}_1} \cap \ov X^{(0)}_{u^{(0)}_1})$ and hence is not greater than $\ov \mu(\ov X^{(1)}_{w^{(1)}_1})$.
Thus, after the first step, the measure extension $\wh {\ov \mu}$ to the path space $\wh X^{(1)}_{\ov B}$ can be estimated from the above by
$$
1 + \ov \mu(\ov X^{(1)}_{w^{(1)}_1}) 
< 1 + \frac{1}{2}. 
$$

In the second step, pick a vertex $w_2^{(1)} \in V_1$ 
such that $w_2^{(1)} > \max\{w_1^{(1)}, m_2^{(1)}\}$. Similar to the first step, add an edge to $\ov E_0$ such that this edge connects the vertex $w^{(1)}_2$ at level $V_1$ to a vertex $u^{(0)}_2$ at level $V_0$. Choose $u^{(0)}_2$ so that, in the initial diagram, there is already an edge between $w^{(1)}_2$ and $u^{(0)}_2$. Also, we pick a vertex $w_3^{(2)} \in V_2$ 
such that $w_3^{(2)} > \max\{m_3^{(2)}, w_2^{(1)} + t_1\}$. This will ensure that
$$
\ov \mu (\ov X^{(2)}_{w_3^{(2)}}) < \frac{1}{2^3}
$$
and the ``upper cone'' that corresponds to $w_3^{(2)}$ does not contain previously chosen vertices $w_1^{(1)}$, $w_2^{(1)}$. 
Similarly as before, add an edge to $\ov E_1$ that connects the vertex $w^{(2)}_3$ at level $V_2$ to a vertex $u^{(1)}_3$ at level $V_1$. Choose $u^{(1)}_3$ such that in the initial diagram we had $E(w^{(2)}_3, u^{(1)}_3) \neq \emptyset$.
Thus, after the second step, the measure extension to the new path space will be estimated from the above by
$$
1 + \ov \mu(\ov X^{(1)}_{w^{(1)}_1}) + \ov \mu(\ov X^{(1)}_{w^{(1)}_2}) + \ov \mu(\ov X^{(2)}_{w^{(2)}_3}) 
< 1 + \frac{1}{2} + \frac{1}{2^2} + \frac{1}{2^3}. 
$$ Proceed similarly by choosing vertices $w^{(n)}_l$ and adding edges $e \in \ov E_{n-1}$ connecting $w^{(n)}_l \in V_n$ and $u^{(n-1)}_l \in V_{n-1}$ such that there already existed an edge between $w^{(n)}_l$ and $u^{(n-1)}_l$, and the following conditions hold: 
$$
w^{(n)}_l > \max\{m^{(n)}_l, w^{(n-1)}_{l-1} + t_{n-1}\}$$
for $n > 1$, and 
$$
w^{(1)}_{l_n + 1} > \max\{m^{(1)}_n, w^{(n)}_{l_{n}} + \sum_{k = 1}^{n-1}t_{k}\}.
$$ Then, the upper cones corresponding to all vertices $w_l^{(n)}$ are pairwise disjoint and
$$
\wh{\ol \mu} (\wh X_{\ol B}) < \sum_{k = 0}^{\infty} \frac{1}{2^k} < \infty. 
$$
\end{proof}


\section{Measure extension for certain classes of generalized Bratteli diagrams}
\label{sect ext from classes sbdgrms} 
In this section, we apply the results proved in \Cref{sect meas ext subdgrms} to some classes of generalized Bratteli diagrams. We will consider measure extensions from a variant of odometer (which we call a \textit{fat odometer}, see \Cref{sec: fat odometer}) and subdiagrams $\ol B$ defined by finite 
subsets $\ol W_n$. 

\subsection{Measure extension from simple subdiagrams} 

Recall that a standard Bratteli diagram $(V,E)$ is called \textit{simple} if there exists a telescoping $(\widetilde{V},\widetilde{E})$ of
$(V,E)$ so that the incidence matrices of $(\widetilde{V},\widetilde{E})$ have non-zero entries at each level. 

Let $B = B(F_n)$ be a generalized Bratteli diagram. In this section, we will impose some additional conditions on the sequence matrices $F_n$ (in particular, we will assume that these matrices have equal row sums). Let $\ov B = B(W_n)$ be a vertex subdiagram of $B$ where for each $n \in \N_0$, $W_n \subset V_n$ 
and $W_n$ is a finite set. In other words, $B(W_n)$ denotes the subdiagram of $B$ supported by vertices $W_n \subset V_n$ for $n \in \N_0$. Suppose that $\ol B$ is a simple standard Bratteli diagram. 
Hence, without loss of generality, we can assume that $E(w, v) \neq \emptyset$ for all $w \in W_n$ and $v \in W_{n+1}$, $n \in N_0$. Let $\nu$ be a probability tail invariant measure on $X_{\ov B}$. For $n \in \N_0$, let $\ov p_v^{(n+1)} = \nu([\ov e])$ for a cylinder set $[\ov e] \subset X_{\ov B}$ with $r(\ov e) = v \in W_{n+1}$. For $n \in \N_0$, we denote by $\ov F_n$ the restriction of $F_n$ onto $\ov B$. Hence, the sequence $(\ol F_n)_{n \in \N_0}$ denotes the incidence matrices for $\ol B$.
This means that $\ol f^{(n)}_{v w} = f^{(n)}_{v w}$ if and only if
$v \in \ol W_{n+1}$ and $w \in \ol W_n$, $n \in \N_0$.
We set
$$
\ov M_n = \max\{f_{v,w}^{(n)} \;| \; v \in W_{n+1}, w \in W_n\}
$$
and 
$$
\ov m_n = \min\{f_{v,w}^{(n)} \; |\;  v \in W_{n+1}, w \in W_n\},
$$
where $n \in \N_0$. 
We are interested in the following \textit{problem:} Find conditions on $B$ and $\ov B$ under which every probability tail invariant measure $\nu$ on $X_{\ov B}$ has a finite extension $\wh \nu$ to the set $\wh X_{\ov B}$.

For $n \in \N_0$, let $\ov H_w^{(n)}$ denote the height of the tower $\ov X_w^{(n)}$, $w \in W_n$, in the subdiagram $\ov B$ (see \eqref{formula for heights}).

\begin{lemma}\label{claim:bd_heighs}  Let $\ov m_n$, $\ov M_n$, and $W_n$ be as above. Also, for $v \in W_n$, let $\ov H_v^{(n)}$ denote the height of the $v$-tower in $\ol B$. Then
    $$
    \prod_{i = 0}^n \ov m_i |W_i| \leq \ov H_v^{(n+1)} \leq \prod_{i = 0}^n \ov M_i |W_i|,\quad n\in \N_0.
    $$    
\end{lemma}

\begin{proof}
Recall that $\ov H_w^{(0)} = 1$ for all $w \in W_0$ and for $n \in \N$, we have $\ov F_{n-1} \ \ \cdots \ov F_0\  \ov H^{(0)} = \ov H^{(n)}$. Since $\ov B$ is simple, we have, for all $v \in W_1$, 
    $$
    \ov m_0 |W_0| \leq \ov H_v^{(1)} = \sum_{w \in W_0} f_{v,w}^{(0)} \leq \ov M_0 |W_0|.
    $$
    It is easy to see that by induction on $n \in \N_0$, we have
    $$
    \ov H_v^{(n+1)} = \sum_{w \in W_n} f_{v,w}^{(n)}\ov H_w^{(n)} \leq \prod_{i = 0}^{n-1}\ov M_i|W_i|\sum_{w \in W_n} f_{v,w}^{(n)} \leq \prod_{i = 0}^n \ov M_i |W_i|.
    $$
    Similarly,
    $$
    \ov H_v^{(n+1)} \geq \prod_{i = 0}^n \ov m_i |W_i|.
    $$
\end{proof}

\begin{lemma}\label{claim:bd_meas_cyl}
For $\ov m_n$, $\ov M_n$, $W_n$, $\ov H_w^{(n)}$ and $\ov p_w^{(n)}$ as above, we have  
    \begin{equation}\label{eq:meas_cyl}
    \frac{1}{\prod_{i = 0}^n\ov M_i|W_i|} \leq \sum_{v \in W_{n+1}}\ov p_v^{(n+1)} \leq \frac{1}{\prod_{i = 0}^n\ov m_i|W_i|}, \quad n\in \N_0.
\end{equation}
    \end{lemma}

\begin{proof} Since $\nu(X_{\ov B}) = 1$, from \eqref{eq: p_v_n_bar} it follows
$$
\sum_{v \in W_{n+1}} \ov p_v^{(n+1)} \ov H_v^{(n+1)} = 1.
$$ Thus \eqref{eq:meas_cyl} follow from the equality above and \Cref{claim:bd_heighs}.
\end{proof}
Recall our main formula for finding the value of the measure extension \eqref{eq: exten_representation}: 
\begin{equation}\label{eq:meas_ext}
    \wh \nu(\wh X_{\ov B}) = 1 + \sum_{n = 0}^{\infty} \sum_{v \in W_{n+1}}\sum_{w \in W_n'} f_{v,w}^{(n)} \ov p_v^{(n+1)}H_w^{(n)},
\end{equation}
where $W_n' = V_n \setminus W_n$, and $\ov p_v^{(n+1)} = \nu([\ov e])$ for a cylinder set $[\ov e] \subset X_{\ov B}$ with $r(\ov e) = v \in W_{n+1}$.
Denote for $n \in \N_0$,
\be\label{eq s4 M'_n}
M'_n = \max\{f_{v,w}^{(n)} \;| \; v \in W_{n+1}, w \in W'_n\}
\ee
and 
\be\label{eq s4 m'_n}
m_n' = \min\{f_{v,w}^{(n)} \; |\;  v \in W_{n+1}, w \in W'_n\}.
\ee

\begin{remark}
Formally, for $n \in \N_0$, the set $W_n'$ is infinite, but in the measure extension procedure, we are interested only in the finite set of vertices 
$$
W_n'' = \bigcup_{v \in W_{n+1}}\{w \in W_n' \; | \;f_{v,w}^{(n)} > 0\}
$$ 
because $W_{n+1}$ is finite and $|r^{-1}(v)| < \infty$ for every $v \in W_{n+1}$. To simplify the notation, we continue using the notation $W_n'$ below, which means the set of vertices $W_n''$. For $n \in \N_0$, we also use the notation $\ol M_n$, $\ol m_n$, $M'_n$, and $m_n'$ introduced above. Recall that the term ERS refers to the property of \textit{equal row sums} (see \Cref{Sec: Basics}).
\end{remark}

\begin{theorem}\label{thm meas ext finite sbdgrm}
Suppose that the generalized Bratteli diagram $B= B(F_n)$ has the ERS property with parameters $(r_n)_{n=0}^{\infty}$. In other words, for all $n \in \N_0$,
\[\sum_{w \in V_n} f_{v,w}^{(n)} = r_n, \quad \forall v \in V_{n+1}.\]
Let $\ol B = B(W_n)$ be a simple standard subdiagram as above. If the set $\left\{\frac{M'_n}{m'_n} : n \in \N\right\}$ is bounded, and 
$$
\lim_{n\to\infty}\alpha_n : = \lim_{n\to\infty} \prod_{i = 0}^n \frac{r_i}
{\ov m_i |W_i|}
$$
exists, then for every probability tail invariant measure  $\nu $ the 
measure extension $\wh \nu(\wh X_{\ol B})$ is finite.
\end{theorem}

\begin{proof} Recall that $\ov H_w^{(0)} = 1$ for all $w \in W_0$. Using the fact that for $n \in \N$, we have $\ov F_{n-1} \ \ \cdots \ov F_0\  \ov H^{(0)} = \ov H^{(n)}$ and the ERS property we get that, for all $n\in \N$ and all $w \in V_n$, $H_w^{(n)} = r_0 \ \cdots \ r_{n-1}$. It follows from the definition of $M'_n$ and \eqref{eq:meas_ext} that
\begin{align*}
    \wh \nu(\wh X_{\ov B}) \leq &\ 1 + \sum_{n = 0}^{\infty} \sum_{v \in W_{n+1}}\sum_{w \in W_n'} M_n'\;\ov p_v^{(n+1)}(r_0\ \cdots\ r_{n-1})\\
= &\ 1 + \sum_{n = 0}^{\infty} M_n'\; (r_0\ \cdots\ r_{n-1})|W_n'|\sum_{v \in W_{n+1}} \ov p_v^{(n+1)},
\end{align*}
where we set $r_{-1} = 1.$
Thus, from \Cref{claim:bd_meas_cyl} it follows
\[
\wh \nu(\wh X_{\ov B}) \leq 1 + \sum_{n = 0}^{\infty} M_n'\; (r_0\ \cdots\ r_{n-1})|W_n'| \frac{1}{\prod_{i = 0}^n\ov m_i|W_i|}
\] 
Note that, for $n \in \N_0$ and every $v \in V_{n+1}$, we have
$$
r_n = \sum_{w \in V_n} f_{v,w}^{(n)} = \sum_{w \in W_n} f_{v,w}^{(n)} + \sum_{w \in W_n'} f_{v,w}^{(n)} \geq \ov m_n |W_n| + m_n'|W_n'|.
$$
Thus
$$
|W_n'| \leq \frac{r_n - \ov m_n|W_n|}{m_n'}. 
$$
Hence, we have
$$
\ba
\wh \nu(\wh X_{\ov B}) \leq &\ 1 +  \sum_{n = 0}^{\infty} \frac{M_n'}{m_n'}(r_0 \cdots r_{n-1})(r_n - \ov m_n |W_n|)\frac{1}{\prod_{i = 0}^n\ov m_i |W_i|}\\
=&\ 1 + \sum_{n = 0}^{\infty} \left(\frac{M_n'}{m_n'}\frac{r_0 \cdots r_{n-1}r_n}{\prod_{i = 0}^n\ov m_i |W_i|} - \frac{M_n'}{m_n'}\frac{r_0 \cdots r_{n-1}}{\prod_{i = 0}^{n-1}\ov m_i |W_i|}\right) \\ 
= & \ 1 + \sum_{n = 0}^{\infty}\frac{M_n'}{m_n'} (\alpha_n - 
\alpha_{n-1}),
\ea
$$ 
where $\alpha_{-1} =1$.

By the assumption of the theorem, we deduce from the last inequality that 
$\wh \nu(\wh X_{\ov B}) < \infty$.
\end{proof}

The proof of \Cref{thm meas ext finite sbdgrm} is essentially based on an application of \Cref{claim:bd_heighs}, where we bound the heights of the subtowers using the equal row sum property.

Next, we discuss the necessary conditions for the finiteness of the measure extension from the subdiagram $\ol B = B(W_n)$. We use the notation introduced in \Cref{thm meas ext finite sbdgrm}. 

\begin{theorem}\label{thm ness cond finite sbdgrm} 
Let $B$ and $\ol B$ be as in \Cref{thm meas ext finite sbdgrm}.
Let $\nu$ be a probability tail invariant measure on $X_{\ol B}$.
If the extension of $\nu$ to $\wh X_{\ol B}$ is finite and the set 
$\left\{\frac{m'_n}{M'_n} : n \in \N\right\}$ is bounded from below by a positive constant $\delta$, then the limit 
$$
\lim_{n\to \infty} \beta_n := \lim_{n\to \infty} 
  \prod_{i = 0}^n \frac{r_i}{\ov M_i |W_i|}
$$
exists. 
\end{theorem}

\begin{proof}
We apply again the method we used in the proof of \Cref{thm meas ext finite sbdgrm}. It follows from the definition of $m'_n$ and from \eqref{eq:meas_ext} that 
\begin{align*}
    \wh \nu(\wh X_{\ov B}) = &\ 
1 + \sum_{n = 0}^{\infty} \sum_{v \in V_{n+1}}\sum_{w \in W_n'}f_{v,w}^{(n)}\ov p_v^{(n+1)}H_w^{(n)}\\
\geq &\ 1 + \sum_{n = 0}^{\infty} m_n'(r_0 \ \cdots \ r_{n-1})|W_n'|\sum_{v \in V_{n+1}}\ov p_v^{(n+1)}
\end{align*} 
with $r_{-1} =1$. By Lemma \ref{claim:bd_meas_cyl}, it follows
\[
\wh \nu(\wh X_{\ov B}) \geq 1 + \sum_{n = 0}^{\infty} m_n'(r_0 \ \cdots\ r_{n-1})|W_n'|\frac{1}{\prod_{i = 0}^n\ov M_i |W_i|}.
\]
Since 
$$
r_n \leq \ov M_n |W_n| + M_n' |W_n'|.
$$
we have 
$$
|W_n'| \geq \frac{r_n - \ov M_n |W_n|}{M_n'}.
$$
Thus, 
$$
\ba
\wh \nu(\wh X_{\ov B}) \geq &\ 1 + \sum_{n = 0}^{\infty} \frac{m_n'}{M_n'}(r_0\ \cdots \ r_{n-1})\frac{r_n - \ov M_n |W_n|}{\prod_{i = 0}^n\ov M_i |W_i|}\\
= &\ 1 + \sum_{n = 0}^{\infty} \frac{m_n'}{M_n'} \left( \prod_{i = 0}^n \frac{r_i}{\ov M_i |W_i|}  - \prod_{i = 0}^{n-1} \frac{r_i}{\ov M_i |W_i|} \right)\\ 
\geq  &\ 1 + \delta \sum_{n = 0}^{\infty} (\beta_{n} - \beta_{n-1}),
\ea
$$ 
where $\beta_{-1} =1$. This proves the theorem. 
\end{proof}

\begin{remark}
It follows from the proofs of \Cref{thm meas ext finite sbdgrm} and \Cref{thm ness cond finite sbdgrm} that if we assume that the sequence $\left(\frac{M'_n}{m'_n}\right)$ is 
increasing, then the existence of the limit 
$$
\lim_{n\to \infty}\frac{M'_n}{m'_n} \prod_{i = 0}^n \frac{r_i}
{\ov m_i |W_i|}
$$
implies that $\wh \nu(\wh X_{\ov B}) < \infty$, and if 
$\wh \nu(\wh X_{\ov B}) < \infty$ then the limit 
$$
\lim_{n\to \infty} 
\frac{m_n'}{M_n'}  \prod_{i = 0}^n \frac{r_i}{\ov M_i |W_i|}
$$
exists.

\end{remark}

\begin{example}\label{example1}
Let $f_{v,w}^{(n)} = 1$ for all $v \in W_{n+1}$,  $w \in W_n \cup W_{n}'$ and $n \in \mathbb{N}_0$. Then the condition that the sequence of matrices $(F_n)_{n=0}^{\infty}$ have equal row sum (ERS) with parameter $(r_n)_{n=0}^{\infty}$ implies that for $n \in \N_0$,
$$
|W_n| + |W_n'| = r_n.
$$
Under this assumption, we obtain that for $n \in \N_0$,
    $$
    M_n' = m_n' = \ov M_n = \ov m_n = 1,   
    $$
    for all $n$ and therefore 
    $$
    \alpha_n = \beta_n = 
    \frac{r_0\ \cdots \ r_n}{|W_0|\ \cdots\ |W_n|} = \left(1 + \frac{|W_0'|}{|W_0|}\right)\ \cdots\ \left(1 + \frac{|W_n'|}{|W_n|}\right).
    $$
The sequence $\alpha_n$ converges, for instance, if for all $n \in \N_0$ there exists $C > 0$ such that $|W_n'| \leq C$ and $|W_n| = n^2$. For the subdiagram $\ol B$ of $B$ in this example, we conclude that 
$$
\wh\nu(\wh X_{\ol B}) < \infty \ \ \Longleftrightarrow \ \ 
\prod_{i=0}^\infty \left(1 + \frac{|W_i'|}{|W_i|}\right) < \infty.
$$
\end{example}

\begin{example}\label{Ex_hor_stat_IO}
In this example, we consider a horizontally stationary (see \Cref{Def: Horz}) Bratteli diagram. Recall that in this case we identify $V_n$ with $\Z$ for every level $n$.
Let $B = B(F_n)$ where the sequence of incidence matrices $(F_n)_{n=0}^{\infty}$ is given by
$$
F_n = \left(
\begin{array}{cccccccccccc}
        \ddots & \vdots & \vdots & \vdots & \vdots & \vdots & \udots\\
\cdots & a_n & 1 & 0 & 0 & 0 & \cdots\\
\cdots & 1 & a_n & 1 & 0 & 0 & \cdots\\
\cdots & 0& 1 & a_n & 1 & 0 & \cdots\\
\cdots & 0 & 0 &  1 & a_n & 1 & \cdots\\
\cdots & 0 & 0 & 0 &  1 & a_n & \cdots\\
        \udots & \vdots & \vdots & \vdots & \vdots & \vdots & \ddots\\
\end{array}
\right).
$$
In other words, we have $F_n = (f_{v,w}^{(n)})$, where
$$
f_{v,w}^{(n)} =
\left\{
\begin{aligned}
&a_n, \quad v = w,\\
&1, \quad |v - w| = 1,\\
&0, \quad \text{otherwise}.
\end{aligned}
\right.
$$
Let $a_n \geq 1$ for all $n$. Let $\ov B = B(W_n)$ be the vertical odometer defined by a vertex $w \in V_0$, i.e. $W_n = \{w\}$ for all $n$. Without loss of generality, we can assume that $W_n = \{w\} = \{0\}, \ n \in \N_0$. Let $\nu$ be the unique invariant probability measure on $\ov B$. Then
$$
\nu([\ov e]) = \frac{1}{a_0 \cdots a_{n-1}} = \ov p_0^{(n)}, 
$$ 
where $\ov e$ is a finite path with $s(\ov e) = 0 \in V_0$ and  $r(\ov e) = 0 \in V_n$. Observe that the sequence of matrices $(F_n)_{n=0}^{\infty}$ have equal row sum (ERS) with parameter $r_n = a_n + 2$ for $n \in \N_0$. We also have $m_n' = M_n' = 1$ for all $n \in N_0$, and 
$$
\alpha_n = \beta_n = \frac{(a_0 + 2)\cdots(a_n + 2)}{a_0 \cdots a_n} = \prod_{i = 0}^n\left(1 + \frac{2}{a_i}\right)
$$
since $\ov m_n = \ov M_n = a_n$ and $|W_n| = 1$ for all $n \in \N_0$. Then
$$
\wh \nu(\wh X_{\ov B}) < \infty \Longleftrightarrow \sum_{n = 0}^{\infty} \frac{1}{a_n} < \infty.
$$
It was proved in \cite[Theorem 4.5]{BezuglyiJorgensenKarpelKwiatkowski2025} that if $\sum_{n = 0}^{\infty} {a_n^{-1}} < \infty$ then $B$ has countably infinitely many ergodic invariant probability measures and they all can be obtained as extensions from vertical odometers.
\end{example}

\begin{remark}[Measure extension for a bounded size Bratteli diagram]
The method used in the proof of \Cref{thm meas ext finite sbdgrm} and \Cref{thm ness cond finite sbdgrm} can be applied to 
a generalized Bratteli diagram of bounded size. Indeed, let $B= B(F_n)$ be a bounded size generalized Bratteli diagram with parameters $(t_n, L_n)$. Suppose that $\ol B= \ol B(W_n)$ is a vertex subdiagram defined by a sequence $\ol W = (W_n)$, where 
$W_n  = [w_n^-, w_n^+ ] \subset V_n$. We choose $W_n$ such that 
$W_0 = \{0\}$, $W_1 \subset [-t_0, t_0]$, and in general, 
$$
- (t_0 + \cdots + t_{n-1}) \leq w_n^- < w_n^+ \leq  t_0 + \cdots + t_{n-1}
$$ This case is more general than the ERS property, but we can still use the fact that $H_w^{(n)} < L_0\ \cdots\ L_{n-1}$. 
Suppose that $\nu$ is a probability measure on $X_{\ol B}$. Assuming that $\ol B$ is simple, one can prove, similar to the proofs of \Cref{thm meas ext finite sbdgrm} and \Cref{thm ness cond finite sbdgrm}, that the measure extension $\wh\nu(\wh X_{\ol B})$ is finite if
the sequence $(M'_n {m'_n}^{-1})$ is bounded and 
$$
\lim_{n\to \infty} \prod_{i=0}^n \frac{L_i}{\ol m_i |W_i|} < \infty.
$$

\end{remark}

\subsection{Measure extension from stationary subdiagrams}
Let $B = B(F_n)$ be a generalized Bratteli diagram defined by the sequence of incidence matrices $(F_n)_{n=0}^{\infty}$. Suppose that $\ov B =  B(\ol F)$ is a stationary vertex subdiagram of $B$ defined by a matrix $\ov F$. In other words, there exists a fixed set of vertices $W \subset V_n$ for every $n \in \N_0$, such that the matrix $\ol F$ is the restriction of the matrices $F_n$ to $W$. In general, we do not assume that the Bratteli diagram $B$ is stationary, but we require that the restriction of every $F_n$ to $W$ is the fixed matrix $\ol F$. We further assume that for the matrix $\ov F$ there exists an eigenpair $(\lambda, \xi)$ such that $1 <\lambda < \infty$, $\sum_{v\in W} \xi_v =1$ (note that if $\sum_{v\in W} \xi_v < \infty$, then we can always normalize $\xi$ so that the sum is $1$), and $\ol F^T\xi = \lambda \xi$.
This allows us to define the probability tail invariant measure 
$\nu$ on the path space $X_{\ov B}$ of the subdiagram $\ov B$ by  
$\lambda$ and $\xi$ (see 
\Cref{thm PFThm}). Note that the value of $\nu$ on a cylinder set $[\ov e]$ such that $r(\ol e)= w \in V_n$ is 
\be\label{eq s4 Perron measure}
\ov p_w^{(n)} = \nu([\ov e] \cap X_{\ov B}) = 
\frac{\xi_{r(w)}}{ \lambda^n}.
\ee

Let $\wh X_{\ov B} = \mathcal{R}(X_{\ov B})$ be the saturation of $X_{\ov B}$ with respect to the tail equivalence relation $\mathcal{R}$, and $\wh{\nu}$ be the 
tail invariant extension of the probability measure $ \nu$ to $\wh X_{B}$. 
We will discuss the following question: For what diagrams $B$ can the measure extension from a stationary subdiagram $\ol B$ be finite? 

We denote as usual for $n \in \N_0$
$W'_n = V_n \setminus W$ (since $W_n = W$ for $n \in \N_0$) and set
$$
\alpha_n := \sum_{w \in W'_n} H_w^{(n)}.
$$
In this setting, using \eqref{eq: exten_representation} and the fact that $W_{n+1} = W$, we express 
\be\label{eq s4 ext stat sbdgrm}
\ba
\wh \nu(\wh X_{\ov B}) = &\  1 + \sum_{n = 0}^{\infty} \sum_{v \in W} \sum_{w\in W'_n} f^{(n)}_{v,w} H_w^{(n)} \ol p^{(n+1)}_v \\
= &\ 1 + \sum_{n = 0}^{\infty} \frac{1}{\lambda^{n+1}} \sum_{w\in W'_n}  H_w^{(n)} \sum_{v \in W} f^{(n)}_{v,w} \xi_v  \\ 
\leq  &\ 1 + \sum_{n = 0}^{\infty} \frac{\alpha_n M'_n}{\lambda^{n+1}}
\sum_{v \in W} \xi_v\\
= &\ 1 + \sum_{n = 0}^{\infty} \frac{\alpha_n M'_n}{\lambda^{n+1}}
\ea
\ee
where $M'_n$ is defined in \eqref{eq s4 M'_n}. We used here that the entries $f^{(n)}_{v,w}$ for $v \in W$ and $w \in W'_{n}$ are bounded and  $\sum_{v\in W} \xi_v =1$. The quantity $M'_n \alpha_n$ estimates the number of new finite paths added from the $n$-th level that play a role in the extension of the measure.

Similarly, we can obtain the following estimate of $\wh \nu(\wh X_{\ov B})$ from below:
\be\label{eq s4 ext stat sbdgrm m'n}
\wh \nu(\wh X_{\ov B}) \geq 1 + \sum_{n = 0}^{\infty} \frac{\alpha_n m'_n}{\lambda^{n+1}}.
\ee

It follows from \eqref{eq s4 ext stat sbdgrm} and \eqref{eq s4 ext stat sbdgrm m'n} that the following sufficient condition holds.

\begin{prop}\label{prop ext from stat BD}
Suppose that the objects $B, \ol B, \lambda,\nu, (M'_n)_{n=0}^{\infty}$, and $(\alpha_n)_{n=0}^{\infty}$ are as defined above. 
\begin{enumerate}
    \item If for every $\varepsilon >0$ there exists $N_\varepsilon\ \in \N$ such that, for all $n \geq N_\varepsilon$, 
$$
\alpha_{n+1} M'_{n+1} < (\lambda - \varepsilon)\alpha_n M'_n,
$$
then we have $\wh \nu(\wh X_{\ov B}) < \infty$.
\item If for every $\varepsilon >0$ there exists $N_\varepsilon\ \in \N$ such that, for all $n \geq N_\varepsilon$, 
$$
\alpha_{n+1} m'_{n+1} > (\lambda + \varepsilon)\alpha_n m'_n,
$$
then we have $\wh \nu(\wh X_{\ov B}) = \infty$.
\end{enumerate}
\end{prop}

\begin{remark}
A small modification of \eqref{eq s4 ext stat sbdgrm} gives a different estimate using \eqref{eq: exten_representation_1}.
$$
\ba 
\ol \nu(\wh X_{\ol B})  = &\ 1+   \sum_{n=0}^\infty \,\,\bigg( \sum_{v \in W_{n+1}}
H_v^{(n+1)}\overline p_v^{(n+1)} - 
 \sum_{w \in W_{n}} H_w^{(n)}\overline p_w^{(n)}\bigg)\\
= &\ 1+ \sum_{n = 0}^{\infty} \left( \sum_{v \in W_{n+1}} \left(\sum_{u \in V_n} f_{vu}^{(n)} H^{(n)}_u\right) \frac{\xi_v}{\lambda^{n+1}} - \sum_{w \in W_{n}} H^{(n)}_w \frac{\xi_w}{\lambda^{n}}\right)\\
= &\ 1+ \sum_{n = 0}^{\infty} \left( \sum_{u \in V_n}H^{(n)}_u \frac{\xi_v}{\lambda^{n+1}} \left(\sum_{v \in W_{n+1}} f_{vu}^{(n)}\right)  - \sum_{w \in W_{n}} H^{(n)}_w \frac{\xi_w}{\lambda^{n}}\right)\\
=&\ 1+ \sum_{n = 0}^{\infty} \left( \sum_{u \in V_n}H^{(n)}_u \frac{\xi_u}{\lambda^n} - \sum_{w \in W_{n}}
H^{(n)}_w \frac{\xi_w}{\lambda^{n}}\right)\\
= &\ 1+ \sum_{w \in W'_n}H^{(n)}_w \frac{\xi_w}{\lambda^{n}}\\
\leq &\ 1+ \sum_{w \in W'_n} \frac{\alpha_n}{\lambda^n}.
\ea
$$ Here we use that $\xi$ is a probability eigenvector for $\ol F$.
\end{remark} 

The following example illustrates Proposition \ref{prop ext from stat BD}.

\begin{example}
Let $B = B(F)$ where the incidence matrix $F$, indexed by $\N_0, $ has the form:
$$
F =  \begin{pmatrix}
    a & 1 &  0 & 0 & 0 & \cdots & \cdots \\
    0 & 1 & 1 & 0 & 0 & \cdots & \cdots\\
    0 & 0 & 1 & 1 & 0 & \cdots & \cdots  \\
\cdots & \cdots & \cdots & \cdots & \cdots & \cdots & \cdots  \\
\end{pmatrix}.
$$
Let $B_a$ denote the stationary odometer viewed as a subdiagram of $B$ sitting on the first vertex. We take the unique tail invariant measure $\nu$ on the odometer $\ol B_a$ such that $\nu ([\ol e]) = a^{-n}$ where $\ol e$ is a cylinder set
from $X_{\ol B_a}$ of length $n$. Then the measure extension is determined by the formula
$$
\wh \nu(\wh X_{\ol B_a}) = 1 + \sum_{n = 0}^{\infty} \frac{H_1^{(n)}}{a^{n+1}}
= 1 + \frac{1}{a} \sum_{n = 0}^{\infty} \frac{2^n}{a^{n}}, 
$$
where $H_1^{(n)} = 2^n$ is the height of the tower corresponding to the vertex $1 \in V_n$. 
Then $\wh \nu(\wh X_{\ol B_a}) < \infty$ iff $a \geq 3$. 
\end{example}

\begin{remark}
The statement in \Cref{prop ext from stat BD} depends on the rate of growth of $H^{(n)}_w, w \in W'_n$, as $n \to \infty$. In this connection, we recall the following result concerning infinite matrices from \cite{Kitchens1998}. 

\begin{theorem*}\cite[Theorem 7.1.3 $(f)$ and Lemma 7.1.19]{Kitchens1998} Let $A$ be an aperiodic and irreducible countably infinite matrix with nonnegative entries. Let $\lambda <\infty$ be the Perron eigenvalue. 
Denote by $a^{(n)}_{i,j}$ the entries of $A^n$.
Then 
$$
\lim_{n\to \infty} \dfrac{a^{(n)}_{i,j}}{\lambda^n} =0
$$ if $A$ is transient or null-recurrent.
\end{theorem*} 

This result can be used for the construction of other examples with 
finite measure extension for generalized stationary Bratteli diagrams. 

\end{remark}

In what follows, we assume that the generalized Bratteli diagrams with equal row sum (ERS) property (see \Cref{Sec: Basics}) have stationary Bratteli diagrams as subdiagrams. 

\begin{prop}\label{prop s4 stat diagr ext}
Suppose that the generalized Bratteli diagram $B= B(F_n)$ has the ERS property with parameters $(r_n)_{n=0}^{\infty}$. Let $\ol B = B (\ol F)$ be a subdiagram of $B$ defined by a matrix $\ol F$. We assume that $\ol F$ satisfies the conditions in \Cref{thm Perron-Frobenius Thm} and $(\lambda,\xi)$ are \textit{eigenpairs} for $\ol{F}^T $, in other words $\ol{F}^T \xi = \lambda \xi$. Let $\nu $ be the probability measure defined by $(\lambda, \xi) $ as in \eqref{eq s4 Perron measure}. Then the measure extension $\wh\nu(\wh X_{\ol B}) $ is infinite. 
\end{prop}

\begin{proof} As a consequence of the equal row sum property, we can assume, without loss of generality, that 
\be\label{eq W'_n finite}
|\{w \in V_n \setminus W_n : f^{(n)}_{v,w} \neq 0\ \mbox{for\ some} \ v \in V_{n+1}\}| < \infty, \quad n\in \N_0.
\ee
Hence, at each level $n \in \N_0$, the measure is extended to only finitely many vertices in $W'_n = V_n \setminus W_n$, although both $W_n$ and $W'_n$ could be infinite. Recall that by \eqref{eq: exten_representation}, we have
$$
\wh \nu(\wh X_{\ov B}) =  
1 + \sum_{n = 0}^{\infty} \sum_{v \in W_{n+1}}\sum_{w \in W_n'}f_{v,w}^{(n)}\ov p_v^{(n+1)}H_w^{(n)}
$$ 
Using \eqref{eq s4 Perron measure} and the ERS property (i.e. for all $n\in \N$ and all $w \in V_n$, $H_w^{(n)} = r_0 \ \cdots \ r_{n-1}$), we get
\begin{align*}
    \wh \nu(\wh X_{\ov B}) & = \sum_{n = 0}^{\infty} \sum_{v \in W_{n+1}}\sum_{w \in W_n'} f_{v,w}^{(n)} 
(r_0 \ \cdots \ r_{n-1}) \frac{\xi_v}{\lambda^{n+1}}\\
&\geq \sum_{n = 0}^{\infty} \frac{r_0 \ \cdots \ r_{n-1}}{\lambda^{n+1}} 
m'_n |W'_n| \sum_{v \in W_{n+1}} \xi_v. 
\end{align*} 
where $m'_n$ is defined in \eqref{eq s4 m'_n}. The result follows from the fact that $\lambda  < r_i$ for all 
$i = 0, \ldots, n-1$.
\end{proof}

\subsection{Measure extension from a fat odometer}\label{sec: fat odometer} 
We consider a class of generalized Bratteli diagrams of bounded size with a vertical subdiagram representing an odometer. Recall that, for such Bratteli diagrams, it is convenient to identify the set of vertices at each level with $\Z$.

Let $B = B(F_n)$ be a generalized Bratteli diagram with incidence matrices $F_n = (f_{v,w}^{(n)})$, for $n \in \N_0$ is given by
\begin{equation}\label{eq: fat_odo}
f_{v,w}^{(n)} =
\left\{
\begin{aligned}
&a_n, \quad v = 0,\;  w = 0,\\
&1, \quad v \neq 0,\;  |v - w| \leq t_n,\\
&0, \quad v \neq 0, \; |v - w| > t_n.
\end{aligned}
\right.
\end{equation}
and $a_n > 1$ for all $n \in \N_0$.
Then $B$ is a generalized Bratteli diagram of bounded size (see \Cref{Def:BD_bdd_size}) where the row sum (the cardinality of $r^{-1}(v)$) for $v \in V_{n+1}$, $n \in \N_0$ is determined 
by the formula 
$$
\sum_{w \in V_n} f_{v,w}^{(n)} =
\left\{
\begin{aligned}
&a_n + 2 t_n, \quad v = 0,\\
&2t_n + 1, \quad v \neq 0.\\
\end{aligned}
\right.
$$

Denote by $\ov B$ the vertex subdiagram with $W_n = \{0\}$, i.e., 
$\ol B$ is the odometer passing through zero vertices on every level
and having $a_n$ edges between the vertices of the $n$-th and $(n+1)$-th levels. We refer to this diagram informally as a \textit{fat odometer}. There is a unique probability ergodic invariant measure $\ol\nu$ on $X_{\ov B}$ which gives the value
$$
\ov p_0^{(n+1)} = \ol\nu([\ov e]) = \frac{1}{a_0\ \cdots\ a_n}
$$
to a cylinder set $[\ov e] \subset X_{\ov B}$ with $r(\ov e) = \{0\} \in V_{n+1}$.
Our goal is to determine when the 
measure extension $\wh {\ol\nu}(\wh X_{\ov B})$ is finite (or infinite). 

\begin{theorem}\label{thm: ext fat odom}
Suppose that the sequences $(a_n)_{n=0}^{\infty}$ and $(t_n)_{n=0}^{\infty}$ as in the definition of $B=(F_n)$ above (see \eqref{eq: fat_odo}) are increasing.
Then 
\be\label{eq3 finite ext}
\sum_{n = 0}^{\infty} \frac{t_n}{a_n} < \infty \ \Longrightarrow \ 
\wh {\ol\nu}(\wh X_{\ov B}) < \infty.
\ee 
\end{theorem}

\begin{proof} By \eqref{eq: exten_representation}, we have 
\be\label{Formula_nu_ext}
\ba 
\wh \nu(\wh X_{\ov B}) &= 1 + \sum_{n = 0}^{\infty} \sum_{w \in W_n'}f_{0,w}^{(n)} H_w^{(n)}\ov p_0^{(n+1)}\\
&= 1 + \sum_{n = 0}^{\infty} \frac{1}{a_0\ \cdots\ a_n} \sum_{w \in [-t_n,t_n] \setminus \{0\}} H_w^{(n)}.
\ea
\ee
We can write 
$$
\ba 
\sum_{w \in [-t_n,t_n] \setminus \{0\}} H_w^{(n)} & = 
\sum_{|w| \leq t_{n-1}, w\neq 0} H_w^{(n)}  + 
\sum_{t_{n-1} < |w| \leq t_{n}} H_w^{(n)}  \\
& \leq 2 t_{n-1} H_0^{(n-1)} + 2 (2t_{n-1} +1) (t_n - t_{n-1}) 
H_0^{(n-1)}.
\ea 
$$ 
It is straightforward to check by induction that for all $n$
\be\label{eq4:heights estimate}
H_0^{(n)} \leq \prod_{i=0}^{n-1} (a_i + 2 t_i).
\ee
Using \eqref{eq4:heights estimate}, we get that
$$
\sum_{|w| \leq t_n, w \neq 0} H_w^{(n)} \leq 
\prod_{i=0}^{n-2} s_n (a_i + 2 t_i), 
$$
where $s_n = 2(2 t_nt_{n-1} + t_n - 2t^2_{n-1})$. Substituting this estimate into \eqref{Formula_nu_ext}, we have 
\be\label{eq4:estimate}
\ba 
\wh \nu(\wh X_{\ov B}) & \leq 1 +  \sum_{n = 0}^{\infty} \frac{2}{a_0\ \cdots\ a_n} 
\prod_{i=0}^{n-2} s_n (a_i + 2 t_i) \\
& = 1 + 2 \sum_{n=0}^\infty \frac{s_n}{a_{n-1} a_n}\left( 1 + 
\frac{2t_0}{a_0}\right) \ \cdots \ \left( 1 + \frac{2t_{n-2}}{a_{n-2}}\right).
\ea
\ee
Note that if $\sum_{n = 0}^{\infty} \frac{t_n}{a_n} < \infty$, then the product $\left( 1 + \frac{2t_0}{a_0}\right) \ \cdots \ \left( 1 + \frac{2t_{n-2}}{a_{n-2}}\right)$ converges as $n \to \infty$. Moreover, 
$$
\sum_{n=0}^\infty \frac{s_n}{a_{n-1} a_n} = \sum_{n=0}^\infty  
\left( \frac{2t_{n-1}(t_n - t_{n-1})}{a_{n-1}a_n} 
+  \frac{t_{n}}{a_{n-1}a_{n}} \right) < \infty.
$$
Thus, by \eqref{eq4:estimate} the measure extension $\wh {\ol\nu}(\wh X_{\ov B})$ is finite.
\end{proof}

\begin{remark}
We can use a slightly weaker estimate for the heights of the towers.
Note that, for $w \in [-t_n, t_n] \setminus \{0\}$, we have
$H_w^{(n)} \leq (2t_{n-1} + 1)H^{(n-1)}_0$, and therefore
$$
\sum_{w \in [-t_n, t_n] \setminus \{0\}} H_w^{(n)} \leq  
2t_n (2t_{n-1} + 1)H^{(n-1)}_0.
$$
Then, as in the proof of Theorem \ref{thm: ext fat odom}, one can see that 
\be\label{eq4:weaker cond}
\wh{\ol \nu}(\wh X_{\ov B}) \leq 
1 + 2\sum_{n = 0}^{\infty} \frac{t_n(2t_{n-1} + 1)}{a_{n-1}a_n}\left(1 + \frac{2t_0}{a_0}\right)\cdots \left(1 + \frac{2t_{n-2}}{a_{n-2}}\right).
\ee
Since $s_n = 2(2 t_nt_{n-1} + t_n - 2t^2_{n-1}) < 2t_n(2t_{n-1} + 1)$,
relation \eqref{eq4:weaker cond} is weaker than relation \eqref{eq4:estimate}. Clearly, $\wh{\ol \nu}(\wh X_{\ov B})$ is finite if the series in \eqref{eq4:weaker cond} converges. 
\end{remark}

\begin{remark}
In \cite[Section 4]{BezuglyiJorgensenKarpelKwiatkowski2025}, the measure extension from odometer was considered for irreducible generalized Bratteli diagrams with countably many vertical odometers as subdiagrams (see also \Cref{Ex_hor_stat_IO}). 
\end{remark}

\section{Step-by-step measure extension}\label{sect step-by-step}

In this section, we discuss a modification of the measure extension procedure used in \Cref{sect meas ext subdgrms} and 
\Cref{sect ext from classes sbdgrms}. We first describe the idea of this method and then illustrate it with examples. 

Let $B = B(F_n)$ be a generalized Bratteli diagram and $\ol B = B(\ol F_n)$ a vertex subdiagram, and let $\nu$ be a probability tail invariant measure defined on $X_{\ol B}$. As before, we want to study whether the extension $\wh\nu$ of $\nu$ to 
$\wh X_{\ol B}$ is finite or infinite.  

\begin{prop}\label{prop step-by-step}
For the diagrams $B$ and $\ol B$ as above, assume that we can find 
another Bratteli diagram $B'$ such that $\ol B$ is a subdiagram of
$B'$ and $B'$ is a subdiagram of $B$ i.e. $X_{\ol B} \subset X_{B'} \subset X_{B}$. As before, we assume $\nu$ to be a probability measure on $X_{\ol B}$ and denote by $\nu'$ the measure extension of $\nu$ to the space $\wh X_{\ol B} \cap X_{B'}$. 
If $\nu'(\wh X_{\ol B} \cap X_{B'}) = \infty$, then measure extension of $\nu$ to $X_B$ is infinite, i.e. $\wh\nu(\wh X_{\ol B}) = \infty$. Similarly, if the measure extension of $\nu'$ to $X_B$ is finite, i.e. $\wh{\nu'}(\wh X_{B'}) < \infty$, then $\wh\nu(\wh X_{\ol B}) < \infty.$
\end{prop}

\begin{proof}
This statement is obvious. 
\end{proof}

\begin{remark}\label{rem: step_step}
The purpose of \Cref{prop step-by-step} is based on the following observation. The basic formula for the measure extension 
\eqref{eq: exten_representation} is efficient if we know how to find 
(or estimate) the values $H_w^{(n)}$ where $w \in W_n'$. Our choice of a wider subdiagram $B'$ is motivated by the possibility of determining the tower heights in $B'$ relatively easily. Clearly, we can use a chain of subdiagrams (if necessary) 
$\ol B\subset B_1\subset B_2\subset \cdots \subset B$ to answer the question about the finiteness of the measure extension from $\ol B$.
\end{remark} 

In what follows, we apply the approach formulated in \Cref{prop step-by-step} (and \Cref{rem: step_step}) to a one-sided infinite Bratteli diagram that contains a 
\textit{fat odometer} (see \Cref{sec: fat odometer}). For this, we describe the three generalized Bratteli diagrams $B$, $\ol B$, and $B'$ that will play an important role in this section.

\textbf{Diagram $B.$} 
Let $(a_n)_{n=0}^{\infty}$ and $(t_n)_{n=0}^{\infty}$ be two sequences of natural numbers. We consider the most interesting case when these sequences are increasing. Let $B = B(F_n)$ be a generalized Bratteli diagram defined by the 
sequence of incidence matrices $(F_n)_{n \in \N_0}$. To define $B= B(F_n)$, we use the bounded size diagram considered in \eqref{eq: fat_odo} and take the right half of this diagram so that $F_n$ is an $\N_0 \times \N_0$ matrix for each $n \in \N_0$. It can also be described as follows:
$$
F_n = \begin{pmatrix}
    a_n & 1 & \cdots & 1 & 0 & 0 & 0 & \cdots & \cdots \\
    1 & 1 &  \cdots  & \cdots & 1 & 0 & 0 & \cdots & \cdots\\
    \cdots & \cdots & \cdots & \cdots & \cdots & \cdots & \cdots & \cdots & \cdots & \\
0 & \cdots & 0 & 1 & \cdots & 1 & 0 &\cdots &\cdots \\
\cdots & \cdots & \cdots & \cdots & \cdots & \cdots & \cdots & \cdots & \cdots & \\
\end{pmatrix},
$$
where the string of $1$s has the length $t_n$ in the zeroth row and $t_n +2$ in the first row. Then the length of such a string beginning at the zero vertex increases by 1 in each subsequent row until we reach the $t_n$th row. The $t_n$th row has exactly $2t_n +1$ entries equal to $1$. For the $t_n +k$-row, $f^{(n)}_{t_n+ k, i} =1$ if and only if $i = k, \ldots, k + 2t_n$. In other words, the entries of the $\N_0 \times \N_0$ $F_n$ are determined as follows: For $i,j \in \N_0$ we have
\[
f_{j,i}^{(n)} =
\left\{
\begin{aligned}
&a_n, \quad i = 0,\;  j = 0,\\
&1, \quad j \neq 0,\; i\neq 0, \quad |j - i| \leq t_n,\\
&0, \quad \mbox{otherwise}.
\end{aligned}
\right.
\]
\textbf{Diagram $\ol B$.} Let $\ol B$ be the odometer sitting on the vertices $0$ with $a_n$ edges between the levels $V_n$ and $V_{n+1}$ for $n \in \N_0$. The unique tail invariant probability measure on $X_{\ol B}$ is defined by 
$$
\nu([\ol e]) = \frac{1}{a_0 \ \cdots\ a_{n-1}} 
$$
where $[\ol e]$ is a cylinder set in $X_{\ol B}$ ending at $0 \in V_{n}$ for $n \in \N$. In what follows, we will also call the diagram $\ol B$ the \textit{$(a_n)$-odometer}.

\textbf{Diagram $B'$.} The standard Bratteli diagram $B'$ is defined as a rooted vertex subdiagram of $B$ whose vertices are chosen as follows: $W_0 =\{0\}$ and $W_n = [0, 1, \ldots, t_{n-1}]$ for $n \in \N$.
Recall that $t_n < t_{n+1}$ for all $n \in \N_0$.
Note that $\ol B$ is also a subdiagram of $B'$. The sequence of incidence matrices $(F'_n)_{n \in \N_0}$ that determines $B'$ can be written in the following form
$$
F'_n = \begin{pmatrix}
    a_n & 1 & \cdots & 1 \\
    1 & 1 & \cdots & 1 \\
\cdots & \cdots & \cdots  & \cdots\\
    1 & 1 & \cdots & 1 \\
\end{pmatrix} 
$$
Note that $F'_0$ is a column vector and the size of $F'_n$ is $|W_{n+1}| \times |W_n| = (t_n +1) \times (t_{n-1} +1)$ where $n \in \N$. 

The role of the diagram $B'$ becomes clear from the following discussion. For $n \in \N_0$, we denote by $h_w^{(n)}$ the height of the tower in $B'$ corresponding to the vertex $w \in W_n$. Recall that we identify $V_n$ with $\N_0$ for every level $n$ $\in \N_0$, thus, for $w =0 \in V_n$, we can use $h_0^{(n)}$ to denote the height of the tower.
It is clear that, for all $n \in \N,$ $h^{(n)} = F'_{n-1}h^{(n-1)} $ where 
$h^{(n)}$ is the vector with the entries $h_w^{(n)}$, $w \in W_n$.

We define a sequence $(k^{(n)})_{n\in \N_0}$ of positive integers as follows. Let  $k^{(0)} : = h^{(0)}_0 = 1$ and $k^{(1)} : =  h^{(0)}_w = 1$ for $w \in [1,\ldots,t_0]$. Observe that every vertex $w \in [1,..., t_{n}]$ is connected by a single edge with every vertex $w' \in W_{n}$. Thus, $h^{(n+1)}_w$ does not depend on $w \in [1,..., t_{n}]$. For $n \geq 2$, we set $k^{(n)}: = 
h^{(n)}_w$ for $w \in W_n \setminus \{0\}$.

\begin{lemma}\label{lem s5 heights B'} $h_{0}^{(1)} = a_0$ and for $n \geq 2$,  
\begin{enumerate}[label=(\alph*)]
    \item $h^{(n)}_0 = a_{n-1}h_0^{(n-1)} + t_{n-2} k^{(n-1)}$,
    \item $k^{(n)} = h^{(n-1)}_0 + t_{n-2} k^{(n-1)}$.
\end{enumerate}
\end{lemma}

\begin{proof} The lemma follows immediately from the relation  $h^{(n)} = F'_{n-1}h^{(n-1)}$ written in the entry-wise form. 
\end{proof}

We will now find expressions for $h^{(n)}_0$ and $k^{(n)}$ in terms of the sequences $(a_n)_{n \in \N_0}$ and $(t_n)_{n \in \N_0}$. 

\begin{lemma}\label{lem s5 k^{(n)}}
The following relation holds for all $n\geq 2$:
\be\label{eq s5 k^{(n)}}
k^{(n)} = h^{(n-1)}_0 + \sum_{i =0}^{n-2} t_{n-2}\ \cdots\ t_{n-2 -i} 
\ h_0^{(n-2 -i)},
\ee
where $h^{(0)}_0 =1$. 
\end{lemma}

\begin{proof}
For $n=2$, we see that $k^{(2)} = a_0 + t_0$ and this agrees with \eqref{eq s5 k^{(n)}}. Using \Cref{lem s5 heights B'}, and then applying induction, we see that for $n \geq 3$,
$$
\ba
k^{(n)} = & \ h^{(n-1)}_0 + t_{n-2} k^{(n-1)} \\
= &\ h^{(n-1)}_0 + t_{n-2} [ h^{(n-2)}_0 + \sum_{i =0}^{n-3} t_{n-3}\ \cdots\ t_{n-3 -i} \ h_0^{(n-3 -i)}]\\
= &\ h^{(n-1)}_0 + \sum_{i =0}^{n-2} t_{n-2}\ \cdots\ t_{n-2 -i} 
\ h_0^{(n-2 -i)},
\ea
$$ as needed.
\end{proof} We note that 
\eqref{eq s5 k^{(n)}} can also be written as follows:
$$
k^{(n)} = h^{(n-1)}_0 + \sum_{i=2}^{n} t_{n-2}\ \cdots\ t_{n -i} 
\ h_0^{(n-i)}.
$$ 
Based on \Cref{lem s5 k^{(n)}}, we obtain the exact expression for the measure extension from $\ol B$ (i.e. the odometer $(a_n)$) to $\wh X_{\ol B} \cap X_{B'}$. For convenience, we set $X'_{\ol B} = \wh X_{\ol B} \cap X_{B'}$. In other words, we will find an expression for the measure extension $\wh\nu(X'_{\ol B})$ where $\nu$ is a tail invariant probability measure in $X_{\ol B}$.

\begin{thm}\label{thm s5}
For the diagrams $\ol B$ and $B'$ as above, the  extension of the measure $\nu$ to $X'_{\ol B}$ is 
\be\label{eq s5 ext to B'}
\wh\nu(X'_{\ol B}) =  1 +   \sum_{n=2}^{\infty} \sum_{k =0}^{n-1}
\frac{(t_{n-1} \ \cdots\ t_{k})\; h_0^{(k)}}{a_0\ \cdots\ a_{n-1}}.
\ee
\end{thm}

\begin{proof} From \Cref{lem s5 k^{(n)}}, it follows that
$$
\ba
\wh\nu(X'_{\ol B}) = & \ 1 + \sum_{n=2}^{\infty} 
\frac{1}{a_0\ \cdots\ a_{n-1}} \sum_{w =1}^{t_{n-1}} h_w^{(n)}\\
= & \ 1 + \sum_{n=2}^{\infty} 
\frac{1}{a_0\ \cdots\ a_{n-1}} t_{n-1} k^{(n)}\\
= &\ 1  + \sum_{n=2}^{\infty} 
\frac{1}{a_0\ \cdots\ a_{n-1}} t_{n-1} \left(h^{(n-1)}_0 + \sum_{i=0}^{n-2} t_{n-2}\ \cdots\ t_{n -2 -i} \ h_0^{(n -2-i)}\right)\\
= &\ 1 +  \sum_{n=2}^{\infty} 
\frac{1}{a_0\ \cdots\ a_{n-1}}  \left(t_{n-1}h_0^{(n-1)} + 
\sum_{j =1}^{n-1} t_{n-1} \ \cdots\ t_{n-1 -j} h_0^{(n-1-j)}\right)
\\
= &\ 1  +  \sum_{n=2}^{\infty} \sum_{i =0}^{n-1}
\frac{(t_{n-1} \ \cdots\ t_{n-i-1})\; h_0^{(n-i-1)}}{a_0\ \cdots\ a_{n-1}}\\
= &\ 1  +  \sum_{n=2}^{\infty} \sum_{k =0}^{n-1}
\frac{(t_{n-1} \ \cdots\ t_{k})\; h_0^{(k)}}{a_0\ \cdots\ a_{n-1}}.
\ea
$$
\end{proof}

\begin{corol} \label{cor s5}
The following statements hold:
\begin{enumerate}
    \item If $\sum_{n =0}^{\infty} \frac{1}{a_n} = \infty$, then $\wh\nu(X'_{\ol B}) = \infty$. 
    \item If $\sum_{n =0}^{\infty} \frac{1}{a_n} < \infty$, then there exists a sequence $(t_n)_{n \in \N_0}$ such that $\wh\nu(X'_{\ol B}) = \infty$. 
\end{enumerate}
In both cases, the measure extension $\wh\nu(\wh X_{\ol B})$ to the diagram $B$ is infinite. 
\end{corol}

\begin{proof}
For (1), we consider the series in \eqref{eq s5 ext to B'}, and show that it diverges. Obviously, we can drop a finite number of terms, proving the divergence. Clearly, we have $h_0^{(1)} = a_0$ and 
$$
h_0^{(k)} > a_0 \ \cdots \ a_{k-1}, \ k\geq 2.  
$$
Using \eqref{eq s5 ext to B'} we get 

\be\label{eq s5 corol}
\ba
\wh\nu(X'_{\ol B}) =  & \sum_{n=2}^{\infty}\left( \frac{t_{n-1} \ \cdots\ t_{0}}{a_{n-1}\ \cdots\ a_{0}} +  \sum_{k =1}^{n-1}
\frac{(t_{n-1} \ \cdots\ t_{k})\; h_0^{(k)}}
{a_0\ \cdots\ a_{n-1}}\right) \\
> &\ \sum_{n=2}^{\infty}\left( \frac{t_{n-1} \ \cdots\ t_{0}}{a_{n-1}\ \cdots\ a_{0}} + \sum_{k =1}^{n-1}
\frac{(t_{n-1} \ \cdots\ t_{k})\; (a_0\ \cdots \ a_{k-1})}
{a_0\ \cdots\ a_{n-1}}\right)\\
 = &\ \sum_{n=2}^{\infty} \ \sum_{k =0}^{n-1} 
 \frac{t_{n-1} \ \cdots\ t_k}{a_{n-1}\ \cdots\ a_k}\\ 
> &\ \sum_{n=1}^{\infty} \frac{t_n}{ a_n} \\
> &\ \sum_{n=1}^{\infty} \frac{1}{ a_n} \\
 = & \ \infty.
 \ea 
\ee
(2) is obvious.
\end{proof}

\begin{remark} Note that
\begin{enumerate}
\item there are other conditions on the sequences $(a_n)_{n \in \N_0}$ and $(t_n)_{n \in \N_0}$ that will imply the divergence in \eqref{eq s5 corol}. For example, take $t_n = a_n = n^2$ for $n \in \N_0$.
\item if we use the inequality $h_0^{(n)} > t_0 \ \cdots \ t_{n-2}$ for $n \geq 2$ (which follows from \Cref{lem s5 k^{(n)}}), 
then we get another lower bound for $\wh\nu(X'_{\ol B})$:
$$
\wh\nu(X'_{\ol B}) > \sum_{n=2}^{\infty} \sum_{k= 0}^{n-1} \frac{1}{t_{k-1}} \left( \prod_{i=0}^{n-1} \frac{t_i}{a_i}\right).
$$
\end{enumerate}
\end{remark}

We now consider the case where the measure extension $\wh\nu(X'_{\ol B})$
can be finite. 

\begin{lemma}
In the notations introduced above, we have $h_0^{(1)} = a_0$ and  for $i \geq 2$,
\be\label{eq s5 h_0 from above}
h_0^{(i)} \leq  (a_{i-1} + t_{i-2}) \ \cdots\ (a_1 + t_0)
a_0.
\ee
\end{lemma}

\begin{proof} 
From \Cref{lem s5 k^{(n)}} we deduce that for $i\geq 1$, 
$k^{(i)} < h_0^{(i)}$. Thus, for $i \geq 2$, we have
$$
h_0^{(i)} \leq (a_{i-1} + t_{i-2} )h_0^{(i-1)}.
$$ 
Therefore, the result follows by induction. 
\end{proof}

Based on this lemma, we can obtain a sufficient condition for the finiteness of the measure extension $\wh\nu(X'_{\ol B})$. 

\begin{thm} \label{thm s5-2}
For the diagrams $\ol B$ and $B'$ as above, suppose that the sequences $(a_n)_{n \in \N_0}$ and $(t_n)_{n \in \N_0}$ are taken such that 
$$
\sum_{n = 0}^{\infty} \frac{t_n}{a_n} < \infty.
$$
Then $\wh\nu(X'_{\ol B}) < \infty$. 
\end{thm}

\begin{proof} 
We recall that the finiteness of $\wh\nu(X'_{\ol B})$ is determined by the convergence of the series \eqref{eq s5 ext to B'}.  
We apply the estimate from \eqref{eq s5 h_0 from above} to \eqref{eq s5 ext to B'}. 
We also use the fact that the sequence $(t_n)_{n \in \N_0}$ is increasing. 
$$
\ba 
\wh\nu(X'_{\ol B}) < \infty \ \Longleftrightarrow \ & \sum_{n=2}^{\infty}  \left( \frac{t_{n-1} \ \cdots\ t_{0}}{a_{n-1}\ \cdots\ a_{0}} + \frac{t_{n-1} \ \cdots\ t_{1}}{a_{n-1}\ \cdots\ a_{1}} +\sum_{k =2 }^{n-1}
\frac{(t_{n-1} \ \cdots\ t_{k})\; h_0^{(k)}}{a_0\ \cdots\ a_{n-1}}
\right)\\
< &\  \sum_{n=2}^{\infty} \left( \frac{t_{n-1} \ \cdots\ t_{0}}{a_{n-1}\ \cdots\ a_{0}} + \frac{t_{n-1} \ \cdots\ t_{1}}{a_{n-1}\ \cdots\ a_{1}} +
\sum_{k =2}^{n-1}
\frac{(t_{n-1} \ \cdots\ t_{k})(a_{k-1} + t_{k-2}) \ \cdots\ (a_1 + t_0)a_0}{a_0\ \cdots\ a_{n-1}}\right)\\
< &\  \sum_{n=2}^{\infty} \left( \frac{t_{n-1} \ \cdots\ t_{0}}{a_{n-1}\ \cdots\ a_{0}} + \frac{t_{n-1} \ \cdots\ t_{1}}{a_{n-1}\ \cdots\ a_{1}} +
\sum_{k=2}^{n-1} \frac{(t_{n-1} \ \cdots\ t_{k})(a_{k-1} + t_{k-1}) \ \cdots\ (a_0 + t_0)}{a_0\ \cdots\ a_{n-1}}\right)\\
= &\   \sum_{n=2}^{\infty} \left( \frac{t_{n-1} \ \cdots\ t_{0}}{a_{n-1}\ \cdots\ a_{0}} + \frac{t_{n-1} \ \cdots\ t_{1}}{a_{n-1}\ \cdots\ a_{1}} +
\sum_{k=2}^{n-1} 
\left(1 + \frac{t_0}{a_0}\right) \ \cdots \ \left(1 + \frac{t_{k-1}}{a_{k-1}}\right)
\frac{t_{k}}{a_{k}}\ \cdots\ \frac{t_{n-1}}{a_{n-1}}\right).\\
\ea
$$

To simplify the notation, set $y_k = \frac{t_k}{a_k}$ for all $k$.
We need to show that the  convergence of $\sum_{k=0}^{\infty} y_k$ implies that the series $S_1$ and $S_2$ converge, where 
\be\label{eq s5 conv}
S_1 := \sum_{n \geq 2} y_0\ \cdots \ y_n, \quad 
S_2 := \sum_{n=2}^{\infty} \sum_{k=2}^{n-1} (1 + y_0) \ \cdots\ (1 + y_{k-1})
y_k \ \cdots\ y_{n-1} < \infty.
\ee
It is obvious that the convergence of $\sum_{k=0}^{\infty} y_k$ implies that $S_1 < \infty$. 
Since $\prod_{k=0}^{\infty} (1 + y_k)$ converges, there exists $M >0$ such that
$$
\prod_{j=0}^{k-1} (1 + y_j) < M, \quad \mbox{for all}\;\; k \in \N.
$$
Hence, relation \eqref{eq s5 conv} will follow if we prove that 
\be\label{eq s5 finiteness}
\sum_{n=2}^{\infty} \sum_{k=2}^{n-1}y_k \ \cdots\ y_{n-1} < \infty.
\ee
Choose $K \in \N_0$, such that $y_k < \frac{1}{2}$ for all $k \geq K$.  Reversing the order of summation, we write 
$$
\ba
\sum_{n=2}^{\infty} \sum_{k=2}^{n-1}y_k \ \cdots\ y_{n-1} = & 
\sum_{k= 0}^K \sum_{n= k+1}^{\infty} y_k \ \cdots\ y_{n-1} 
+ \sum_{k\geq K} y_k (1 + y_{k+1} + y_{k+1}y_{k+2} + \ \cdots) \\
<  & \sum_{k =2}^K \sum_{n= k+1}^{\infty} y_k \ \cdots\ y_{n-1}
 + \sum_{k\geq K}  y_k \left( 1 + \frac{1}{2} + \frac{1}{4} + \ \cdots \right)\\
= & \sum_{k =2}^K \sum_{n= k+1}^{\infty} y_k \ \cdots\ y_{n-1} 
+ 2 \sum_{k\geq K}  y_k.\\
\ea 
$$
By the same method, we can show that 
$$
\sum_{k\geq 2}^K \sum_{n= k+1}^{\infty} y_k \ \cdots\ y_{n-1} < \infty, 
$$
because $y_n < 1/2$ for all sufficiently large $n$. 
Thus, \eqref{eq s5 finiteness} holds, and this completes the proof of the theorem.
\end{proof}

\begin{remark}
We note that the convergence of the series in \eqref{eq s5 conv}
implies that 
$$
\sum_{n =0}^{\infty} \frac{t_n}{a_n} < \infty, 
$$
assuming that for all $n$ there exists $c < 1$ such that $\frac{t_n}{a_n} < c$. 
\end{remark} 

\section{Approximation by subdiagrams} \label{sect Approximation}

In this section, we discuss how a generalized Bratteli diagram can be approximated by a sequence of standard subdiagrams. This problem is related to the approximation of infinite matrices by sequences of finite matrices and leads to a notion of convergence of measures (see \Cref{Def: weak_conv}) supported by the path spaces of subdiagrams. We highlight that the convergence of measures using sequences of finite subgraphs or sublattices is a crucial technique in fields such as thermodynamic formalism and statistical mechanics. We can mention, for instance, the proof of the existence of an eigenmeasure for the Ruelle-Perron-Frobenius operator for a class of countable shift spaces and potentials \cite{MauUr2001}, which uses an approximation of subgraphs in the symbolic graph, generating a sequence of eigenmeasures in the subshifts that converges to an eigenmeasure for the whole space. In the case of lattices, the notion of thermodynamic limit is realized by using a sequence of increasing finite sublattices that covers $\mathbb{Z}^d$, $d \in \mathbb{N}$. We focus primarily on sequences of vertex subdiagrams.

\subsection{Approximation by stationary standard subdiagrams} 

The problem of approximating a stationary generalized Bratteli diagram by standard Bratteli diagrams is motivated by the following theorem, taken from \cite{Kitchens1998} (we have adopted this statement for the case of generalized Bratteli diagrams). 

\begin{thm}\cite[Theorem 7.1.4]{Kitchens1998}\label{Thm_approximation_Kitchens} 
Let $A$ be an infinite $\N \times \N$ matrix with nonnegative entries. Assume that $A$
is irreducible, aperiodic, and recurrent. For $i\in \N$, let $W_i = [1, \ldots, i]$ and $A_i$ be the $W_i \times W_i$ matrix obatined by restricting $A$ to $W_i$ rows and columns. Assume that $(A_i)_{i \in \N}$ is a sequence of finite, aperiodic, and irreducible matrices. Let $(\lambda, \ell = (\ell_v)_{v \in \N})$ be the Perron eigenpair for $A$, i.e. $A\ell = \lambda\ell$, normalized such that $\ell_1 =1$. 
For every submatrix $A_i$, let $(\lambda_i, \ell(i))$ be the Perron eigenpair, normalized such that $\ell_1(i) =1$. Then 
$$
\lim_{i \to\infty} \lambda_i = \lambda, \qquad \lim_{i\to\infty} 
\ell_v(i) = \ell_v, \ \ (v\in \N).
$$
\end{thm} Note that this result can also be reformulated for a two-sided infinite matrix, i.e., matrix indexed by $\Z \times \Z$.

Given a generalized Bratteli diagram $B$, we consider a sequence of standard subdiagrams $(\overline B_i)_{i \in \N}$ that in some sense ``converge'' to $B$. If $B = B(F)$ is stationary, then the $\N \times \N$ incidence matrix $F$ can be approximated by finite truncations $(F_i)$, $i \in \N$ of $F$, where $F_i$ is the north-west corner of $F$ of size $i \times i$. By \Cref{Thm_approximation_Kitchens}, the Perron eigenvalues and corresponding eigenvectors of the finite truncations $F_i$ converge to the Perron eigenvalue and eigenvector of $F$. Let $\mu$ and $\mu_i$ be the tail invariant measures on $X_B$ and $X_{\ol B_i}$ respectively, which are defined in \Cref{thm PFThm} and \Cref{Standard_stationary}. In this section, our goal is to study the following ``In what sense do the measures $\mu_i$ converge (or not) to $\mu$ as $i \rightarrow \infty$"? Since $\ol B_i$ for $i \in \N$ is a subdiagram of $B$, it makes sense to study the convergence of $\wh \mu_i$ to $\mu$. Here, $\wh \mu_i$ denotes the extension of the measure $\mu_i$ to $\wh X_{\overline B_i}$. To make our study more consistent, we set $\wh \mu_i (X_B \setminus \wh X_{\overline B_i} ) = 0$. Thus, we can consider $\wh \mu_i$ as a measure on $X_{B}$.

Let the matrix $F = A^T$, where $A$ satisfies the conditions of \Cref{Thm_approximation_Kitchens}. As mentioned above, let $B(F)$ be a stationary generalized Bratteli diagram defined by the incidence matrix $F$ with vertices indexed by natural numbers, that is, 
$V_n = \N$. As before, let $(W_i = [1, \ldots, i])_{i \in \N}$ be the sequence of subsets of vertices taken at every level
$V_n$. To emphasize that the set $W_i$ is viewed as a subset of $V_n$, we will denote it $W_{i, n}$. In this notation, we obtain a sequence $(\ol B_i)_{i \in \N}$ of standard stationary vertex subdiagrams,
where every $\ol B_i$, $i \in \N$ 
is determined by the $ i\times i$ submatrix $F_i$ of $F$ with the index set $W_i \times W_i$. 

For $i \in \N$, let $\lambda_i$ denote the Perron eigenvalue for the matrix $F_i$, and
let $\ol x(i)$ be the eigenvector for $A_i = F_i^T, A_i \ol x(i) = \lambda_i \ol x(i)$, normalized by the condition $x_1(i) =1$, where $\ol x(i) = (x_1(i),\ldots, x_i(i)) $. For a standard stationary Bratteli diagram with an irreducible incidence matrix, there exists a unique finite tail invariant measure $\mu_i$,
(see \Cref{Standard_stationary}). This unique invariant finite measure $\mu_i$ is defined as follows : Let $[\overline e] \subset X_{\ol B_i}$ is the cylinder set corresponding to a finite path $\overline e$ such that $r(\overline e) = v \in W_{i,n}$, then 
\begin{equation}\label{mui}
    \mu_i ([\overline e]) = \dfrac{1}{\sigma_i} \,\dfrac{x_v (i)}{\lambda_i^n},
\end{equation} 
where the coefficient
$$
\sigma_i = \sum_{v=1}^i x_v(i)
$$
normalizes the measure $\mu_i$ to a probability measure.

As in previous sections, we denote by $\wh X_{\ol B_i}: = \mathcal{R}(X_{\ol B_i}) \subset X_B$ the saturation of $X_{\ol B_i}$ under the tail equivalence relation, and let $\wh \mu_i$ be the measure extension of $\mu_i$, i.e., for a cylinder set $[\ol{f}] \in \wh X_{\ol B_i}$ such that $r(\ol{f}) = r(\ol{e})$ for some finite path $\ol{e} \in X_{\ol B_i}$ we set $\wh \mu_i([\ol{f}]) = \mu_i([\ol{e}])$.

We set $A = F^T$ and assume that $A$ satisfies the conditions of \Cref{Thm_approximation_Kitchens} (and \Cref{thm Perron-Frobenius Thm}). Since $B$ is a stationary generalized Bratteli diagram, we use \Cref{thm PFThm} and define a tail invariant measure $\mu$ on $X_B$. We denote by $(\ol \xi = (\xi_v)_{v \in \N}, \lambda)$ the Perron eigenpair for $A = F^T$, i.e. $A \ol \xi = \lambda \ol \xi$. By assuming that $A$ satisfies \Cref{thm PFThm}, part $(2)$ i.e.
\[
\sigma = \underset{v \in \N}{\sum} \xi_v < \infty, \qquad  \xi = (\xi_v : v \in \N).
\] we obtain that $\mu$ is a finite measure. The fact that $\sigma < \infty$ also allows us to normalize so that $\xi_1 =1$. 
More precisely, (see \Cref{thm PFThm}) the value of $\mu$ on a cylinder set $[\ol e] \subset X_B$ is determined by the formula
\begin{equation}\label{mu}
\mu ([\overline e]) = \frac{1}{\sigma} \frac{\xi_v}{\lambda^n}, \qquad r(\ol e) = v \in V_n. 
\end{equation}Observe that the term $\frac{1}{\sigma}$ in the expression above is to ensure that $\mu$ is a probability measure.

\vspace{3mm}

\noindent We now define the notion of \textit{weak measure convergence on cylinder sets}. 

\begin{definition}\label{Def: weak_conv}
Let $(\Omega, \B, \nu)$ be a probability measure space where $\Omega$ is a zero-dimensional Polish space (a countable family of cylinder sets generates the topology). Let $(\Omega_i)_{i \in \N}$ be an increasing sequence of Borel subsets of $\Omega$ such that for every cylinder set $C$ there exists some $i'$ such that $C \cap \Omega_i\neq \emptyset$ 
for all $i \geq i'$. Let $\nu_i$ be a probability measure supported by $\Omega_i$. We say that the sequence $(\nu_i)_{i \in \N}$ \textit{weakly converges to $\nu$ on cylinder sets} if, 
for every cylinder set $C \subset \Omega$, 
$$
\lim_{i\to \infty} \nu_i (C \cap \Omega_i) = \nu (C).
$$ 
We will denote convergence on cylinders by the symbol $\nu_i \xrightarrow{w-[C]} \nu$. 
\end{definition}

\begin{remark}
In \cite{Iommi_Veloso_2021}, a similar notion of converging on cylinders is defined (see \cite[Definition 3.14]{Iommi_Veloso_2021}), where the authors considered subprobability invariant measures on the path space of a countable Markov shift. The difference is that in our definition, measures $\nu_i$ are supported by subsets $\Omega_i$ of $\Omega$ that may have zero measure $\nu$. The notion of convergence on cylinder sets is discussed in \Cref{Sec: conv of measures} (see \Cref{def_conv_cyl_set}).
\end{remark}

The following statement shows that the natural approximation of a stationary Bratteli diagram by an increasing sequence of standard Bratteli diagrams leads to the convergence of the corresponding measures on cylinder sets. 

\begin{prop} \label{prop appr truncated matrices}
Suppose that a generalized stationary Bratteli diagram $B$ is determined by the incidence matrix $F$ and its subdiagram $\ol B_i$ is constructed by the matrix $F_i$, the truncation of $F$ on the set $[1, ..., i], i \in \N,$ as described above in this subsection. Suppose that the matrices $F$ and $F_i$
satisfy \Cref{Thm_approximation_Kitchens}. If $\ol \xi$ and
$\ol x(i)$ are the right Perron eigenvectors for $A = F^T$ and 
$A_i = F_i^T$, respectively, such that 
\be\label{eq s7 lim x is xi}
\lim_{i\to \infty}|| \ol x(i)||_{\ell^1} = || \ol \xi||_{\ell^1},
\ee
then $\wh \mu_i \xrightarrow{w-[C]} \mu$ on $(X_B)$, where  
the measure $\mu$ and the sequence of measures $(\mu_i)_{i \in \N}$ are defined as in \eqref{mu} and \eqref{mui}, respectively, and $\wh \mu_i$ is the extension of the measure of $\mu_i$ for $i \in \N$. 
 \end{prop}

\begin{proof}
Since $(\ol B_i)_{i \in \N}$ is an increasing sequence of subdiagrams of $B$, the sequence of path spaces $X_{\ol B_i}$ is also increasing and, for every cylinder set $[\ol e]$ in $X_B$, we have
$[\ol e] \cap X_{\ol B_i} \neq\emptyset$ for all sufficiently large $i$. We need to show that 
\be\label{eq s7 conv in meas}
\lim_{i\to\infty}\wh \mu_i ([\ol e] \cap \wh X_{\ol B_i}) = \mu ([\ol e])
\ee
Note that for every finite path $\ol e = (e_0, \ldots, e_{m-1})$, there exists $i' \in \N$ such that $s(e_j) \in [1, \ldots, i']$, $j = 0, \dots, m-1$, and $r(\ol e) = r(e_{m-1}) = v \in [1, \ldots, i']$. In other words, $\ol e$ belongs to all standard Bratteli diagrams $\ol B_i$ for $i \geq i'$. By tail invariance, for $i \geq i'$, 
$$
\wh \mu_i([\ol e] \cap \wh X_{\ol B_i}) = \mu_i([\ol e] \cap X_{\ol B_i})= 
\frac{1}{\sigma_i}\,\frac{x_{v}(i)}{\lambda_i^m}.
$$ 
Hence, \eqref{eq s7 conv in meas} will hold if,
$$
\lim_{i \to \infty} \frac{1}{\sigma_i} \,\frac{x_{v} (i)}{\lambda_i^m} = \frac{1}{\sigma} \frac{\xi_v}{\lambda^m}. 
$$
Note that by \Cref{Thm_approximation_Kitchens}, for every $v \in \Z$, 
 $x_{v}(i) \xrightarrow{i \rightarrow \infty} \xi_v$ and $\lambda_i \xrightarrow{i \rightarrow \infty} \lambda$. 
Thus, it remains to check that $\sigma_i \xrightarrow{i \rightarrow \infty} \sigma$ to finish the proof. But the latter convergence is given in \eqref{eq s7 lim x is xi}.

\end{proof}
\subsection{Infinite Leslie matrix} 
In this subsection, we consider a class of infinite matrices for which the conditions of \Cref{prop appr truncated matrices} hold. Define the following sequence of square matrices:
$$
A_k = \begin{pmatrix}\label{eq: Leslie_finite}
b_1 & b_2 & b_3 &  \ldots &b_{k-1} & b_k \\
s_1 & 0& 0 &   \ldots & 0 & 0 \\
0 &  s_2 & 0 &   \ldots & 0 & 0 \\
\vdots & \vdots & \vdots & \vdots & \vdots & \vdots \\
0 & 0& 0& \ldots & s_{k-1} & 0 \\ 
\end{pmatrix}, \quad  k \in \N,
$$
where all $s_i$ and $b_i$ are positive integers.
These matrices are known as \textit{Leslie matrices} and arise naturally in models of age-structured population growth
(see, e.g., \cite{Leslie1948}).
Our interest in this family of matrices is motivated by the application of \Cref{Thm_approximation_Kitchens} and \Cref{prop appr truncated matrices}. 

We first calculate the Perron eigenvalue $\lambda_k$ and eigenvector 
$\ol x(k) =(x_i(k) :i =1,..., k)$ for $A_k$. Using the relation
$A_k \ol x(k) = \lambda_k \ol x(k)$, we have the following system of equations:
\be\label{eq Leslie eigenpair A_k}
\ba 
b_1 x_1(k) + \cdots + b_k x_k(k) = & \lambda_k x_1\\
s_i x_i(k) = & \lambda_k x_{i+1}(k), \quad i = 1, ..., k-1.
\ea
\ee
We normalize the vectors $\ol x(k)$ by choosing $x_1(k) =1,\ k \in \N$. It follows from \eqref{eq Leslie eigenpair A_k},
\be\label{eq x(k) for Leslie A_k}
\ol x(k)  = \left(1, \frac{s_1}{\lambda_k}, \ldots, \frac{s_1\ \cdots \ s_{k-1}}{\lambda_k^{k-1}}\right).
\ee
To find $\lambda_k$, we use the first row of $A_k$:
\be\label{eq s7 lambda_k}
b_1 + b_2 \frac{s_1}{\lambda_k} + \cdots + b_k\frac{s_1\ \cdots\  s_{k-1}}{\lambda_k^{k-1}} = \lambda_k. 
\ee
Let 
$$
p_k(\lambda) := \dfrac{b_1}{\lambda} + \sum_{i=2}^k \frac{b_i s_1\cdots s_{i-1} }{\lambda^{i}}.
$$ 
It follows from \eqref{eq s7 lambda_k} that $\lambda_k$ must be a solution to equation $p_k(\lambda) =1$. The following observations ensure the existence of such a solution: 
$p_k(\lambda)$ is continuous, $p_k(\lambda)> 1$ for sufficiently small $\lambda$, and $p_k(\lambda) < 1$ sufficiently large $\lambda$. This means that $\lambda_k$ is the Perron eigenvalue of $A_k$, and its corresponding Perron eigenvector is $\ol x(k)$. Note that since $A_k$ is a truncation of $A_m$ for $m>k$, the sequence $(\lambda_k)_{k \in \N}$ is increasing monotonically. 

Let now $A$ be the infinite Leslie matrix given by
\begin{equation}\label{Mat: Les_inf}
    A = \begin{pmatrix}
b_1 & b_2 & b_3 &  \ldots & b_i & b_{i+1} &\ldots\\
s_1 & 0& 0 &   \ldots & 0 & 0 &\ldots \\
0 &  s_2 & 0 &   \ldots & 0 & 0 &\ldots \\
\vdots & \vdots & \vdots & \vdots & \vdots & \vdots & \\
0 & 0& 0& \ldots & s_i & 0 & \ldots \\
\vdots & \vdots & \vdots & \vdots & \vdots & \vdots & \ddots\\
\end{pmatrix},
\end{equation} 
where $b_i,s_i$ are positive integers for $i \in \N$. For $k \in \N$, the matrix $A_k$ defined in \eqref{eq: Leslie_finite} is the $k \times k$ truncation of $A$. Observe that if $b_i = s_i = 1$ for all $i \in \mathbb{N}$, then the matrix $A$ generates a shift space known as a \textit{renewal shift} (see Example 7.16 in \cite{BezuglyiJorgensenKarpelSanadhya2025} and \cite{Raszeja2021}, \cite{BEFR2022}).

\begin{remark}
The renewal shift class \cite{Iommi2006} is well known in thermodynamic formalism for countable Markov shifts \cite{SarigTFCMS}. For example, several works concerning phase transitions in shift spaces have the renewal shift as an example \cite{Beltran_2021, BEFR2022, Raszeja2021, Sarig2001PT}. Its nature allows us to obtain explicit central objects, such as conformal measures and pressure. This is no different for generalized Bratteli diagrams: the path space of the stationary Bratteli diagram given by the renewal shift incidence matrix, its Perron value (it coincides with the Gurevich entropy \cite{Gurevichtopological1969, Gurevichshift1970}), left and right eigenvectors, and therefore its corresponding tail invariant measure \cite{BezuglyiJorgensenKarpelSanadhya2025}.
\end{remark}

For our purpose, we will work with the transpose $F = A^{T}$ and consider generalized Bratteli diagrams $B(F)$. 

Let $\lambda_0$ be the Perron eigenvalue for $A$ and $\ol \xi$ the corresponding positive eigenvector such that $A\ol\xi = \lambda_0 \ol\xi$. To find $\lambda_0$ and $\ol\xi = (\xi_i)$, we solve the following infinite system of equations:
\be\label{eq Leslie eigenpair}
\ba 
b_1 \xi_1 + \cdots + b_i \xi_i + \cdots = & \lambda_0 \xi_1\\
s_i \xi_i = & \lambda_0 \xi_{i+1}, \quad i = 1, 2, ...
\ea
\ee
Normalizing $\ol\xi$ so that $\xi_1 =1$, we obtain
\be\label{eq xi for Leslie}
\ol\xi = \left(1, \frac{s_1}{\lambda_0}, \ldots, \frac{s_1\ \cdots \ s_{i-1}}{\lambda_0^{i-1}}, \ldots\right).
\ee
We see from the first equation in 
\eqref{eq Leslie eigenpair} that $\lambda_0$ must satisfy equation $p(\lambda) =1$ where 
\be\label{eq s7 p(lambda)}
p(\lambda) := \dfrac{b_1}{\lambda}+ \sum_{i=2}^\infty \frac{b_i s_1\ \cdots\ s_{i-1} }{\lambda^{i}}.
\ee
The series \eqref{eq s7 p(lambda)} converges for 
\begin{equation}\label{eq: lambda_L}
    \lambda > L: =\limsup_{i\to \infty} \dfrac{b_{i+1}s_i}{b_i}.
\end{equation}
Note that we can choose a finite $\lambda$ if and only if the sequence $(s_i)_{i \in \N}$ is bounded and $\sup_i \frac{b_{i+1}}{b_i} < \infty$.

For such $\lambda$, the function $p(\lambda)$ is continuous in $\lambda$ and decreasing. Note that when $\lambda$ runs over the convergence interval, $p(\lambda)$ takes values as close to $1$ from above and below as needed. This, together with continuity of $p(\lambda)$, implies the existence of a solution to $p(\lambda) = 1 $. To illustrate this approach, we consider a particular case when
\be\label{eq bounded Leslie}
s_i = s,  \qquad m \leq b_i \leq M, \quad  i \in \N.
\ee
Hence, we can find the upper and lower bounds for $p(\lambda)$ as follows :
$$
\frac{m}{\lambda} \sum_{i = 0}^{\infty} \left(\frac{s}{\lambda}\right)^i \leq p(\lambda) \leq  \frac{M}{\lambda} \sum_{i = 0}^{\infty} \left(\frac{s}{\lambda}\right)^i.
$$
The series $\sum_{i = 0}^{\infty} \left(\frac{s}{\lambda}\right)^i$ converges for $\lambda >s$, and
$$
\frac{m}{\lambda -s} \leq p(\lambda) \leq \frac{M}{\lambda -s}.
$$
Note that, if $\lambda \in (m +s, M +s)$, then $p(\lambda)$ takes the values as close to $1$ from above and below as needed. Thus by continuity of $p(\lambda)$ there exists $\lambda'$ such that $p(\lambda') = 1$.

Now we are ready to deduce the following results. 

\begin{lemma}\label{Lem: Leslie_conv}
Let the matrices $A$ and $A_k$ be defined as above. Then $\lambda_k \to \lambda_0$ and $\ol x(k) \to \ol\xi$ as $k \to \infty$. Moreover, the sequence $(\lambda_k)$ is increasing and 
the sequence $(\ol x(k))$ is decreasing.  
\end{lemma}

\begin{proof} The statement follows from the continuity of $p_k$ for $k \in \N$, the fact that $p(\lambda) = \lim_{n\to \infty} p_k (\lambda_k)$ and using \eqref{eq x(k) for Leslie A_k}, \eqref{eq xi for Leslie}.
\end{proof}

\begin{prop}\label{Prop norm conv Leslie} 
Suppose that the entries of the infinite Leslie matrix $A$ satisfy conditions \eqref{eq bounded Leslie}. Then for $\lambda_0 > s$, $|| \ol \xi||_{\ell^1} < \infty$ and 
$$
\lim_{k\to \infty}|| \ol x(k)||_{\ell^1} = || \ol \xi||_{\ell^1}.
$$
\end{prop}

\begin{proof}
Indeed, for sufficiently large $k$, we have $\lambda_k > s$. 
Then we see that 
$$
\lim_{k\to \infty} \frac{1 - (\frac{s}{\lambda_k})^k}{1 - \frac{s}{\lambda_k} } = \frac{1}{1 - \frac{s}{\lambda_0}}.
$$
This gives the required result.
\end{proof}
Thus, by \Cref{Lem: Leslie_conv} and \Cref{Prop norm conv Leslie}, we obtain the following result.

\begin{corol} Let $A$ be an infinite Leslie matrix as defined in \eqref{Mat: Les_inf} such that the parameters $(s_i)_{i \in \N}$ and $(b_i)_{i \in \N}$ satisfy \eqref{eq bounded Leslie}. For $k \in \N$, let $A_k$ for be the $k \times k$ truncation of $A$ and set $F = A^T$ and $F_k = A_k^T$. Let $B(F)$ be the generalized Bratteli diagram with incidence matrix $F$ and $(B(F_k))_{k \in \N}$ be a sequence of standard subdiagrams of $B(F)$ corresponding to incidence matrices $(F_k)_{k \in \N}$. Then $\wh \mu_k \xrightarrow{w-[C]} \mu$ on $(X_B)$, where  
the measure $\mu$ and sequence of measures $(\mu_k)_{k \in \N}$ are as in \eqref{mu} and \eqref{mui}, respectively, and $\wh \mu_k$ is the measure extension of $\mu_k$ for $k \in \N$. 
\end{corol}

\begin{remark}
One can apply a similar analysis to a more general class of matrices of the form:
$$
C = \begin{pmatrix}
b_1 & b_2 & b_3 &  \ldots & b_i & b_{i+1} &\ldots\\
s_1 & c_1& 0 &   \ldots & 0 & 0 &\ldots \\
0 &  s_2 & c_2 &   \ldots & 0 & 0 &\ldots \\
\vdots & \vdots & \vdots & \vdots & \vdots & \vdots & \\
0 & 0& 0& \ldots & s_i & c_i & \ldots \\
\vdots & \vdots & \vdots & \vdots & \vdots & \vdots & \ddots\\
\end{pmatrix},
$$
where $(c_i)_{i \in \N}$ is a sequence of nonnegative integers. If all nonzero entries of $C$ are equal to 1, then this matrix describes the so-called \textit{lazy renewal shift}. For $k \in \N$, we denote by $C_k$ the $k \times k$-truncation of $C$.

For the matrix $C$ and matrices $(C_k)_{k \in \N}$, we can easily compute the right eigenvector $\ol\xi$ and $\ol x(k)$, respectively. We show the vector $\ol\xi$ below.
$$
\ol\xi = \left(1, \frac{s_1}{\lambda - c_1}, \ldots, \frac{s_1\ \cdots\ s_{i-1}}{(\lambda - c_1) \ \cdots \ (\lambda - c_{i-1})}, \ldots\right). 
$$
Here $\lambda$ is the Perron eigenvalue that satisfies the equation $p(\lambda) =1$ where
$$
p(\lambda) := \frac{b_1}{\lambda} + \lambda^{-1} 
\sum_{i\geq 2} \frac{b_is_1\ \cdots\ s_{i-1}}{(\lambda - c_1) \ \cdots \ (\lambda - c_{i-1})}.
$$

\end{remark}

\begin{remark}
We can also find a left eigenvector $\ol\eta = (\eta_i)_{i \in \N}$ for the infinite Leslie matrix $A$ given by \eqref{Mat: Les_inf}:
$$
\ol\eta = \left(1, \frac{\lambda - b_1}{s_1}, \frac{\lambda^2 - 
\lambda b_1 - b_2s_1}{s_1s_2}, \ldots, \frac{\lambda^{i-1} - \lambda^{i-2}b_1 - \lambda^{i-3}b_2s_1 - \cdots - b_{i-1}s_1\ \cdots \ s_{i-2}} {s_1\ \cdots\ s_{i-1}}, \ldots\right).
$$ Here $\lambda = \lambda_0$ is the Perron eigenvalue of $A$. Then the inner product $\langle \ol\xi, \ol \eta\rangle$ is determined by the series: 
$$
\ba 
\langle \ol\xi, \ol \eta\rangle = & 1 + \frac{\lambda - b_1}{\lambda} + \cdots + \frac{\lambda^{i-1} - \lambda^{i-2}b_1  - \cdots - b_{i-1}s_1\ \cdots \ s_{i-2}}{\lambda^{i-1}} + \ldots \\
= & 1 + \sum_{n=2}^{\infty} \left(1 - \frac{b_1}{\lambda} - \frac{b_2s_1}{\lambda^2} - \cdots - \frac{b_{i-1}s_1 \cdots s_{i-2}}{\lambda^{i-1}}\right). 
\ea
$$
It follows that if $b_i = b$ and $s_i =s$ for $i \in \N$, then 
$A$ is an equal column sum (ECS) matrix with $\lambda = b+s$. In this case, we can easily see that 
$\langle \ol\xi, \ol \eta\rangle < \infty$ and the matrix $A$ is \textit{positive recurrent}. Indeed,
$$
\ba
\langle \ol\xi, \ol \eta\rangle = & 1 + \sum_{i\geq 2} \left(1 - 
\frac{b}{\lambda} \sum_{j=0}^{i-2} (\frac{s}{\lambda})^j \right)\\
= &  1 + \sum_{i\geq 2} \left(1 - \frac{b}{\lambda} \cdot 
\frac{1- (\frac{s}{\lambda})^{i-1}}{1 - \frac{s}{\lambda}}\right)\\
= & 1 + \sum_{i\geq 2}\left(\frac{s}{\lambda}\right)^{i-1} < \infty. 
\ea
$$
\end{remark}

\subsection{Approximation by vertex subdiagrams} 
In this subsection, we focus on the following problem. 
Given a generalized Bratteli diagram $B$ and a probability tail
invariant measure $\mu$ on the path space $X_B$, can one find an
increasing sequence of vertex subdiagrams $(\ol B_k)_{k \in \N}$ such that
\be\label{eq meas appr by sbdgrms}
\mu \left(X_B \setminus \bigcup_{k=1}^{\infty} \wh X_{B_k}\right) = 0?
\ee
We first provide examples of generalized Bratteli diagrams for which relation \eqref{eq meas appr by sbdgrms} fails and where it holds.

We define a lower-triangular generalized Bratteli diagram $B_{\Delta}$ as follows (see \cite[section 8]{BezuglyiKarpelKwiatkowskiWata2024}):
For all $n \in \N_{0}$, identify $V_n$ with $\N$ and let
$B_{\Delta} = B_{\Delta}(F)$ be a stationary generalized Bratteli diagram where $F = (f_{j,i})$ is given by 
\[
f_{j,i} = 
\left\{
\begin{aligned}
& 1, \mbox{ for } i \leq j,\\
& 0, \mbox{ otherwise. }
\end{aligned}
\right.
\] 
In other words, the incidence matrix for $B_{\Delta}$ is a lower-triangular infinite matrix
\begin{equation}\label{eq: lo_T_matrix}
F = 
\begin{pmatrix}
1 & 0 & 0 & 0 &\ldots\\
1 & 1 & 0 & 0 &\ldots\\
1 & 1 & 1  & 0 &\ldots\\
1 & 1 & 1  & 1 & \ldots\\
\vdots & \vdots & \vdots & \vdots & \ddots
    \end{pmatrix}.
\end{equation}

In \cite{BezuglyiKarpelKwiatkowskiWata2024}, we described all ergodic probability tail invariant measures $\mu$ on the path space $X_{B_\Delta}$.
It was proven that  $B_\Delta$ has uncountably many ergodic probability measures $\{\mu_a : a \in (0,1)\}$ (see \cite[Theorem 8.7.]{BezuglyiKarpelKwiatkowskiWata2024}). For every $a \in (0,1)$, the measure $\mu_a$ is completely determined by the sequence of infinite vectors $(\ov{p}^{(n)})_{n \in \N_0} = (\ov{p}^{(n)}_i : i \in \N)$ (see \Cref{BKMS_measures=invlimits}), where
\[
\ov{p}^{(n)}_i = \bigg(\dfrac{a}{1+a}\bigg)^{i-1} \dfrac{1}{(1+a)^n}.
\] 

\begin{prop}\label{prop meas triangle BD}
Let the stationary Bratteli diagram $B_{\Delta} = B_{\Delta} (F)$ be the lower-triangular generalized given by $F$ as in \eqref{eq: lo_T_matrix}, and $\mu$ a probability tail invariant measure 
on $X_{B_\Delta}$. Denote by $(W_k)_{k \in \N}$ a nested increasing sequence of finite sets such that $\bigcup_{k=1}^{\infty} W_k = \N$. Let $(\ol B_k)_{k \in \N}$ be a stationary standard Bratteli subdiagram of $B_{\Delta}$ defined by the vertex set $W_k$. Then
$$
\mu \left(X_B \setminus \bigcup_{k=1}^{\infty} \wh X_{\ol B_k} \right) = 1.
$$
\end{prop}

\begin{proof}
Without loss of generality, for $k \in \N$, we set $W_k = \{1, \ldots, k\}$. Observe that

$(a)$ $X_{\ov{B}_k} = \wh{X}_{\ov{B}_k}$ for each $k \in \N$.

$(b)$ the set $X_{\ov{B}_k}$ (and therefore the set 
$\bigcup_{k=1}^{\infty} \wh X_{\ov{B}_k}$)
is countable for each $k \in \N$.

These facts prove the proposition. 
\end{proof}

\Cref{prop meas triangle BD} describes an atypical case. 
We consider generalized Bratteli diagrams of bounded size below, where the opposite situation is realized.

Let $B = B(F_n)$ 
be a generalized Bratteli diagram of bounded size determined by the parameters $(t_n, L_n)$, see \Cref{Def:BD_bdd_size}. We will show that there is a naturally defined sequence of subdiagrams of $B = B(F_n)$ that exhaust the path space $X_{B}$  in the sense of \eqref{eq meas appr by sbdgrms} for any probability tail invariant measure $\mu$ on $X_{B}$.

As usual, we identify the vertex set at each level with $\Z$. Since $B(F_n)$, is bounded size the incidence matrices $(F_n)_{n \in \N_0}$ are banded with the width of the band equal to $2t_n +1$ for $n \in \N$. Define the sequence of standard Bratteli diagrams $(\ol B_k)_{k \in \N}$ as follows. For fixed $k \in \N$, 
we define a set of vertices $W_n(k)$ as follows: Let $W_0(k) = \{-k, \ldots, k\}$, for $n \in \N$, set $W_n(k)$ as
\begin{equation}\label{eq: wnk}
    W_n(k) = \left\{-k - \sum_{i=0}^{n-1} t_i,\ldots,  k + 
    \sum_{i=0}^{n-1} t_i\right\}.
\end{equation}
Clearly, for $n \in \N_0$, $\bigcup_{k\in \N} W_n(k) = \Z $. 

The subdiagram $\ov{B}_k$ is defined as the restriction of the diagram $B$ to the sequence of vertices from the sets $\{W_n(k)\}_{n \in \N_0}$. Thus, for $k \in \N$, $X_{\ov{B}_k}$ denotes the path space of $\ov{B}_k$ and
$\wh{X}_{\ov{B}_k}$ is the extension of $X_{\ov{B}_k}$ under the tail equivalence relation, i.e., $\wh{X}_{\ov{B}_k} = \mathcal{R} (X_{\ov{B}_k})$. For $n \in \N_0$, 
we denote
\[
\wh{X}^{(n)}_{\ov{B}_k} = \{x = (x_i) \in \wh{X}_{\ov{B}_k} : s(x_i) \in W_{i}(k), i \geq n\}. 
\]
Thus, we have
\[
\wh{X}^{(n)}_{\ov{B}_k} = \underset{w \in W_n(k)}{\bigcup} X_w^{(n)}, \ \ \  k \in \N,\  n \in \N_0,
\] 
where $X_w^{(n)}$ is as in \Cref{Def:Kakutani-Rokhlin}. Let $\mu$ be a probability tail invariant measure on $X_B$. Then 
we have
\begin{equation}\label{eq:measure_whX}
    \mu(\wh{X}^{(n)}_{\ov{B}_k}) = \underset{w \in W_n(k)}{\sum} H_w^{(n)}p_w^{(n)}.
\end{equation} 
Here, for all $n \in \N$ and $w \in W_n(k)$, we have $p_w^{(n)} = \mu( [\ov{e}])$, where $r(\ov{e}) = w\in W_n(k) $ and 
$F_n^{T} \ov{p}^{(n+1)} =\ov{p}^{(n)}$. Note that if $w \in W_n(k)$ for some $n \in \N$, then $w \in W_i(k)$ for all $i>n$. 
It is easy to see that
$\wh{X}^{(n)}_{\ov{B}_k} \subset \wh{X}^{(n+1)}_{\ov{B}_k}$. We will also use the fact that for every $n\in \N_0$, $\underset{v \in V_n}{\sum} H_v^{(n)}p_v^{(n)} = 1$.

\begin{theorem}\label{thm: muX_B_k} Let $B= B(t_n)$ be a bounded size generalized Bratteli diagram. For $k\in \N$, let $\ov{B}_k$ be the restriction of the diagram $B$ to the sequence of vertices belonging 
to the sets $\{W_n(k)\}_{n \in \N_0}$ as defined in \eqref{eq: wnk}. Let $\mu$ be a probability tail invariant measure on $X_B$. Then
 \[
 \mu \, \big(X_B \setminus \underset{k \in \N}{\bigcup} \wh{X}_{\ov{B}_k} \big) = 0.
\] 
 \end{theorem}

\begin{proof} 
For any level $V_N$ and any $\varepsilon> 0$, we can find $L\in \N$ such that 
 \[
 \sum_{v=-L}^{L} \mu(X_v^{(N)}) = \sum_{v=-L}^{L} H_v^{(N)}p_v^{(N)} > 1-\varepsilon.
 \] Find $k_0 \in \N$ such that $[L,-L] \subset W_n(k_0)$. Thus, we get
\[
\mu(\wh{X}_{\ov{B}_{k_0}}) = \underset{w \in W_N(k)}{\sum} H_w^{(N)}p_w^{(N)}  > 1-\varepsilon.
\]
Hence, 
\[
\mu \, \big(X_B \setminus \underset{k \in \N}{\bigcup} \wh{X}_{\ov{B}_k} \big) < \mu (X_B)  - \mu(\wh{X}_{\ov{B}_{k_0}})  < \varepsilon. 
\] 
This completes the proof.
\end{proof}

\begin{remark} The statement of \Cref{thm: muX_B_k} holds for weaker assumptions on the choice of sequence of subdiagrams $\{\ov{B}_k\}_{k \in \N_0}$. Let $B$ be a generalized Bratteli diagram of bounded size defined by the parameters $(t_n)_{n\in \N_0}$. Suppose that for every $k \in \N$, 
the subdiagram $\ol B_k$, 
obtained by restricting $B$ to a sequence of vertices given by sets $W_n(k)$ for $n \in \N_0$, 
has the property: there exist infinitely many $n\in \N_0$ such that 
\[
 \bigg[-\sum_{i=0}^{n-1} t_i,\,\, \sum_{i=0}^{n-1} t_i\bigg] \subset  W_n(k). 
\] 
Then the statement of \Cref{thm: muX_B_k} holds.
    
\end{remark}

Let $B$ be a generalized Bratteli diagram and $\{\ov{B}_k\}_{k \in \N}$ be a sequence of vertex subdiagrams of $B$. For $k \in \N$, let $\ov{B}_k$ be the restriction of $B$ to a sequence of vertices $\{U_n(k) \subset V_n : n \in \N_0\}$. We assume that for every $n \in \N_0$, $k \in \N$, we have $U_n(k) \subset U_n(k+1)$ and 
$$
\underset{k\in \N}{\bigcup} U_n(k) = V_n.
$$ 

\begin{theorem} 
Let $\mu$ be a tail invariant probability measure on $X_B$. If  
\[
\mu \, \big(X_B \setminus \underset{k \in \N}{\bigcup} \wh{X}_{\ov{B}_k} \big) = 0,
\]
then
\begin{equation}\label{H_w,p_w_limit}
\lim_{n\to \infty} \lim_{k\to \infty}
   \sum_{w \in U_n(k)} H_w^{(n)}p_w^{(n)}  = 1.
\end{equation}
\end{theorem}

\begin{proof} Since $X_B \setminus \underset{k \in \N}{\bigcup} \wh{X}_{\ov{B}_k} = \underset{k \in \N}{ \bigcap} (X_B \setminus \wh{X}_{\ov{B}_k})$ and $(X_B \setminus \wh{X}_{\ov{B}_{k+1}}) \subset (X_B \setminus \wh{X}_{\ov{B}_k})$, we obtain
\begin{equation}\label{X_B_minusX_Bbar}
    \mu \, \big(X_B \setminus \underset{k \in \N}{\bigcup} \wh{X}_{\ov{B}_k} \big) = \underset{k \rightarrow \infty}{\textrm{lim}} \mu (X_B \setminus \wh{X}_{\ov{B}_k}). 
\end{equation} Now we use the fact that $\wh{X}_{\ov{B}_k} = \underset{n \in \N_0}{\bigcup} \wh{X}^{(n)}_{\ov{B}_k}$. Since this is an increasing sequence of sets, we get $X_B \setminus \wh{X}_{\ov{B}_k} = \underset{n \in \N_0}{\bigcap} (X_B \setminus \wh{X}^{(n)}_{\ov{B}_k})$. In other words,
\[
\mu (X_B \setminus \wh{X}_{\ov{B}_k}) =  \underset{n \rightarrow \infty}{\textrm{lim}} \mu (X_B \setminus \wh{X}^{(n)}_{\ov{B}_k}).
\] Thus by \eqref{X_B_minusX_Bbar}, we get

\begin{equation}\label{X_b_minus}
    0=\mu \, \big(X_B \setminus \underset{k \in \N}{\bigcup} \wh{X}_{\ov{B}_k} \big) = \underset{ \underset{k \rightarrow \infty}{ n \rightarrow \infty}}{\mathrm{lim}} \,\,\, \mu (X_B \setminus \wh{X}^{(n)}_{\ov{B}_k}).
\end{equation}
Observe that for $n \in \N_0$, $k \in \N$, we have $\wh{X}^{(n)}_{\ov{B}_k} \subset \underset{w \in U_n(k)}{\bigcup} X_w^{(n)}$. Hence $\mu(\wh{X}^{(n)}_{\ov{B}_k} ) \leq \underset{w \in U_n(k)}{\sum} H_w^{(n)}p_w^{(n)}$. This implies that
\[
\mu \, \big(X_B \setminus \wh{X}^{(n)}_{\ov{B}_k} \big) \geq 1 - \underset{w \in U_n(k)}{\sum} H_w^{(n)}p_w^{(n)}.
\] 
Thus, \eqref{H_w,p_w_limit} follows from \eqref{X_b_minus}.  

\end{proof}

\section{Generalized Bratteli diagrams with no probability tail invariant measures}\label{sect B(2^k)} 

In this section, we consider a class of generalized Bratteli diagrams of bounded size with unusual dynamics on the path space. It will be shown that every cylinder set is a wandering set under the Vershik map.
In particular, this implies that there is no probability tail invariant measure on the path space. These diagrams will be defined using a sequence of natural numbers $(t_n)_{n \in \N_0}$ (see \eqref{eq: t_n_eq} ). We denote them by $B = B(t_n)$ (see \Cref{def:GBD_t_n}). Before we discuss the properties of diagrams $B(t_n)$, we consider a particular case where $t_n = 2^n$ for a more transparent way to see the dynamics on the path space.

\begin{remark} Note that the notation $B=B(t_n)$ refers to a single generalized Bratteli diagram defined by the sequence $(t_n)_{n \in \N_0}$. A similar notation is applied for the diagram $B= B(2^n)$. 
This approach is consistent with our earlier notation $B(F_n)$ where the sequence of incidence matrices $(F_n)$ determined a Bratteli diagram.
\end{remark}

\subsection{The diagram $B=B(2^n)$}\label{ssect B(2n)} 

We define generalized Bratteli diagram $B=B(2^n) = B(V,E)$ (see \Cref{Fig:ContinVMap1}) as follows : 
\begin{itemize}
    \item let $V_n = \Z$, for every $n \in \N_0$,
    \item for every $n \in \N_0$, $i \in V_n$ and $j \in V_{n+1}$ we have $|r^{-1} (j)| = |s^{-1} (i)|=2$, 
\item for every $n \in \N_0$, $i \in V_n$, $s^{-1} (i) = \{e_{\ell}^i, e_r^i\}$, where $r(e_{\ell}^i) = i -2^n \in V_{n+1}$, $r(e_{r}^i) = i + 2^n \in V_{n+1}$ (notation $\ell$ and $r$ stand for ``left'' and ``right'').
\end{itemize} Note that this diagram is also considered in \Cref{Ex:OmegaisXB}. 

    \begin{figure}[ht]
\unitlength=1cm
\begin{tikzpicture}[scale=1, every node/.style={circle, draw, fill=black, minimum size=0.2cm, inner sep=0pt}]

\node (V11) at (2,7) {};
\node (V12) at (4,7) {};
\node (V13) at (6,7) {};
\node (V14) at (8,7) {};
\node (V15) at (10,7) {};

\node (V21) at (2,5) {};
\node (V22) at (4,5) {};
\node (V23) at (6,5) {};
\node (V24) at (8,5) {};
\node (V25) at (10,5) {};

\node (V31) at (2,3) {};
\node (V32) at (4,3) {};
\node (V33) at (6,3) {};
\node (V34) at (8,3) {};
\node (V35) at (10,3) {};

\node (V41) at (2,1) {};
\node (V42) at (4,1) {};
\node (V43) at (6,1) {};
\node (V44) at (8,1) {};
\node (V45) at (10,1) {};

\draw[thin] (V22) -- (V11);
\draw[thin] (V21) -- (V12);
\draw[thin] (V23) -- (V12);
\draw[thin] (V22) -- (V13);
\draw[thin] (V24) -- (V13);
\draw[thin] (V23) -- (V14);
\draw[thin] (V25) -- (V14);
\draw[thin] (V24) -- (V15);

\draw[thin] (V32) -- (V21);
\draw[thin] (V31) -- (V22);
\draw[thin] (V33) -- (V22);
\draw[thin] (V32) -- (V23);
\draw[thin] (V34) -- (V23);
\draw[thin] (V33) -- (V24);
\draw[thin] (V35) -- (V24);
\draw[thin] (V34) -- (V25);

\draw[thin] (V41) -- (V33);
\draw[thin] (V43) -- (V31);
\draw[thin] (V42) -- (V34);
\draw[thin] (V44) -- (V32);
\draw[thin] (V43) -- (V35);
\draw[thin] (V45) -- (V33);

\node[draw=none, fill=none] at (10.9,7) {$\ldots$};
\node[draw=none, fill=none] at (10.9,5) {$\ldots$};
\node[draw=none, fill=none] at (10.9,3) {$\ldots$};
\node[draw=none, fill=none] at (10.9,1) {$\ldots$};
\node[draw=none, fill=none] at (1,7) {$\ldots$};
\node[draw=none, fill=none] at (1,5) {$\ldots$};
\node[draw=none, fill=none] at (1,3) {$\ldots$};
\node[draw=none, fill=none] at (1,1) {$\ldots$};
\node[draw=none, fill=none] at (2,0.5) {$\vdots$};
\node[draw=none, fill=none] at (4,0.5) {$\vdots$};
\node[draw=none, fill=none] at (6,0.5) {$\vdots$};
\node[draw=none, fill=none] at (8,0.5) {$\vdots$};
\node[draw=none, fill=none] at (10,0.5) {$\vdots$};

\end{tikzpicture}

\caption{The diagram $B(2^n)$}\label{Fig:ContinVMap1}
\end{figure}
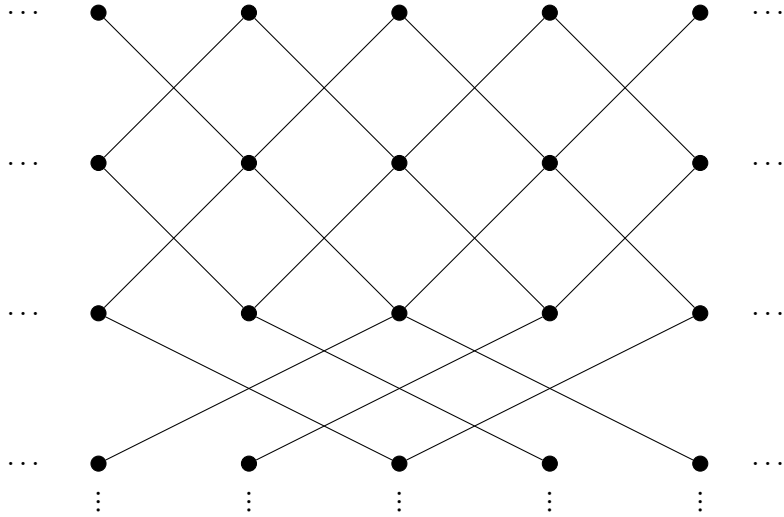

It is convenient to use the following notation: for $n \in \N_0$, $j \in V_{n+1}$, we write $r^{-1} (j) = \{f_{\ell}^j, f_r^j\}$, where $s(f_{\ell}^j) = j -2^n$ and $s(f_{r}^j) = j + 2^n$. Hence, $e^i_r = f_\ell^{i+2^n}$ and $e^i_\ell = f_r^{i-2^n}$ where the vertices $i \pm 2^n$ are in the level $V_{n+1}$.
This definition shows that $B(2^n)$ belongs to the class of generalized horizontally stationary Bratteli diagrams. 
Thus, $B(2^n) = B(F_n)$ is defined by the sequence of Toeplitz matrices $(F_n)_{n\in \N_0}$ where there are exactly two 1's in every row and every column, and the rest of the entries are $0$'s. Moreover, the distance between the nonzero entries is $2n +1$ for every row (column) of the matrix $F_n$. 
For example, $F_0$ and $F_1$ are the matrices shown below. 
$$F_0 = 
    \begin{pmatrix}
    \ddots & \vdots & \vdots & \vdots &\vdots &\vdots  &\vdots \\
     \ldots & 0 & 1 & \textbf 0 & 0  & 0 &\ldots\\
     \ldots & 1 & 0 & \textbf 1 & 0 & 0 &\ldots\\
     \ldots & \textbf 0 & \textbf 1 & \textbf 0  & \textbf 1 & \textbf 0 &\ldots\\
     \ldots & 0 & 0 & \textbf 1  & 0 & 1 &\ldots\\
     \vdots & \vdots & \vdots & \vdots& \vdots & \vdots & \ddots
    \end{pmatrix}, \quad
F_1 = 
    \begin{pmatrix}
    \ddots & \vdots & \vdots & \vdots & \vdots &\vdots &\vdots &\vdots &\vdots &\vdots\\
     \ldots &\ldots & 0 & 0 & 0 & \textbf 1 &\ldots &\ldots &\ldots &\ldots\\
     \ldots &\ldots & 1 & 0 & 0 & \textbf 0 & 1 & \ldots &\ldots &\ldots\\
     \ldots &\ldots &\textbf 0 & \textbf 1 & \textbf 0  & \textbf 0 & \textbf 0 & \textbf 1 &\ldots &\ldots\\
     \ldots &\ldots & 0 & 0 & 1  & \textbf 0 & 0 & 0 & 1 &\ldots \\
     \vdots & \vdots & \vdots & \vdots & \vdots & \vdots & \vdots & \vdots &\vdots & \ddots
    \end{pmatrix}. 
$$
(We use the bold font to indicate the 0-th row and 0-th column in the matrix.) In other words, $F_n = (f^{(n)}_{j,i})$ is a banded matrix where the $j$-th row has all entries equal to $0$, except for entries
$f^{(n)}_{j,j-2^n}$ and $f^{(n)}_{j,j+2^n}$ which are equal to $1$. 

We consider the generalized Bratteli diagram $B(2^n)$ equipped with the left-to-right order (denoted by $\omega$). This means that for every $n \in \N$ and $j \in V_{n}$, 
we order $r^{-1} (j)$ (which has two edges) from left to right. Fix $n \in \N$ and $j \in V_{n}$ and denote by $\ol{e}_{\mathrm{min}} (j)$ (respectively $\ol{e}_{\mathrm{max}} (j)$) the minimal finite path (respectively maximal) with $r(\ol{e}_{\mathrm{min}} (j)) = r(\ol{e}_{\mathrm{max}} (j)) = j$, between $V_0$ and $j\in V_n$. Note that
\[
s(\ol{e}_{\mathrm{min}} (j)) = j-  2^{n-1} - \cdots - 2 - 1 = j - (2^n-1) \in V_0
\] and 
\[
s(\ol{e}_{\mathrm{max}} (j)) = j+  2^{n-1} + \cdots + 2 + 1 = j + (2^n-1) \in V_0
\] 
In the remainder of the section, we will assume that $B(2^n)$ is equipped with order $\omega$.

Observe that in the space $X_{B(2^n}(\omega)$, the sets of infinite maximal and infinite minimal paths are countable. Indeed, for every $i \in V_0$, there is a unique infinite maximal and unique minimal path (denoted by
$\ol{x}_{\mathrm{max}} (i)$ and $\ol{x}_{\mathrm{min}} (i) $) with the source $i$. In other words, $\ol{x}_{\mathrm{max}} (i) = (x_n)_{n\in \N_0}$ is the path such that $x_n = e_\ell^{i - (2^n-1)}$ for every $n \in \N_0$ (see the definition of $B(2^n) = B(V,E)$ above). We can similarly define $\ol{x}_{\mathrm{min}} (i) $ for $i\in V_0$. Hence, $X_{\mathrm{max}} (\omega)$ and $X_{\mathrm{min}} (\omega)$ are countable sets. 

For $i \in V_0$, we denote by $X_i$ the set of all infinite paths with source $i$. Since at every level exactly two edges emerge from a vertex, $X_i$ can be identified with $\{0,1\}^\N$. Hence,
\[
X_{B(2^n)} = \bigsqcup_{i \in \Z} X_i =\bigsqcup_{i \in \Z} \,\, \{0,1\}^\N. 
\]
For ease of notation, we denote $X_{B(2^n)} = X_B$ and consider the dynamical system $(X_B,\varphi)$ where $\varphi: X_B \rightarrow X_B$ is the Vershik map defined by the order $\omega$. 

\begin{remark}
In \cite[Theorem 3.10]{BezuglyiJorgensenKarpelSanadhya2025}, 
we proved that the Vershik map $\varphi$ defined on $X_B$ (where $t_n = 2^n, n \in \N_0$) is discontinuous at every infinite extreme path.
It is worth noting that if the sequence $(t_n)$ is modified such that $t_0 = t_1 =1$ and $t_n = 2^{n-1}$ for $n \geq 2$, then $\varphi $ is a homeomorphism of $X_B$, see \cite[Theorem 3.12]{BezuglyiJorgensenKarpelSanadhya2025} for details.
\end{remark}

\begin{theorem}\label{Th: X_i_2}
For every $i \in V_0$, we have \[\varphi(X_i) =  X_{i+2}.\]    
\end{theorem} 

\begin{proof} Let $\ol x \in X_i$ be an infinite non-maximal path. Let $m \in \N$ be such that $x_m$ is the first edge in $\ol x$ which is not maximal in $r^{-1} (r(x_m))$. Observe that $s(\ol x_m) = i - (2^m-1) \in V_m$. For brevity, set $i - (2^m-1) = k$. Then $\ol e_{\mathrm{max}}(k)$ is the finite, unique maximal path between $V_0$ and $k = s(\ol x_m) = i - (2^m-1) \in V_m$. Thus, we can write $\ol x = (\ol e_{\mathrm{max}}(k),(x_n)_{n \geq m})$. 

Observe that $r(x_m) = i - (2^m-1) + 2^m = i+1$. Let $x'_m $ be the maximal edge in $r^{-1} (r(x_m))$. It follows from the definition of the Bratteli diagram $B(2^k)$ that
\[
s(x'_m) = i+1 + 2^m \in V_m.
\] For brevity, let $\ell = s(x'_m) =  i+1 + 2^m \in V_m$. We denote by $\ol e_{\mathrm{min}} (\ell)$ the finite, unique minimal path between $V_0$ and $\ell$. Observe that
\[
s(e_{\mathrm{min}} (\ell)) = s(x'_m) - 1-2- \cdots - 2^{m-1} = i+1 + 2^m - (2^m - 1) = i+2.
\] 
This shows that $s(\varphi (\ol x)) = i+2$. 

Now we consider the case where $m = 0$. In other words, $x_0$ is minimal in $r^{-1}(r(x_0))$. Since $r(x_0) = i+1$, it follows that $s(\varphi (\ol x)) = i+2$. This shows that $\varphi(X_i) \subset  X_{i+2}$. Similarly, working with $\varphi^{-1}$, we can show $\varphi^{-1}(X_{i+2}) \subset  X_{i}$, which completes the proof. 
    \end{proof} 

The following result is deduced directly from \Cref{Th: X_i_2}. 

\begin{corol}\label{cor: phi_wander} The following statements hold for the generalized Bratteli diagram $B(2^n)$:
\begin{enumerate}
    \item For every $i \in V_0$ the set $X_i$ is a $\varphi$-wandering set. In other words, $X_i \cap \varphi^k(X_i) = \emptyset$ for all $k \in \Z$. 
    \item If $X_{\mathrm{even}} = \underset{i\in 2\Z}{\bigsqcup}X_i$ and $X_{\mathrm{odd}} = \underset{i\in 2\Z+1}{\bigsqcup}X_i$, then $\varphi(X_{\mathrm{even}}) = X_{\mathrm{even}}$ and $\varphi(X_{\mathrm{odd}}) = X_{\mathrm{odd}}$. 
    \item There is no $\varphi$-invariant probability measure on $X_{B(2^n)}$.
    \item Every cylinder set in $X_{B(2^n)}$ is a 
       $\varphi$-wandering set. 
\end{enumerate} 
\end{corol} 

\begin{remark} It is worth mentioning that if we slightly change 
the diagram $B(2^n)$, then we get a different type of dynamics on the path space. More precisely, take $t'_0 = t'_1 =1$ and $t'_n = 2^{n-1}$. Then, as proved in 
 \cite[Example 3.13]{BezuglyiJorgensenKarpelSanadhya2025}, $\varphi(X_i) = X_i$, and the Vershik map is a homeomorphism of the path space. In this case, there are no wandering cylinder sets. 
\end{remark}

\subsection{Bratteli diagram $B(t_n)$}

As mentioned above, the definition of the Bratteli diagram discussed in  \Cref{ssect B(2n)} can be extended to a wider family
of Bratteli diagrams. To this end, we consider a sequence of natural numbers $\{t_{n}\}_{n\in \N_0}$ satisfying the following conditions:
\begin{equation}\label{eq: t_n_eq}
    t_0 >0,\hspace{5mm} \mathrm{and} \hspace{5mm} t_n > \sum_{i=0}^{n-1} t_i \hspace{5mm} \mathrm{for \,\,every} \hspace{2mm} n \in \N.
\end{equation} Observe that if $\{t_{n}\}_{n\in \N_0}$ satisfies \eqref{eq: t_n_eq} then for every $n \in \N$, $t_n \geq 2^n$.  We now consider the following generalization of the $B(2^n)$ diagram.  

\begin{definition}\label{def:GBD_t_n}
For a fixed sequence $\{t_{n}\}_{n\in \N}$ as in \eqref{eq: t_n_eq}, we define the generalized Bratteli diagram $B(t_n)$ as follows: 
\begin{itemize}
    \item Let $V_n = \Z$, for every $n \in \N_0$. 

\item For every $n \in \N_0$, $i \in V_n$, $s^{-1} (i) = \{e_{-i}^n, e_{+i}^n\}$, where $r(e_{-i}^n) = i -t_n$, $r(e_{+i}^n) = i + t_n$.
\end{itemize}
    
\end{definition}

We consider the generalized Bratteli diagram $B(t_n)$ with the left-to-right order (denoted by $\omega$). In other words, for every $n\in \N_0$ and $j \in V_{n+1}$, the set $r^{-1}(j)$ consists of two edges $e(j-t_n, j)$ and $e(i+t_n, j)$ where $j \pm t_n \in V_{n}$. We label these edges with 0 and 1 from left to right.

Fix $n \in \N$ and $j \in V_n$; denote by $\ol{e}_{\mathrm{min}} (j)$ ($\ol{e}_{\mathrm{max}} (j)$, respectively) the minimal finite path (maximal finite path, respectively), where 
$r(\ol{e}_{\mathrm{min}} (j)) = r(\ol{e}_{\mathrm{max}} (j)) = j$. Observe that for $j \in V_n$,
\[
s(\ol{e}_{\mathrm{min}} (j)) = j + \sum_{i=0}^{n} t_i \in V_0
\] and 
\[
s(\ol{e}_{\mathrm{max}} (j)) = j- \sum_{i=0}^{n} t_i \in V_0
\] 
On the other hand, if we fix $j \in V_0$ and consider the infinite maximal (minimal, respectively) path with source at $j$, then it consists of the left (right, respectively) outgoing edges. Note that the maximal path goes through the vertices
\[j \in V_0,\ \ j-t_0 \in V_1,\ \ j-t_0-t_1 \in V_2, \ldots, j- \sum_{i=0}^n j-t_i \in V_{n+1}\] 
and so forth. Similarly, the minimal path goes through
\[j \in V_0,\ \ j+t_0 \in V_1,\ \ j+t_0+t_1 \in V_2, \ldots, j+ \sum_{i=0}^n j-t_i \in V_{n+1}\] 
and so forth. 

As in the case of $X_{B(2^k)}$, there are countably many infinite maximal and minimal paths in the path space of $B(t_n)$: for every $i \in V_0$, there exists a uniquely determined pair of extreme paths with a source as $i$. Thus, the left-to-right order $\omega$ on the path space of $B(t_n)$ gives rise to a Vershik map $\varphi_{B(t_n)}: X_{B(t_n)} \rightarrow X_{B(t_n)}$. As in the case of $X_{B(2^k)}$ diagrams, the Vershik map defined by using left-to-right order is also discontinuous. As mentioned above, it follows from an argument similar to \cite[Theorem 3.10]{BezuglyiJorgensenKarpelSanadhya2025}.

Since we have fixed the sequence $\{t_n\}_{n\in \N}$, for brevity in the rest of the section, we will denote $X_{B(t_n)} $ by $X_{B} $ and $\varphi_{B(t_n)}$ by $\varphi_{B}$. For $i \in V_0$, let $[i]$ denote the cylinder set consisting of all infinite paths with the source $i$. Similarly, for $j\in V_n$ and $i \in V_m, m <n$, we denote by $[i, j]$ the cylinder set defined by the finite path $\ol e(i,j)$ (assuming that it is unique). In this notation, $[i]$ is partitioned into two cylinder sets $[i,i+t_0]$ and $[i,i-t_0]$ where $i\pm t_0 \in V_1$. 

Our first observation is that 
\be\label{eq s7 i, i_t_0}
\varphi_{B}([i,i+t_0])  = [i+2t_0, i+t_0] 
\ee
where $i$ and $i+2 t_0$ are in $V_0$ and $i + t_0 \in V_1$. 

We set $P_0 = t_0$ and for $ n \in \N$, let $P_n := t_n - \sum_{i=0}^{n-1} t_i$, $n \in \N$. The, for $n\in \N_0$, 
\[
P_{n+1} - P_n = t_{n+1} - \sum_{i=0}^{n} t_i - \left( t_n - \sum_{i=0}^{n-1} t_i \right) = t_{n+1} - 2 t_n.
\]
Thus, the sequence $\{P_n\}_{n\in \N_0}$ increases if and only if $t_{n+1} - 2 t_n >0$ for every $n \in \N_0$. We have already considered the extreme case where $t_n = 2^n$ for $n \in \N_0$ and 
$P_n =0$. In what follows, we assume that $t_n > 2^n$ for $n \in \N_0$.  

For $i\in V_0$, we consider the cylinder set $[i,i-t_0]$. Note  that
\begin{equation}\label{eq:i,i-t_0}
    [i,i-t_0] = \{\ol{e}_{\mathrm{max}} (i)\}\cup  \bigsqcup_{n\geq 1} C_n(i)
\end{equation}
where the cylinder set 
\begin{equation}\label{eq:i,C_n(i)}
    C_n(i) = [x_0,x_1,\ldots,x_{n-1},y_n]
\end{equation}
is such that $x_0 = e^0_{-i}$ and for $k \in \N$, $x_k = e^k_{-(r(x_{k-1}))}$, and $y_k = e^k_{+(r(x_{k-1}))}$. 

\begin{theorem}\label{thm:varphi_i_i-t0} 
For $i \in V_0$ and the cylinder set $[i,i-t_0]$ as in
\eqref{eq:i,i-t_0}, we have
\[
\varphi_B([i,i-t_0]) = \{\ol{e}_{\mathrm{min}} (i)\} \cup \bigsqcup_{n\geq 1} D_n(i)
\] where, for $n\in \N$,
\[
D_n(i) = [f_0,f_1,\ldots,f_{n-1},z_n].
\] 
Here $r(z_n) = r(y_n)$ and $z_n$ is the successor of $y_n$, $s(z_n) = r(f_{n-1})$, and $(f_0,f_1,\ldots,f_{n-1})$ is the minimal path connecting the vertex $i+2P_n \in V_0$ and $s(z_n)$. 
    
\end{theorem}

\begin{proof}
We need to show $\varphi_B(C_n(i)) = D_n(i)$ for $n \in \N$. Let $\overline{x} \in C_n$. Note that $y_n$ is the first nonmaximal edge in $\overline{x}$, therefore $\varphi_B(\overline{x})$ is defined by $z_n$ (the successor of $y_n$ in $r^{-1}(r(yn))$) and the minimal path connecting $s(z_n)$ to $V_0$. It follows from the definition of $B(t_n)$ that
\[
s(y_n) = r(x_{n-1}) = i - \sum_{k=0}^{n-1} t_k,\hspace{6mm} r(y_n) = i - \sum_{k=0}^{n-1} t_k + t_n = r(z_n).
\]
and 
\[
s(z_n) = i - \sum_{k=0}^{n-1} t_k + 2t_n.
\] 
For $j \in V_0$, let $\ol{e}_{\mathrm{min}}^{(n)} (j)$ be the minimal path connecting $j$ and $s(z_n)$. Then we find that 
\[
j = i - \sum_{k=0}^{n-1} t_k + 2t_n - \sum_{m=0}^{n-1} t_m = i + 2\left( t_n - \sum_{k=0}^{n-1} t_k\right) = i+2P_n.
\] 
This shows that the minimal path connecting the vertex $s(z_n)$ and $V_0$, that is, $(f_0,f_1,\ldots,f_{n-1})$, has the source in $i+ 2P_n$. This proves that $\varphi(C_n(i)) = D_n(i)$ as needed.
\end{proof}

\begin{corol} \label{s8 cor1}
The generalized Bratteli diagrams $B(t_n)$, defined by a sequence $(t_n)_{n\in\N_0}$ as in \eqref{eq: t_n_eq}, has no probability tail invariant measure. 
\end{corol}

\begin{proof} We show that, for every $i \in V_0$, the cylinder set 
$[i]$ is wandering, that is
\begin{equation}\label{eq:i_disjoint}
    [i] \cap \varphi_B^n([i]) = \emptyset, \qquad n \in \N.
\end{equation}
To see this, we note that $[i] = [i,i-t_0] \bigsqcup [i,i+t_0]$ and
by \eqref{eq s7 i, i_t_0}
\begin{equation}\label{eq:i+t_wander}
    \varphi_{B}([i,i+t_0])  = [i+2t_0, i+t_0] \subset [i+2P_0].
\end{equation}  This implies $[i,i+t_0] \cap \varphi_{B}([i,i+t_0]) = \emptyset$. By \Cref{thm:varphi_i_i-t0}, it follows
\begin{equation}\label{eq:i-t_wander}
    \varphi_B([i,i-t_0]\setminus \{\ol{e}_{\mathrm{max}} (i)\}) = \bigcup_{n\in \N} [i+2P_n],\,\,\mathrm{and}\,\, \varphi_B(\ol{e}_{\mathrm{max}} (i)) = \ol{e}_{\mathrm{min}} (i) \notin [i,i-t_0]. 
\end{equation}
Thus, we have shown $[i] \cap \varphi_B([i]) = \emptyset$. To see that $\varphi^2_B[i]$ and $[i]$ are disjoint, we apply the Vershik map to cylinder sets of the form $[i+P_k]$ for $k\geq 0$. For $k\in\N_0$, using an argument similar to that used above, it can be shown that $\varphi_B([i+2P_k])$ is the union of cylinder sets determined by vertices in $V_0$, each of which is indexed by a natural number greater than $[i+2P_k]$. This proves that $[i] \cap \varphi^2_B([i]) = \emptyset$. A simple induction argument proves \eqref{eq:i_disjoint}. Now for $i \in \Z$, consider the set
\[
A_i = \{\overline{x} \in X_B : s(\overline{x}) \geq i\}.
\]
It follows from \eqref{eq:i_disjoint} that $\varphi_B (A_i) \subsetneqq A_i$ for every $i \in \Z$. Hence, the generalized Bratteli diagram defined by \eqref{eq: t_n_eq} has no probability tail invariant measure.
\end{proof}

\begin{corol} Every cylinder set in the generalized Bratteli diagrams $B(t_n)$ defined by a sequence $(t_n)_{n\in\N_0}$ as in \eqref{eq: t_n_eq} is a wandering set.
    
\end{corol} 

\begin{proof} Observe that any cylinder set in $B(t_n)$ is a subset of either $[i,i+t_0]$ or $[i,i-t_0]$ for some $i \in V_0$. Note that equations \eqref{eq:i+t_wander} and \eqref{eq:i-t_wander} imply that both sets are wandering. Hence, all their subsets are wandering.
\end{proof}

At the end of this section, we consider the case where $(t_n)$ is an arbitrary sequence of natural numbers. Then we can define the Bratteli diagram $B(t_n)$ as above, and make it an ordered Bratteli diagram by taking the left-to-right order. Let $\varphi_B= \varphi_{B(t_n)}$ be the Vershik map on the path space $X_{B(t_n)}$. It is well defined as a Borel automorphism.

Denote $L_0 = 2 t_0$ and
\begin{equation}\label{eq:def_Ln}
   L_n = 2(t_n - \sum_{k=0}^{n-1}t_k), \quad n \in \N.
\end{equation}
We note that $L_n$ can be positive or negative. 

\begin{theorem}\label{Thm:wand_1} Let $\{t_{n}\}_{n\in \N_0}$ be a sequence of natural numbers that defines the generalized Bratteli diagram $B(t_n)$ as in Definition \ref{def:GBD_t_n}. 
Then for every $i \in V_0$, the cylinder set $[i]$ is $\varphi_{B(t_n)}$-wandering if for every finite set $s_1, \ldots, 
s_n \in \N_0$,  
$$
L_{s_1}+\cdots+L_{s_n} \neq 0.
$$
\end{theorem}

\begin{proof} The proof repeats the approach used in Theorem \ref{thm:varphi_i_i-t0}, so we skip some details. 

We describe the forward $\varphi_B$-orbit of $[i]$. As mentioned above, $[i] = [i,i-t_0] \bigsqcup [i,i+t_0]$ and $\varphi_{B}([i,i+t_0])  = [i+2t_0, i+t_0]$. Consider the forward orbit of $[i,i-t_0]$. By \eqref{eq:i,i-t_0}, we have
$$
[i,i-t_0] = \{\ol{e}_{\mathrm{max}} (i) \}\cup  \bigsqcup_{n\geq 1} C_n(i)
$$ 
where $C_n(i) = C_n(i) = [x_0,x_1,\ldots,x_{n-1},y_n]$ is as in \eqref{eq:i,C_n(i)}. It can be shown that $\varphi_{B}(C_n(i)) = 
D_n(i)$ where 
\[
D_n(i) = [f_0,f_1,\ldots,f_{n-1},z_n].
\] 
Here for $n\in \N$, $r(z_n) = r(y_n)$ and $z_n$ is the successor of $y_n$, $s(z_n) = r(f_{n-1})$ and $(f_0,f_1,\ldots,f_{n-1})$ is the minimal path connecting vertex $s(f_0)$ at level $V_0$ and $s(z_n)$. We claim that $s(f_0) = i+L_n$. To see this, observe that
\[
s(z_n) = i - \sum_{k=0}^{n-1} t_k.
\]
This implies
\[
r(y_n) =r(z_n) = \left(i - \sum_{k=0}^{n-1} t_k\right) + t_n.
\]
Hence, the source of $z_n$ is given by
\[
s(z_n) = \left(i - \sum_{k=0}^{n-1} t_k\right) + 2t_n.
\] 
Thus, the source of $f_0$ is given by
\[
s(f_0) = s(z_n) - \sum_{k=0}^{n-1} t_k = i+ 2\left(t_n - \sum_{k=0}^{n-1} t_k\right) = i + L_n
\] 
The vertex $s(f_0) = i+ L_n$ can be to the right of $i$ or to the left of $i$ depending on the sign of $L_n$. 

This result and the fact that $\varphi_{B}([i,i+t_0])  = [i+2t_0, i+t_0]$ imply that
\[
\varphi_B([i]) \subseteq \bigcup_{\ell_1=0}^{\infty} [i+L_{s_1}],
\] 
where $L_{s_1}$ is as in \eqref{eq:def_Ln}. It follows 
\[
\varphi_B^n([i]) \subseteq \bigcup_{s_1=0}^{\infty} \ldots \bigcup_{s_n=0}^{\infty}[i+L_{s_1}+\cdots+L_{s_n}].
\] 
Hence if $L_{s_1}+\cdots+L_{s_n} \neq 0$, then $[i] \cap \varphi_B^n([i]) = \emptyset$ and the result follows. 
\end{proof} 

The following corollary follows immediately from \Cref{Thm:wand_1}.

\begin{corol} \label{s8 cor2}
For the generalized Bratteli diagrams $B(t_n)$ defined by the sequence $(t_n)_{n\in\N_0}$ as in \eqref{eq: t_n_eq}, we have 

\begin{enumerate}
    \item If for every finite collection $s_1,\ldots,s_n \in \N_0$, the sum $L_{s_1}+\cdots+L_{s_n} \neq 0$, then there is no probability tail invariant measure on $X_{B(t_n)}$. 
    \item There is no $\varphi_{B(t_n)}$ dense orbit in $X_{B(t_n)}$. 
\end{enumerate}

\end{corol}

\section{Path space of a generalized Bratteli diagram}
\label{ssect path space}

\subsection{Properties of path spaces}
In this subsection, we discuss the path space of generalized Bratteli diagrams and their subdiagrams. We also consider when a given closed subset of the path space of a generalized Bratteli diagram corresponds to the path space of a Bratteli subdiagram.

Let $B = B(F_n)$ be a generalized Bratteli diagram. For $n \in \N$, we identify the set $V_n$ with $\Z$ for convenience. 

\begin{definition}\label{def_Om_Y} Let $B = B(F_n)$ be a generalized Bratteli diagram. We denote by $\Omega \subset X_B$ the set of all paths $\ov y = (y_n) \in X_B$ such that $\{|s(y_n)|\}_{n \in \N_0}$ is an unbounded sequence. In other words,
$$
\Omega = \{\ov x = (x_n) \in X_B\ | \ \forall M \in \N, \; \exists n \in \N_0, \; \mbox{ such\ that} \ |s(x_n)| > M\}.
$$
We set $Y := X_B \setminus \Omega$ which is the set of all infinite paths $\ov y = (y_n)$ such that $\{|s(y_n)|\}_{n \in \N_0}$ is bounded:
$$
\ov y= (y_n) \in Y \ \Longleftrightarrow\  \exists C \in \N \; 
\mbox{\ such\ that}\ \forall n \in \mathbb{N}_0 \; \ |s(y_n)| \leq C.
$$
\end{definition}

\begin{lemma}\label{lem s4 Y and Omega}
The path space $X_B$ is the disjoint union of $Y$ and $\Omega$. 
The set $Y$ is an $\mathcal{R}$-invariant $F_\sigma$-subset of $X_B$. Hence, the set $\Omega$ is an $\mathcal{R}$-invariant $G_{\delta}$-subset of $X_B$.
\end{lemma}
 
\begin{proof}
Note that     
\begin{equation*}
    Y = \bigcup_{C \in \mathbb{N}} Y_C,
\end{equation*}
where
$$
Y_C = \{\ov y= (y_n) \in X_B \ : \ \forall n \in \mathbb{N}_0, \; | s(y_n) |\leq C\}.
$$
For every $C \in \mathbb{N}$, the set $Y_C$ is closed. Indeed, pick a sequence $\ov y_k = (y^{(k)}_i) \in Y_C$ that converges to $\ov y = (y_i)$ as $k$ tends to infinity. Hence, for any $n$, there exists some sufficiently large $k$ such that $s(y^{(k)}_i) = s(y_i)$ for $i = 0, \ldots, n$. It follows that $|s(y_n)| < C$ for all $n \in N_0$, that is, $\ov y \in Y_C$, and $Y$ is an $F_{\sigma}$-set.

Let $\ov y = (y_i) \in Y$ and $\ov z \mathcal{R} \ov y$ for some $\ov z = (z_i) \in X_B$. Since $\ov y \in Y$, there is $C = C(\ov y)$ such that $| s(y_i) |\leq C$ for all $i$.
Since $\ov z \mathcal{R} \ov y$, there exists $N$ such that $z_i = y_i$ for all $i > N$. Hence $|z_i| < C$ for all $i > N$. Since $\{z_i\}_{i \leq N}$ is a finite set, there is $C_1$ such that $|z_i| < C_1$ for all $i \in \N_0$. Thus, $\ov z \in Y$ and $Y$ is $\mathcal{R}$-invariant. Observe that
$$
\Omega = \bigcap_{M \in \mathbb{N}} \Omega_M,
$$
where
$$
\Omega_M = \{\ov x= (x_n) \in X_B \ :\  \exists n \in \N_0 \; \mbox{ s.t. } |s(x_n)| > M\}. 
$$ 
It follows from the definition of $\Omega_M$ that $\Omega_M = X_B \setminus Y_M$. Hence $\Omega_M$ is closed and $\Omega$ is a $G_{\delta}$-subset of $X_B$. The $\mathcal{R}$-invariance of $\Omega$ follows from the $\mathcal{R}$-invariance of $Y$.
\end{proof}

The following examples demonstrate various possibilities for the sets $\Omega$ and $Y$. In the case where $B$ has no ``vertical'' paths, the set $\Omega$ can be equal to the whole path space $X_B$, see \Cref{Ex:OmegaisXB}. In \Cref{Ex:Omegaisempty}, we show that the set $\Omega$ can be empty. In \Cref{Ex:OmegaYdense}, both sets $Y$ and $\Omega$ are dense in $X_B$.

\begin{example}[$\Omega = X_B$]\label{Ex:OmegaisXB} Consider the generalized Bratteli diagram $B(2^n)$ defined in \Cref{ssect B(2n)} (see \Cref{Fig:ContinVMap1}). It is a bounded size diagram (see \Cref{Def:BD_bdd_size}) with parameters $t_n = 2^n$ and $L_n = 2$. Note that all infinite paths are unbounded. In other words, we have $X_B = \Omega$. Recall that this diagram has interesting properties, as discussed in \Cref{ssect B(2n)}. 

\end{example}

In the \Cref{Ex:Omegaisempty} below we assume that for every $i \in V_n$, $n \in \N_0$, we have $f^{(n)}_{i,i} \neq 0$. In other words, for every $i \in V_0$, the diagram $B$ has a ``vertical'' path that passes through the vertices $i \in V_n$, $n \in \mathbb{N}$.

\begin{example}[$Y = X_B$]\label{Ex:Omegaisempty} Consider the diagram given in \Cref{Fig:Z-DIOExample}. It is a generalized Bratteli diagram that contains infinitely many odometers (see \cite[Section 5]{BezuglyiKarpelKwiatkowski2024}), 
where $V_n = \mathbb{Z}$ for all $n$. The diagram consists of infinitely many odometers 
connected with neighboring odometers by single edges. It is easy to see that all the paths in $X_B$ eventually become vertical; thus, we have $Y = X_B$. 

\begin{figure}[hbt!]
\centering
\begin{tikzpicture}[scale=0.9,
    vertex/.style={circle, draw, fill=black, minimum size=0.2cm, inner sep=0pt},
    every edge/.style={thin, black}]

\node[vertex] (V11) at (1,5) {};
\node[vertex] (V12) at (3,5) {};
\node[vertex] (V13) at (5,5) {};
\node[vertex] (V14) at (7,5) {};
\node[vertex] (V15) at (9,5) {};

\node[vertex] (V21) at (1,3) {};
\node[vertex] (V22) at (3,3) {};
\node[vertex] (V23) at (5,3) {};
\node[vertex] (V24) at (7,3) {};
\node[vertex] (V25) at (9,3) {};

\node[vertex] (V31) at (1,1) {};
\node[vertex] (V32) at (3,1) {};
\node[vertex] (V33) at (5,1) {};
\node[vertex] (V34) at (7,1) {};
\node[vertex] (V35) at (9,1) {};


\draw (V21) to [out=130, in=230, looseness=0.9] (V11);
\draw (V21) to [out=110, in=250, looseness=0.8] (V11);
\draw (V21) to [out=70, in=280, looseness=0.8] (V11);
\draw (V21) to [out=50, in=310, looseness=0.9] (V11);
\draw (V22) -- (V11);

\draw (V22) to [out=120, in=240, looseness=0.8] (V12);
\draw (V22) to [out=60, in=300, looseness=0.8] (V12);
\draw (V22) -- (V12);
\draw (V23) -- (V12);

\draw (V23) to [out=120, in=240, looseness=0.8] (V13);
\draw (V23) to [out=60, in=300, looseness=0.8] (V13);
\draw (V23) -- (V14);

\draw (V24) to [out=120, in=240, looseness=0.8] (V14);
\draw (V24) to [out=60, in=300, looseness=0.8] (V14);
\draw (V24) -- (V14);

\draw (V25) to [out=130, in=230, looseness=0.9] (V15);
\draw (V25) to [out=110, in=250, looseness=0.8] (V15);
\draw (V25) to [out=70, in=280, looseness=0.8] (V15);
\draw (V25) to [out=50, in=310, looseness=0.9] (V15);
\draw (V24) -- (V15);


\draw (V31) to [out=130, in=230, looseness=0.9] (V21);
\draw (V31) to [out=110, in=250, looseness=0.8] (V21);
\draw (V31) to [out=70, in=280, looseness=0.8] (V21);
\draw (V31) to [out=50, in=310, looseness=0.9] (V21);
\draw (V32) -- (V21);

\draw (V32) to [out=120, in=240, looseness=0.8] (V22);
\draw (V32) to [out=60, in=300, looseness=0.8] (V22);
\draw (V32) -- (V22);
\draw (V33) -- (V22);

\draw (V33) to [out=120, in=240, looseness=0.8] (V23);
\draw (V33) to [out=60, in=300, looseness=0.8] (V23);
\draw (V33) -- (V24);

\draw (V34) to [out=120, in=240, looseness=0.8] (V24);
\draw (V34) to [out=60, in=300, looseness=0.8] (V24);
\draw (V34) -- (V24);

\draw (V35) to [out=130, in=230, looseness=0.9] (V25);
\draw (V35) to [out=110, in=250, looseness=0.8] (V25);
\draw (V35) to [out=70, in=280, looseness=0.8] (V25);
\draw (V35) to [out=50, in=310, looseness=0.9] (V25);
\draw (V34) -- (V25);

\node[draw=none, fill=none] at (9.9,5) {$\ldots$};
\node[draw=none, fill=none] at (9.9,3) {$\ldots$};
\node[draw=none, fill=none] at (9.9,1) {$\ldots$};

\node[draw=none, fill=none] at (0,5) {$\ldots$};
\node[draw=none, fill=none] at (0,3) {$\ldots$};
\node[draw=none, fill=none] at (0,1) {$\ldots$};

\node[draw=none, fill=none] at (1,0.1) {$\vdots$};
\node[draw=none, fill=none] at (3,0.1) {$\vdots$};
\node[draw=none, fill=none] at (5,0.1) {$\vdots$};
\node[draw=none, fill=none] at (7,0.1) {$\vdots$};
\node[draw=none, fill=none] at (9,0.1) {$\vdots$};

\end{tikzpicture}
\caption{$Y = X_B$.}\label{Fig:Z-DIOExample}
\end{figure}

\end{example}

\begin{example}[$\Omega$ and $Y$ are dense in $X_B$] \label{Ex:OmegaYdense} Consider the diagram given in \Cref{Fig:ERSECSnotHS}. It is a stationary, horizontally stationary (see \Cref{Def: Horz}), and bounded size Bratteli diagrams.
Every finite path can be prolonged in two ways: either to a path from $\Omega$ or to a path from $Y$. The sets $\bigcup Z_w^{+}$ and $\bigcup Z_w^{-}$ are $\mathcal{R}$-invariant dense subsets of $X_B$, where we set
   $$
Z_w^+ = \left\{x = (x_n) \in X_B : s(x_0) \geq w \mbox{ and } 
r(x_n) \geq w + \sum_{i = 0}^{n} t_i \mbox{ for all } n \in 
\mathbb{N}_0\right\}.
$$
and
$$
Z_w^- = \left\{x = (x_n) \in X_B : s(x_0) \leq w \mbox{ and } 
r(x_n) \leq w - \sum_{i = 0}^{n} t_i \mbox{ for all } n \in 
\mathbb{N}_0\right\}.
$$ 
Recall that $(t_i)_{i \in \N_0}$ are the parameters of the bounded size diagram (see \Cref{Def:BD_bdd_size}).
Note that for bounded size diagrams, the slanted sets 
$Z_w^{+}$ and $Z_w^{-}$ are closed and $\mathcal R$-invariant (see \cite[Lemma 5.3.]{BezuglyiJorgensenKarpelSanadhya2025}) 

    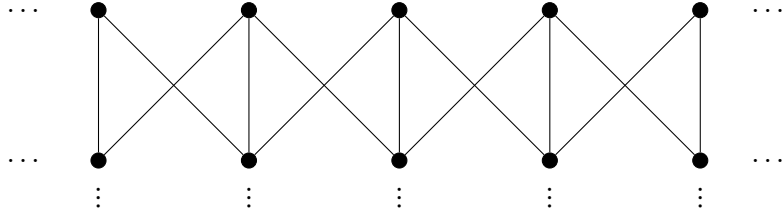
\begin{figure}
\unitlength=0,7cm

\begin{tikzpicture}[scale=1,
    vertex/.style={circle, draw, fill=black, minimum size=0.2cm, inner sep=0pt},
    every edge/.style={thin, black}]

\node[vertex] (V11) at (2,3) {};
\node[vertex] (V12) at (4,3) {};
\node[vertex] (V13) at (6,3) {};
\node[vertex] (V14) at (8,3) {};
\node[vertex] (V15) at (10,3) {};

\node[vertex] (V21) at (2,1) {};
\node[vertex] (V22) at (4,1) {};
\node[vertex] (V23) at (6,1) {};
\node[vertex] (V24) at (8,1) {};
\node[vertex] (V25) at (10,1) {};

\draw (V21) -- (V11);
\draw (V22) -- (V11);
\draw (V21) -- (V12);
\draw (V22) -- (V12);
\draw (V22) -- (V13);
\draw (V23) -- (V12);
\draw (V23) -- (V13);
\draw (V24) -- (V13);
\draw (V23) -- (V14);
\draw (V24) -- (V14);
\draw (V25) -- (V14);
\draw (V24) -- (V15);
\draw (V25) -- (V15);

\node[draw=none, fill=none] at (10.9,3) {$\ldots$};
\node[draw=none, fill=none] at (10.9,1) {$\ldots$};

\node[draw=none, fill=none] at (1,3) {$\ldots$};
\node[draw=none, fill=none] at (1,1) {$\ldots$};

\node[draw=none, fill=none] at (2,0.5) {$\vdots$};
\node[draw=none, fill=none] at (4,0.5) {$\vdots$};
\node[draw=none, fill=none] at (6,0.5) {$\vdots$};
\node[draw=none, fill=none] at (8,0.5) {$\vdots$};
\node[draw=none, fill=none] at (10,0.5) {$\vdots$};

\end{tikzpicture}
\caption{$cl(\Omega) = cl(Y) = X_B$.}\label{Fig:ERSECSnotHS}
\end{figure}
\end{example}

The following statement gives a sufficient condition for a generalized Bratteli diagram $B$ to have $X_B = \Omega$, i.e., all infinite paths are unbounded in the sense discussed above.

\begin{prop} Let $B = B(F_n)$ be a two-sided infinite generalized Bratteli diagram (identify $V_n = \Z$ for $n\in \N_0$). Assume that there exist infinitely many levels $\{n_k\}$ and an increasing sequence of natural numbers $(M_{n_k})$ such that for every $e \in E_{n_k}$  
$$
|r(e) - s(e)| \geq M_{n_k}.
$$
    Then $X_B = \Omega$.
\end{prop}

\begin{proof} Assume by contradiction $X_B \neq \Omega$, and $Y$ is non-empty, where $Y \subset X_B$ is as in \Cref{def_Om_Y}. Consider $\ov y = (y_i) \in Y$, then there exists $C > 0$ such that for all $n \in \mathbb{N}_0$, we have $| s(y_n) |\leq C$. There exists $k$ large enough such that $M_{n_k} \geq 2C + 1$. Hence, at level $n_k$ we have $|r(y_{n_k})| = |s(y_{n_k+1})| > C$, and we get a contradiction.
\end{proof}

\begin{remark}
 Denote by $Erg_1(X_B)$ the set of all ergodic probability tail invariant measures on $X_B$.
It follows from \Cref{lem s4 Y and Omega} and the $\mathcal{R}$-invariance of the sets $Y$ and $\Om$ that if $\mu \in Erg_1(X_B)$ then either $\mu(\Omega) = 0$ or $\mu(Y) = 0$. Thus, all ergodic probability tail invariant measures are divided into two sets: 
$$
Erg(Y) = \{\mu \in Erg_1(X_B) \; | \; \mu(Y) = 1\}
$$ 
and 
$$
Erg (\Om) = \{\mu \in Erg_1(X_B) \; | \; \mu(\Omega) = 1\}.
$$

As an example, we observe that finite extensions of measures from stationary standard Bratteli subdiagrams belong to the class $Erg(Y)$. 
\end{remark}

\subsection{Path space of subdiagrams}
Suppose $B$ is a generalized Bratteli diagram and $\ol B$ is a vertex subdiagram of $B$ defined by the sequence $\ol W= (W_n)$ of proper subsets $W_n$ (see \Cref{def: vertex sbdgr}). In what follows, we describe the path space of $\ol B$.

For every path $\ol x = (x_n)_{n \in \N_0} \in X_B$, we associate the sequence of vertices $(v_n)_{n \in \N_0}$ where $v_n = s(x_n)$ for $n \in \N_0$. This defines a map
$$
\tau : X_B \to \prod_{n \in \N_0} V_n. 
$$
If $c= [v_o, \ldots ,v_n]$ is a cylinder set in $\prod_{n \in \N_0} V_n$,
then the set $\tau^{-1}(c)$ consists of all the paths that go through 
$v_0, \ldots, v_n$. Hence $\tau^{-1}(c)$ is a finite union of cylinder sets in $X_B$. Note that $\tau$ is not surjective. Recall that for $n \in \N_0$, we set $W'_n = V_n \setminus W_n$, (see \Cref{def: vertex sbdgr}).

\begin{prop} \label{prop path space sbdgrm}
The map $\tau$ is continuous and, in the above settings,
$$
X_{\ol B} = X_B \setminus  \bigcup_{n \geq 0} \ 
\bigcup_{w_i \in W_i, v_n \in W'_n } \tau^{-1}( [w_0, \ldots, w_{n-1}, v_n]) 
$$
\end{prop}

The proof of \Cref{prop path space sbdgrm} is straightforward.

\begin{remark}
 From this statement, it follows that $X_{\ol B}$ is closed. Moreover, if $w \in W_n$, $w' \in W_{n+1}$ and there exists a path $\ol y = (y_n)_{n \in \N_0} \in X_{\ol B}$ such that $y_n \in E(w, w')$, then, by \Cref{prop path space sbdgrm}, for every $e \in E(w, w')$ there exists a path $\ol y' = (y'_n)_{n \in \N_0} \in X_{\ol B}$ such that $e = y'_n$. 
\end{remark}

\begin{remark} We can similarly consider an edge subdiagram $\ol B$ of $B$ (see \Cref{Def_edge_Sub_dig}). Let $\ol G = (G_n)$ be a sequence of subsets $G_n$ of the sets of
edges $E_n$ that corresponds to the subdiagram $\ol B$. Observe that we can identify the path space $X_B$ with a subset 
of $\prod_n E_n$ because every infinite path $\ol x = (x_n)$
is a sequence of edges $x_n \in E_n$. Hence, if $\ol y = (y_n)$ 
is not in 
$X_{\ol B}$, then there exists a cylinder set $[e_0, \ldots, e_{k-1}, f_k]$ where $e_i \in G_i$, but $f_k$ is not in $G_k$ for some $k$. 
Thus, we have  
$$
X_{\ol B} = X_B \setminus  \bigcup_{k \geq 0} \ 
\bigcup_{e_i \in G_i, f_k \notin G_k} [e_0, \ldots, e_{k-1}, f_k]). 
$$
\end{remark}

We now consider the converse question. Let $B$ be a given generalized Bratteli diagram and $P$ a subset of $X_B$. We ask the following question: Under what conditions does $P \subset X_B$ correspond to the path space $X_{\ol B}$ of a subdiagram $\ol B$ of $B$? In other words, does there exist a vertex (or edge) subdiagram $\ol B$ of $B$ such that $P = X_{\ol B}$? As mentioned in \Cref{Sec: Basics}, we are interested in the Bratteli diagrams whose path space $X_B$ has no isolated points and the tail equivalence relation $\mathcal R$ is an aperiodic CBER (countable Borel equivalence relation).

Suppose that $B = (V, E)$ is a generalized Bratteli diagram with path space $X_B$. We say that a perfect subset $P$ of $X_B$ is \textit{full} if it satisfies the following conditions:

\begin{enumerate}[label=(\roman*)]
\item $(P \times P)|_{\ol{\mathcal R}}$ 
is an aperiodic CBER where $\ol{\mathcal R}$ is the restriction of $\mathcal R$ to $P$.
\item If $C$ is a cylinder set in $X_B$ and $C \cap P \neq \emptyset$, then $C \cap P$ is perfect.
\item If $\ol y = (y_n)_{n \in \N_0} \in P$ and $v_n = s(y_n)$, then for any $k \in \N_0,m \in \N$ and any finite path $\ol f = (f_k, \ldots, f_{k+m}) \in E(v_k, v_{k+m+1})$, there exists $\ol z = (z_n)_{n \in \N_0} \in P$ such that $f_i = z_i,\ i = k, \ldots k+m$. 
\end{enumerate}

Recall that in \Cref{rem on edge diagram} we considered the case where a subset $Y$ of the path space $X_B$ is obtained by forbidding cylinder sets. We will denote this by $Y \notin \mc F_B $. We explained there that such subsets cannot be path spaces of edge subdiagrams. 

\begin{thm}\label{thm s8 path space}
Let $B=B(V, E)$ be a generalized Bratteli diagram with the path space $X_B$. Suppose that a closed perfect subset $P$ of $X_B$ is full, that is, $P$ satisfies conditions  $(i), (ii)$ and $(iii)$ above. Then there exists a vertex subdiagram $\ol B$ of $B$ such that $X_{\ol B} = P$.

Let $P$ be a perfect set which is not in $\mc F_B$. Suppose that $P$ satisfies $(i)$ and $(ii)$ of the definition of fullness above, and instead of property $(iii)$, it satisfies the property

 $(iii')$ for every $n\in \N_0$ and every $w \in V_n$, there exists an edge $e \in s^{-1}(w)$ such that $[e] \cap P$ is a nonempty perfect set. 

Then there exists an edge subdiagram $\ol B$ of $B$ such that $X_{\ol B} =P$.
\end{thm}

\begin{proof}
Every infinite path $\ol y = (y_n)_{n \in \N_0}\in P$ determines the sequence of vertices
$(v_n)_{n \in \N_0}$ where $v_n = s(y_n)$, $n \in \N_0$. We construct by induction, a sequence $(W_n)_{n \in \N_0}$ where for each $n \in \N_0$, $W_n$ is a subsets of $V_n$. Here for $n \in \N_0$, $V_n$ denotes the vertex set in $B$ at level $n$.  

Let $W_0 = \{w \in V_0 : [w] \cap P \neq \emptyset\}$, where $[w]$ denotes the subset of $X_B$ that contains all infinite paths that begin at the vertex $w \in V_0$. Having $W_0$ defined, we take $w_0 \in W_0$ and consider the set
of vertices $v\in r(s^{-1}(w_0)) \subset V_1$ and take $W_1(w_0) = \{w_1 \in r(s^{-1}(w_0)) \ :\ [w_0, w_1] \cap Y \neq \emptyset\}$. Here, $[w_0, w_1]$ is the set of all paths that pass through the vertices $ w_0$ and $ w_1$. Then we set 
$$
W_1 = \bigcup_{w_0 \in W_0}W_1(w_0).
$$
Note that this construction can be repeated for cylinder sets of any finite length. In such a way, we define inductively the sequence $(W_n)_{n \in \N_0}$ where for each $n \in \N_0$, $W_n$ is a subset of $V_n$. 

We now set $\ol B= B(\ol W)$ to be the vertex subdiagram of $B$ supported by the sets $(W_n)_{n \in \N_0}$. 
Set $ G_n \subset E_n$, where for every $n\in \N_0$, edge $e \in E_n$ lies in $G_n$ if and only if $s(e) \in W_{n}$ and $r(e) \in W_{n+1}$.

We claim that $P = X_{\ol B}$. By construction, 
$P \subset X_{\ol B}$. On the other hand, condition $(iii)$ implies that
if two vertices, $w \in W_n$ and  $v \in W_{n+m}$, are connected by a finite path in $E(w, v)$, then they are connected in $\ol B$. Moreover, every finite path $f \in E(w, v)$ appears as a finite segment in some infinite path from $P$. This shows that $X_{\ol B} \subset P$ for the vertex subdiagram $\ol B$. 

For the second statement, we use a similar approach to construct an 
edge subdiagram $\ol B$. We need to find the subset of edges $G_n$ of $E_n$ that will form an edge subdiagram. Take $w_0 \in V_0$ and choose the set $G(w_0)$ of all edges $e$ from $s^{-1}(w_0)$ satisfying condition $(iii')$. Set $G_0 = \bigcup_{w_0 \in V_0} G(w_0)$. We note that every vertex $w_1 \in V_1$ is the range of some edge from
$G_0$. Indeed, condition $(iii')$ says that there exists a path  $ \ol y = (y_n) \in [f] \cap P$ where $s(f) = v_1$. But then we see that $y_0 \in G_0$ by the definition of this set and $r(y_0) = v_1$. 
This process can be applied to every level $V_n$ inductively, and this will give us the sequence $\ol G = ( G_n)$ of subsets of edges in the diagram $B$. 
Hence, we constructed an edge subdiagram $\ol B$. It remains to observe that $P = X_{\ol B}$ from the construction of $\ol B$. 

\end{proof}

\subsection{$0-1$ procedure}
In this subsection, we discuss a procedure that allows one to pass from an arbitrary (standard or generalized) Bratteli diagram to a Bratteli diagram with $0-1$ incidence matrices. In particular, this procedure allows the reduction of the study of edge subdiagrams to the study of vertex subdiagrams. However, it may increase the amount of computations, for instance, needed to compute the measure extension. The discussed procedure is well known in the case of standard Bratteli diagrams, see e.g. \cite[Section 5]{Fisher2009}. The $0-1$ procedure can also be obtained using the procedure of microscoping described in \cite[Section 3]{GiordanoPutnamSkau1995}.

The $0-1$ procedure can be described as follows. We start with a Bratteli diagram (standard or generalized) $B = (V, E)$ and construct a new diagram $\tl B  = (\tl V, \tl E)$ such that its incidence matrices are $0-1$ matrices. Let the set of vertices $\tl V_n$ of level $n \in \N_0$, for the new diagram $\tl B$ be the set of edges $E_n$ for the diagram $B$, i.e., there is a bijection $f \colon  E\rightarrow \tl V$ such that $f(E_n) = \tl V_n$ for all $n \in \N_0$. Two vertices $\tl v \in \tl V_{n+1}$ and $\tl w \in \tl V_n$ are connected by an edge in $\tl B$ if and only if $s(f^{-1}(\tl v)) = r(f^{-1}(\tl w))$. This means that the edges $f^{-1}(\tl v)$ and $f^{-1}(\tl w)$ from $B$, corresponding to $\tl v$ and $\tl w$, are concatenated in $B$. This construction induces a homeomorphism $\psi \colon \ X_B\rightarrow X_{\tl B}$ between the path spaces of $B$ and $\tl B$.
The cylinder sets in $X_B$ of length  $n \in \N$ are mapped to cylinder sets of length 
$n - 1$ in $X_{\tl B}$.
Remark that $\psi$ preserves the tail equivalence relations: $(\psi\times \psi) (\mc R_B) = \mc R_{\tl B}$. This implies that the measure $\mu$ is $\mathcal{R}_B$-invariant if and only if the measure $\nu = \mu \circ \psi^{-1}$ is $\mc R_{\tl B}$-invariant, and every $\mc R_{\tl B}$-invariant measure $\nu$ can be represented as $\mu \circ \psi^{-1}$ for some  $\mathcal{R}_B$-invariant measure $\mu$ (see also \cite[Lemma 5.2]{Fisher2009} for a detailed discussion).

The following proposition easily follows from the definition of the $0-1$ procedure. Here the notation $f_{v,w}^{(n)}$ corresponds to the$(v,w)$-entry of the incidence matrix $F_n$ for the diagram $B = B(F_n)$. 
\begin{prop}\label{prop:0-1procedure}
    Let $B = (V,E) = B(F_n)$ be a standard Bratteli diagram and $\tl{B} = (\tl V, \tl E) = \tl B(\tl F_n)$ be a diagram obtained from $B$ using the $0-1$ procedure. Then for $n \in \N_0$,
    $$
    |\tl V_n| = \sum_{v\in V_{n+1}}\sum_{w \in V_n} f_{v,w}^{(n)}
    $$
    and
    $$
    |\tl E_n| = \sum_{w \in V_{n}}\sum_{u \in V_{n+1}}\sum_{v \in V_{n+2}}f^{(n+1)}_{v,u}f^{(n)}_{u,w}.
    $$ 
    In other words for $n \in \N_0$, $|\tl V_n|$ is the number of all edges in $B$ on level $n$ and $\tl E_n$ is the number of all finite paths of length $2$ between the vertices of $|V_n|$ and $|V_{n+2}|$.
\end{prop}

\begin{example}\label{Ex:0-1ProcOdometer}
    Let $B$ be an odometer with $F_n = (a_n)$ for all $n \in \N_0$, where $(a_n)_{n \in \N_0}$ is a sequence of positive integers. Then for the image $\tl B$ of $B$ under the $0-1$ procedure, we have $|\tl V_n| = a_n$ for all $n \in \N_0$, and $\tl F_n$ is an $a_{n+1} \times a_n$ matrix such that all its entries are equal to $1$. 
    We can reverse the $0-1$ procedure: starting with the diagram $\tl B$, we obtain that its preimage $B$ under the $0-1$ procedure should have $a_n$ edges on level $n$. The fact, that all vertices of $\tl V_n$ and $\tl V_{n+1}$ are pairwise connected means that for all $e \in E_n$ and $f \in E_{n+1}$, we have $r(e) = s(f)$, hence the diagram $B$ is an odometer (up to the level $V_0$). The inverse of the $0-1$ procedure cannot retrieve the information on the number of vertices on the level $V_0$. 
\end{example}

\begin{example}
    Let $B = B(F)$ be a stationary standard
Bratteli diagram with the incidence matrix
$$
    F_n = \begin{pmatrix}
        2 & 0 \\
        1 & 3
    \end{pmatrix}
$$
for all $n \in \N_0$. 
Then for all $n \in \N_0$, we have $\tl V_n = 6$, $\tl E_n = 18$ and 
$$
\tl F_n = \begin{pmatrix}
    1 & 1 & 0 & 0 & 0 & 0\\
    1 & 1 & 0 & 0 & 0 & 0\\
    1 & 1 & 0 & 0 & 0 & 0\\
    0 & 0 & 1 & 1 & 1 & 1\\
    0 & 0 & 1 & 1 & 1 & 1\\
    0 & 0 & 1 & 1 & 1 & 1\\
\end{pmatrix}.
$$
\end{example}

Note that the image $\tl B = (\tl V, \tl E)$ of $B = (V, E)$, under the $0-1$ procedure, is defined up to isomorphism. Indeed, we can enumerate the edges of $E_n$ differently and obtain a different enumeration of the set $\tl V_n$. We define a \textit{canonical image} of $B$ under the $0-1$ procedure by enumerating $E_n, n \in N_0$ as follows. Suppose that for all $n \in \N_0$, we identify $V_n$ with $\N_0$. If $B$ is an unordered diagram, then we begin with the vertex $0 \in V_{n+1}$ and enumerate the edges of $r^{-1}(0)$ from left to right by the numbers $0, \ldots, |r^{-1}(0) - 1|$. Then we consider the vertex $1 \in V_{n+1}$ and enumerate all edges of $r^{-1}(1)$ from left to right starting with the number $|r^{-1}(0)|$, and so on. In general, for any $l \geq 0$, the edges of $r^{-1}(l)$ are enumerated from left to right by the numbers $\sum_{i = 0}^{l-1}|r^{-1}(i)|, \ldots, \sum_{i = 0}^{l}|r^{-1}(i)| - 1$. If for $n \in \N_0$ the set $V_{n}$ is identified with $ \Z$, then we enumerate the edges of $r^{-1}(-l)$ from left to right by the numbers $- \sum_{i=1}^{l}|r^{-1}(i)|, \ldots, -\sum_{i=1}^{l-1}|r^{-1}(i)| - 1$ for all $l > 0$. If $B$ is ordered, we use the given order on each set $r^{-1}(l)$ rather than the left-to-right order. Recall the equal row sums (ERS ($r_n$)) and equal column sums (ECS($c_n$)) properties for a diagram where $(r_n)_{n \in \N_0}$ and $(c_n)_{n \in \N_0}$ are sequence of positive integers (see \Cref{Sec: Basics}). 

\begin{prop}\label{prop:0-1properties} Let $B = B(F_n)$ be a (standard or generalized) Bratteli diagram and $\tl B = \tl B(\tl F_n)$ be its canonical image under the $0-1$ procedure. Let $(r_n)_{n \in \N_0}$ and $(c_n)_{n \in \N_0}$ be two sequence of positive integers. Then
    \begin{enumerate}[label=\upshape(\roman*), leftmargin=*, widest=iii]

        \item\label{item:0-1ers} If $B$ has the $ERS(r_n)$ (respectively $ECS(c_n)$) property, then $\tl B$ has the $ERS(r_n)$ (respectively $ECS(c_{n+1})$) property.

        \item\label{item:0-1stat} If $B$ is stationary then $\tl B$ is stationary.
        
        \item\label{item:0-1horstat} Assume that $B \in ERS(r_n)$ is a horizontally stationary generalized Bratteli diagram. Then the diagram $\tl B$ is not horizontally stationary. 
        

        
    \end{enumerate}
\end{prop}

\begin{proof}
    \ref{item:0-1ers} Assume that $B$ have $ERS(r_n)$ property. Then for every $n \in \N_0$ and every $w \in V_{n+1}$ we have $|r^{-1}(w)| = r_n$. Hence, every edge $e \in E_{n+1}$ with $s(e) = w$ will have exactly $r_n$ edges that can be concatenated with it from the level above. Hence, for $\tl B$ we have $|{\tl r}^{-1}(\tl v)| = r_n$ for all $\tl w \in \tl V_n$. Similarly, for all $n \in \N_0$, the number of edges that can be concatenated with the edges of $E_n$ from a level below is $c_{n+1}$. Hence, $\tl B$ has $ECS(c_{n+1})$ property for $n \in \N_0$.

    \ref{item:0-1stat} The statement is obvious since we pick the same way of enumerating $E_n$ on all levels.

    \ref{item:0-1horstat} Without loss of generality, we can assume that there are infinitely many levels $n \in \N_0$ such that $r_n \geq 2$, and in particular, $r_0 \geq 2$. Otherwise, the diagram would be degenerated with a countable path space. Then, there are edges $e_0^{(0)}, e_1^{(0)} \in E_0$ from the set of edges of the diagram $B$ (the lower index corresponds to the enumeration defined for the canonical image under the $0-1$ procedure), both end at the vertex $v_0^{(1)}$ on level $V_1$. The edges $e_0^{(0)}$, $e_1^{(0)}$ correspond to the vertices $\tl v_0^{(0)}$, $\tl v_1^{(0)}$ on level $0$ for $\tl B$. Since the ranges of $e_0^{(0)}$, $e_1^{(0)}$ in $B$ are the same, these edges are connected to the same set of edges on the level $E_1$ below. Thus, the vertices $\tl v_0^{(0)}$, $\tl v_1^{(0)}$ are connected to the same set of vertices in $\tl V_1$ of $\tl B$. In other words, the incidence matrix $\tl F_0$ of $\tl B$ has the property $\tl f^{(0)}_{i,0} = \tl f^{(0)}_{i,1}$ for all $i \in \Z$. Assume that $\tl B$ is horizontally stationary. Then $\tl f^{(0)}_{i,j}$ is a function of $i - j$ (see \cite[Proposition 3.2]{BezuglyiJorgensenKarpelKwiatkowski2025}). Denote this function by $\tl f^{(0)}$. Then $\tl f^{(0)}_{i,0} = \tl f^{(0)}(i - 0) = \tl f^{(0)}_{i,1} = \tl f^{(0)}(i-1)$ for all $i \in \Z$. Thus, the function $\tl f^{(0)}$ is constant on $\Z$ and $\tl f^{(0)}_{i,0} = \tl f^{(0)}_{j,0}$ for all $i, j$. We get a contradiction, since $F_0$ is a non-zero matrix with the finite row and column sums.
\end{proof}

\begin{remark}
    The part \ref{item:0-1ers} of \Cref{prop:0-1properties} holds for an arbitrary image $\tl B$ of $B$, not necessarily a canonical one. Part \ref{item:0-1stat} is true if we pick the same order on $E_n$ for all $n \in \N_0$.
\end{remark}

The following example illustrates \Cref{prop:0-1properties}.

\begin{example}\label{ex:0-1proc}
Let $B = B(F)$ be a two-sided (vertically) stationary and horizontally stationary generalized Bratteli diagram defined by the incidence matrix $F = (f_{i,j})$:
$$
f_{i,j} = 
\left\{
\begin{aligned}
& 1, \mbox{ for } |i - j| = 1,\\
& 0, \mbox{ otherwise. }
\end{aligned}
\right.
$$
Then $B$ has both $ERS(r_n)$ and $ ECS(c_n)$ properties, where $r_n = c_n =2$ for all $n \in \N_0$. Let $\tl B = \tl B(\tl F)$ be the canonical image of $B$ under the $0-1$ procedure. Then $\tl B$ also have $ ERS(r_n)$ and $ ECS(r_n)$ property, where $r_n = 2$ for all $n \in \N_0$ and $\tl B$ is (vertically) stationary. Though $\tl B$ is not horizontally stationary (see \Cref{fig:0-1procex}, we draw only the edges that connect the vertices shown on the picture), we have $\tl f_{i,j}^{(n)} = \tl f_{i+2,j+2}^{(n)}$ for every $n \in \N_0$, $i \in \tl V_{n+1}$ and $j \in \tl V_n$.

\begin{figure}
\centering
\begin{minipage}[c]{0.45\textwidth}
\centering
\begin{tikzpicture}[scale=0.8,
    vertex/.style={circle, draw, fill=black, minimum size=0.2cm, inner sep=0pt},
    every edge/.style={thin, black}]

\node[vertex] (V11) at (1,5) {};
\node[vertex] (V12) at (3,5) {};
\node[vertex] (V13) at (5,5) {};

\node[vertex] (V21) at (1,3) {};
\node[vertex] (V22) at (3,3) {};
\node[vertex] (V23) at (5,3) {};

\node[vertex] (V31) at (1,1) {};
\node[vertex] (V32) at (3,1) {};
\node[vertex] (V33) at (5,1) {};

\draw (V21) -- (V12);
\draw (V22) -- (V11);
\draw (V22) -- (V13);
\draw (V23) -- (V12);

\draw (V31) -- (V22);
\draw (V32) -- (V21);
\draw (V32) -- (V23);
\draw (V33) -- (V22);

\node[draw=none, fill=none] at (1.3,3.4) {$-1$};
\node[draw=none, fill=none] at (2.7,3.4) {$0$};
\node[draw=none, fill=none] at (3.3,3.4) {$1$};
\node[draw=none, fill=none] at (4.7,3.4) {$2$};

\node[draw=none, fill=none] at (1.3,1.4) {$-1$};
\node[draw=none, fill=none] at (2.7,1.4) {$0$};
\node[draw=none, fill=none] at (3.3,1.4) {$1$};
\node[draw=none, fill=none] at (4.7,1.4) {$2$};

\node[draw=none, fill=none] at (0.1,5) {$\ldots$};
\node[draw=none, fill=none] at (0.1,3) {$\ldots$};
\node[draw=none, fill=none] at (0.1,1) {$\ldots$};

\node[draw=none, fill=none] at (5.9,5) {$\ldots$};
\node[draw=none, fill=none] at (5.9,3) {$\ldots$};
\node[draw=none, fill=none] at (5.9,1) {$\ldots$};

\node[draw=none, fill=none] at (1,0.1) {$\vdots$};
\node[draw=none, fill=none] at (3,0.1) {$\vdots$};
\node[draw=none, fill=none] at (5,0.1) {$\vdots$};

\end{tikzpicture}
\end{minipage}
\quad
\begin{minipage}[c]{0.45\textwidth}
\centering
\begin{tikzpicture}[scale=0.8,
    vertex/.style={circle, draw, fill=black, minimum size=0.2cm, inner sep=0pt},
    every edge/.style={thin, black}]

\node[vertex] (V11) at (1,5) {};
\node[vertex] (V12) at (3,5) {};
\node[vertex] (V13) at (5,5) {};
\node[vertex] (V14) at (7,5) {};

\node[vertex] (V21) at (1,3) {};
\node[vertex] (V22) at (3,3) {};
\node[vertex] (V23) at (5,3) {};
\node[vertex] (V24) at (7,3) {};

\node[vertex] (V31) at (1,1) {};
\node[vertex] (V32) at (3,1) {};
\node[vertex] (V33) at (5,1) {};
\node[vertex] (V34) at (7,1) {};

\node[draw=none, fill=none] at (1,5.4) {$-1$};
\node[draw=none, fill=none] at (3,5.4) {$0$};
\node[draw=none, fill=none] at (5,5.4) {$1$};
\node[draw=none, fill=none] at (7,5.4) {$2$};

\node[draw=none, fill=none] at (1,3.4) {$-1$};
\node[draw=none, fill=none] at (3,3.4) {$0$};
\node[draw=none, fill=none] at (5,3.4) {$1$};
\node[draw=none, fill=none] at (7,3.4) {$2$};

\node[draw=none, fill=none] at (1,1.4) {$-1$};
\node[draw=none, fill=none] at (3,1.4) {$0$};
\node[draw=none, fill=none] at (5,1.4) {$1$};
\node[draw=none, fill=none] at (7,1.4) {$2$};

\draw (V21) -- (V12);
\draw (V22) -- (V11);
\draw (V21) -- (V13);
\draw (V24) -- (V12);
\draw (V24) -- (V13);
\draw (V23) -- (V14);

\draw (V31) -- (V22);
\draw (V32) -- (V21);
\draw (V31) -- (V23);
\draw (V34) -- (V22);
\draw (V34) -- (V23);
\draw (V33) -- (V24);

\node[draw=none, fill=none] at (0.1,5) {$\ldots$};
\node[draw=none, fill=none] at (0.1,3) {$\ldots$};
\node[draw=none, fill=none] at (0.1,1) {$\ldots$};

\node[draw=none, fill=none] at (7.9,5) {$\ldots$};
\node[draw=none, fill=none] at (7.9,3) {$\ldots$};
\node[draw=none, fill=none] at (7.9,1) {$\ldots$};

\node[draw=none, fill=none] at (1,0.1) {$\vdots$};
\node[draw=none, fill=none] at (3,0.1) {$\vdots$};
\node[draw=none, fill=none] at (5,0.1) {$\vdots$};
\node[draw=none, fill=none] at (7,0.1) {$\vdots$};

\end{tikzpicture}
\end{minipage}
\caption{Diagrams $B$ (on the left) and $\tilde B$ (on the right) from the \Cref{ex:0-1proc}}
\label{fig:0-1procex}
\end{figure}
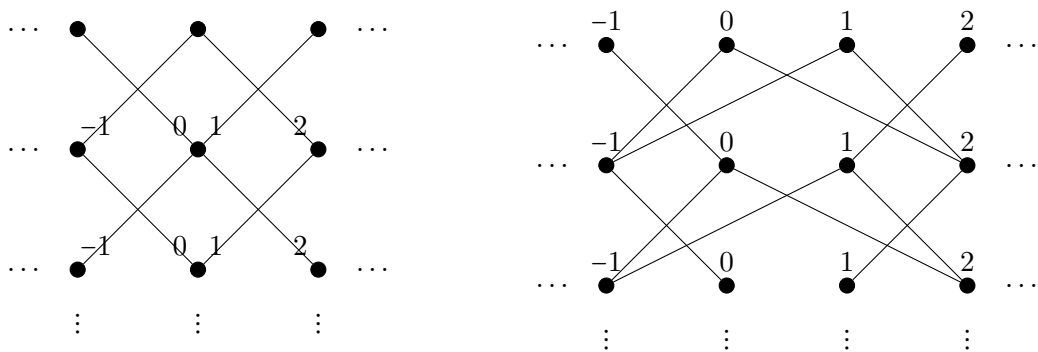

\end{example}

\begin{example}
In this example, we show how the $0-1$ procedure allows us to pass from edge to vertex subdiagrams. Let $(b_n)_{n \in \N_0}$ be a sequence of positive integers and $B = B(F_n)$ be an odometer with $F_n = (b_n)$ for $n \in \N_0$. Let $(a_n)_{n \in \N_0}$ be a sequence of positive integers such that $a_n \leq b_n$ for $n \in \N_0$. Thus $\ov B = \ov B(\ov F_n)$ is an edge subdiagram of $B$ where $\ov F_n = (a_n)$ for $n \in \N_0$. After applying $0-1$ procedure to $B$, we obtain a diagram $\tl B$ with $b_{n+1} \times b_n$ incidence matrices with all entries equal to $1$ and a vertex subdiagram $\widetilde{\ov B}$ with $a_{n+1} \times a_n$ incidence matrices whose all entries are equal to $1$ (see \Cref{Ex:0-1ProcOdometer}). Let $\ov \mu$ and $\tl {\ov \mu}$ be the unique probability invariant measures defined on $\ov B$ and $\tl {\ov B}$, respectfully. Then 
$(a_0 \cdots a_n)^{-1}$ is a measure of a cylinder set of length $n+1$ for $\ov \mu$ and of length $n$ for $\tl{\ov \mu}$. The finiteness of the measure extension from the edge subdiagram $\ov B$ is determined by the convergence of the series \eqref{eq:edge_sbd_mu_ext}:
$$
\ba 
\sum_{n=0}^{\infty} \sum_{v \in V_{n+1}} \sum_{w \in V_{n}} f_{v,w}^{'(n)} H_w^{(n)}\overline{p}_v^{(n+1)} = &\ \sum_{n=0}^{\infty} \frac{(b_n - a_n)b_0 \ \cdots\ b_{n-1}}{a_0 \cdots a_n}\\
=&\  \sum_{n=0}^{\infty} \left(\prod_{i = 0}^n\frac{b_i}{a_i} - \prod_{i = 0}^{n-1}\frac{b_i}{a_i}\right).
\ea
$$
Similarly, finiteness of the measure extension from the vertex subdiagram $\tl{\ov B}$ is determined by convergence of the series \eqref{eq3:three criteria}:
$$
\ba 
\sum_{n=0}^\infty \sum_{v \in W_{n+1}} 
\sum_{w \in W'_n} f^{(n)}_{v,w} H_w^{(n)} \ol p_v^{(n+1)} = &\ \sum_{n=0}^\infty a_{n+1}(b_{n} - a_{n})\frac{b_0\ \cdots\ b_{n-1}}{a_0 \ \cdots \ a_{n+1}}\\
= & \sum_{n=0}^{\infty} \left(\prod_{i = 0}^n\frac{b_i}{a_i} - \prod_{i = 0}^{n-1}\frac{b_i}{a_i}\right).
\ea
$$
Naturally, we obtain the same condition for the finiteness of measure extension. Both series converge if and only if the product $\prod_{i =0}^\infty \frac{b_i}{a_i}$ converges.
\end{example}

\section{Convergence of measures}\label{Sec: conv of measures}

In this section, we consider the convergence of (finite and infinite) Borel measures defined on the path space of a (standard or generalized) Bratteli diagram $B$.

Let $\mu$ be a Borel measure defined on the path space $X_B$ of a (standard or generalized) Bratteli diagram $B$. 
We say that a measure $\mu$ on $X$ \textit{takes finite values on cylinder sets} if for every cylinder set $C$ we have $\mu(C) < \infty$. The following definition of convergence on cylinder sets appears in \cite{Iommi_Veloso_2021} in the context of Markov shifts. Here we apply it for the measures defined on a path space of a (standard or generalized) Bratteli diagram:

\begin{definition} \cite[Definition 3.14]{Iommi_Veloso_2021} \label{def_conv_cyl_set} Let $B$ be a (standard or generalized) Bratteli diagram and $X_B$ be its path space. 
Let $(\mu_n)_{n \in \N}$ and $\mu$ be measures on $X_B$ that take finite values on cylinder sets. Then the sequence $(\mu_n)_{n \in \N}$ \textit{converges on cylinders} to the measure $\mu$ if $$
\underset{n\rightarrow \infty}{\textrm{lim}} \mu_n (C) = \mu (C)
$$ for every cylinder set $C \subset X_B$. We will denote convergence on cylinders by $\mu_n \xrightarrow{[C]} \mu$. \end{definition}

\begin{remark}
In \Cref{sect Approximation}, we introduced the notion of weak convergence of measures on cylinder sets (see \Cref{Def: weak_conv}). The difference between that notion and \Cref{def_conv_cyl_set} is the following. In \Cref{def_conv_cyl_set}, the measures $(\mu_n)_{n \in \N}$ have full support $X$, whereas in \Cref{Def: weak_conv} the measures $(\nu_i)_{i \in \N}$ are supported by an increasing sequence of Borel subsets $(\Omega_i)_{i \in \N}$. 
\end{remark}

\begin{remark}
In \cite[Lemma 3.17]{Iommi_Veloso_2021}, it was proved that 
a sequence of Borel probability measures defined on a phase space of a countable Markov shift converges in weak* topology if and only if it converges on cylinder sets. The same proof works for Borel probability measures defined on a path space of a generalized Bratteli diagram. 
\end{remark}

\begin{prop}\label{prop_cyl_sets_conv_subprob}
Let $X_B$ be a path space of a (standard or generalized) Bratteli diagram.
Let $(\mu_n)_{n \in \N}$ be Borel subprobability measures on $X$ such that 
\[
\lim_{n \rightarrow \infty}\mu_n(C) = \mu(C)
\]
for every cylinder set $C$, where $\mu$ is a probability measure on $X$. Then 
\[
\lim_{n \rightarrow \infty}\mu_n(X) = \mu(X).
\]
\end{prop}

\begin{proof}
First, we note that for every finite union of cylinder sets $C_i$
$$
D = \bigcup_{i \in I} C_i, \quad |I| < \infty,
$$ 
we have
\[
\lim_{n \rightarrow \infty}\mu_n(D) = \mu(D).
\]
Now suppose that $\mu_n(X)$ does not converge to $\mu(X) = 1$. 
This means that there exists a subsequence $(n_k)_{k \in \N}$ of positive integers and $c < 1$ such that 
$$
\mu_{n_k}(X) < c
$$
for all $(n_k)_{k \in \N}$. Hence, there exists $\alpha \leq c$ and a subsequence $(n_i)_{i \in \N}$ of $(n_k)_{k \in \N}$ such that
\[
\lim_{n \rightarrow \infty}\mu_{n_i}(X) = \alpha.
\] {Without loss of generality, we can assume that 
\[
\lim_{n \rightarrow \infty}\mu_{n}(X) = c < 1.
\]
and $\mu_n(X) \leq c$ for all $n \in \N$.}

Consider a finite union of cylinder sets $D \subset X$
such that $\mu(D) > c_1 > c$ for some $c_1 < 1$. Then we have 
$$
\mu_n(D) \rightarrow \mu(D)
$$
and
$$
\lim_{n \rightarrow \infty} \mu_n(D) > c_1 > c.
$$
On the other hand, 
$$
\mu_n(D) \leq \mu_n(X) < c.
$$
Thus, we obtain a contradiction. Hence 
$$
\mu_n(X) \rightarrow \mu(X) = 1.
$$
\end{proof}

\begin{remark} 
Note that the proof above holds for any finite measure $\mu$ on $X$ and a sequence of measures $\mu_n$ which converge to $\mu$ on cylinder sets with $\mu_n(X) < \mu(X)$. 
\end{remark}

\Cref{Prop_general_conv_meas_fin} provides a sequence of non-subprobability finite measures $(\mu_n)_{n \in \N}$ for which the statement of \Cref{prop_cyl_sets_conv_subprob} still holds.

\begin{example}\label{Prop_general_conv_meas_fin}
Let $X_B$ be a path space of a (standard or generalized) Bratteli diagram and $\mu$ and $\nu$ be Borel probability measures on $X_B$. For $n \in \N$, we define a sequence of non-subprobability finite measures $\mu_n$ as
$$
\mu_n = \mu + \frac{1}{n} \nu.
$$ 
Then $\mu_n \xrightarrow{[C]} \mu$ and 
$$
\lim_{n \rightarrow \infty} \mu_n(X) = \mu(X) = 1.
$$
\end{example}
\Cref{Prop_general_conv_meas_inf} shows that in general \Cref{prop_cyl_sets_conv_subprob} is not true if the measures $(\mu_n)_{n \in \N}$ are not subprobability measures. In particular, it shows that infinite $\sigma$-finite measures can converge to a probability measure on cylinder sets. 
\begin{example}\label{Prop_general_conv_meas_inf}
Let $B$ be a generalized Bratteli diagram and $X_B$ be its path space.
Let $\mu$ be a probability measure on $X_B$.  
Let $\nu$ be an infinite $\sigma$-finite measure which takes finite values on cylinder sets. Note that such measures can exists only on the path spaces of generalized Bratteli diagrams, not of standard ones. Examples of such tail invariant measures can be found for instance in \cite[Proposition 7.10]{BezuglyiJorgensenKarpelSanadhya2025} and \cite[Theorem 5.5]{BezuglyiKarpelKwiatkowski2024}. For $n \in \N$, we define a sequence of $\sigma$-finite measures $\mu_n$ as
$$
\mu_n = \mu + \frac{1}{n} \nu.
$$
Then $\mu_n \xrightarrow{[C]} \mu$, but $\mu_n(X)$ is infinite for all $n$.
\end{example}
Now we consider a standard Bratteli diagram $B$ (with a compact path space $X_B$) and a non-compact locally compact subspace $X \subset X_B$. We give example of a sequence of finite measures $(\mu_n)_{n \in \N}$ on $X$ that converge to a probability measure $\mu$ on cylinder sets i.e. $\mu_n \xrightarrow{[C]} \mu$, but the values $\mu_n(X)$ tend to infinity as $n$ tends to infinity.

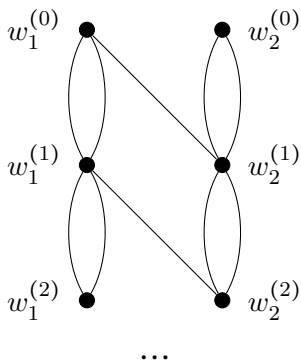
\begin{figure}[h]
\unitlength=1cm
\begin{tikzpicture}[scale=0.9,
    vertex/.style={circle, draw, fill=black, minimum size=0.2cm, inner sep=0pt},
    every edge/.style={thin, black},
    bow/.style={to path={.. controls ++(#1,0.2) and ++(#1,-0.2) .. (\tikztotarget)}}]

\node[vertex, label={[shift={(-0.7,-.5)}]$w_1^{(0)}$}] (V11) at (2,5) {};
\node[vertex, label={[shift={(0.7,-.5)}]$w_2^{(0)}$}] (V12) at (4,5) {};

\node[vertex, label={[shift={(-0.7,-.5)}]$w_1^{(1)}$}] (V21) at (2,3) {};
\node[vertex, label={[shift={(0.7,-.5)}]$w_2^{(1)}$}] (V22) at (4,3) {};

\node[vertex, label={[shift={(-0.7,-.5)}]$w_1^{(2)}$}] (V31) at (2,1) {};
\node[vertex, label={[shift={(0.7,-.5)}]$w_2^{(2)}$}] (V32) at (4,1) {};

\draw (V21) to [out=120, in=240, looseness=0.8] (V11);
\draw (V21) to [out=60, in=300, looseness=0.8] (V11);
\draw (V22) -- (V11);

\draw (V22) to [out=120, in=240, looseness=0.8] (V12);
\draw (V22) to [out=60, in=300, looseness=0.8] (V12);

\draw (V31) to [out=120, in=240, looseness=0.8] (V21);
\draw (V31) to [out=60, in=300, looseness=0.8] (V21);
\draw (V32) -- (V21);

\draw (V32) to [out=120, in=240, looseness=0.8] (V22);
\draw (V32) to [out=60, in=300, looseness=0.8] (V22);

\node[draw=none, fill=none, font=\Large] at (3,0.1) {$\cdots$};


\end{tikzpicture}
\caption{The Bratteli diagram for \Cref{example_measure_convergence_BD_inf}.}
\label{dig_example_9_7}
\end{figure}

\begin{example}\label{example_measure_convergence_BD_inf} Let $B$ be a stationary Bratteli diagram of rank $2$ with the vertex sets $V_n = \{w_1^{(n)}, w_2^{(n)}\}$ for $n \in \N_0$ and the incidence matrix
$$
F = \begin{pmatrix}
2 & 0\\
1 & 2
\end{pmatrix}
$$
(see \Cref{dig_example_9_7}). Let $B_1$ be a vertex subdiagram of $B$ which consists of the paths passing only through the vertices $w_1^{(n)}$, for $n \in \N_0$, and $B_2$ be a vertex subdiagram which consists of the paths passing only through the vertices $w_2^{(n)}$, for $n \in \N_0$.
For every $n \in \mathbb{N}$, let $C_n$ be the union of cylinder sets of length $n$ which pass through the vertex $w_1^{(n-1)}$ and end in the vertex $w_2^{(n)}$. 
In particular, the set $C_1$ consists of one cylinder set which starts at $w_1^{(0)}$ and ends in $w_2^{(1)}$, the set $C_2$ consists of two cylinder sets which pass through the vertices $w_1^{(0)}$, $w_1^{(1)}$ and $w_2^{(2)}$, and the set $C_n$ consists of all paths that pass through the vertices $w_1^{(0)}, \ldots, w_1^{(n-1)}$ and $w_2^{(n)}$.
Let $C_0$ be a cylinder set corresponding to the vertex $w_2^{(0)}$.   Let 
$$
X = \bigsqcup_{n = 0}^{\infty} C_n.
$$
Then $X$ is a locally compact Cantor set and it consists of all the paths in $X_B$ that do not belong to subdiagram $B_1$. Consider subdiagram $B_2$, and let $\mu$ be a unique probability $\mathcal{R}$-invariant measure defined on $X_{B_2} = C_0$. For the cylinder set $[\ov e^{(n)}_2]$ lying in $B_2$ and ending in $w_2^{(n)}$ we have
$$
\mu([\ov e^{(n)}_2]) = \frac{1}{2^{n}}.
$$
We extend the measure $\mu$ by $\mathcal{R}$-invariance to the measure $\wh{\mu}$ on $X$. Then we have
$$
\wh {\mu}|_{C_0} = {\mu}|_{C_0}, \quad
\wh {\mu} (C_n) = 2^{n-1} \mu ([\ov e^{(n)}_2]) = \frac{2^{n-1}}{2^{n}} = \frac{1}{2}
$$
for all $n \in \mathbb{N}$. Thus the extension $\wh {\mu}$ of the measure $\mu$ to the space $X$ by $\mathcal{R}$-invariance gives an infinite $\sigma$-finite measure $\wh {\mu}$ on $X$. 

Now we define a sequence of finite measures $(\mu_n)_{n \in \N}$ on $X$ such that $\mu_n \xrightarrow{[C]} \mu$, i.e. $(\mu_n)_{n \in \N}$ converges to $\mu$ on cylinder sets. Recall that the set $X$ consists of all infinite paths in $X_B$ that eventually pass through the vertices $\{w_2^{(n)}\}_{n = 0}^{\infty}$. By cylinder sets in $X$ we mean the cylinder sets generated by the finite paths which end in the vertices $\{w_2^{(n)}\}_{n = 0}^{\infty}$. These sets generate the induced topology on $X$.
For $n \in \N$, we obtain $\mu_n$ by extending $\mu|_{X_{B_2}}$ only to the finite number of sets $\bigsqcup_{i = 1}^n C_i$ and using coefficients to make the obtained measures converge to $\mu$ on cylinder sets.
Set
$$
{\mu}_n|_{C_0} = {\mu}|_{C_0} = {\wh \mu}|_{C_0}, \quad
\mu_n|_{C_i} = \frac{1}{n + 1 - i} \, {\wh \mu}|_{C_i} \;
\mbox{ for } i = 1, 2 , \ldots n$$
and
$$
 \mu_n|_{X \setminus \left(\bigsqcup_{i = 0}^n C_i\right)} = 0.
$$
Thus we have 
$$
\mu_n(X) = 1 + \frac{1}{n} \wh \mu(C_{1}) + \frac{1}{n-1} \wh \mu(C_{2}) + \ldots + \frac{1}{2} \wh \mu(C_{n-1}) + \wh \mu(C_{n}) = 1 + \frac{1}{2}\sum_{i = 1}^{n} \frac{1}{i}
$$
and $\mu_n(X)$ grows to infinity as $n$ tends to infinity. On the other hand, for every cylinder set $U \subset X \setminus C_0$, there exists $m \in \mathbb{N}$ such that
$$
U \subset \bigsqcup_{i = 1}^m C_i.
$$
We have
$$
\mu_n(U) \leq \mu_n\left(\bigsqcup_{i = 1}^m C_i\right) = \frac{1}{2} \sum_{i = 1}^m  \frac{1}{n + 1 - i} \rightarrow 0
$$
as $n$ tends to infinity.
For every cylinder set $U \subset C_0$ we have $\mu_n(U) = \mu(U)$. Thus, the measures $\mu_n$ converge to the measure $\mu$ on cylinder sets. 
\end{example}

\begin{remark} We can modify the example above to obtain finite measures $\mu_n$ whose values $\mu_n(X)$ tend to any number greater than $1$. Indeed, we can use any convergent series instead of $\sum_{n = 1}^{\infty} \frac{1}{n}$. For any $s > 0$, we can pick a nonnegative series $\sum_{n = 1}^\infty a_n$ with $\sum_{n = 1}^\infty a_n= 2s$ and define the measures $\mu_n$ by setting
$$
\mu_n|_{C_i} = a_n \, {\wh \mu}|_{C_i}.
$$
Thus, we would obtain finite measures $\mu_n$ which converge to $\mu$ on cylinder sets and such that $\lim_{n \rightarrow \infty} \mu_n (X) = 1 + s$.
\end{remark}
\medskip

\textbf{Acknowledgements.} 
The authors are pleased to thank our colleagues and collaborators, especially, R. Curto, J. Kwiatkowski, P. Muhly, W. Polyzou for valuable and stimulating discussions. S.B. and O.K. are also grateful to the Nicolas Copernicus University in Torun for its hospitality and support. S.B. acknowledges the hospitality of AGH University during his visit to Krakow. O.K. acknowledges the hospitality of the University of Iowa during her visits to Iowa City.

This work was partially supported by the Simons Foundation grant (award no. SFI-MPS-T-Institutes-00010825) and from State Treasury funds as part of a task commissioned by the Minister of Science and Higher Education under the project “Organization of the Simons Semesters at the Banach Center - New Energies in 2026-2028” (agreement no. MNiSW/2025/DAP/491).
S.B. is supported by ``AGH World 3M'' program of the AGH University of Krakow. 

O.K. is partially supported by a subsidy from the Polish Ministry of Science and Higher Education for the AGH University of Krakow. 
S.S. would like to thank the support of the Australian Research Council via the grant DP240100472 and Israel Science Foundation via grant No. 1180/22. T.R. is supported by Coordenação de Aperfeiçoamento de Pessoal de Nível Superior – Brasil (CAPES) – Finance Code 001.

\bibliographystyle{alpha}
\bibliography{ReferencesGBD2}

\end{document}